\documentclass[reqno]{amsart}

\usepackage[colorlinks=true, pdfstartview=FitV, linkcolor=blue, citecolor=blue, urlcolor=blue]{hyperref}

\usepackage{mathtools,amsmath, amsfonts, amssymb, latexsym, amsthm,
  amscd, cancel, enumerate, rotating, comment, appendix, ulem,
  makecell, paralist,times,graphicx,pgf,tikz,tikz-cd, times,algorithm,
  algpseudocode,dirtytalk}
\usepackage{siunitx}
\usepackage[margin=1.4in,letterpaper]{geometry} 
\usepackage[capitalise, noabbrev]{cleveref}
\usepackage[all]{xy}
\usepackage[utf8]{inputenc}
\usepackage[english]{babel}
\usepackage[mathscr]{eucal}
\usepackage{xxcolor}
\usepackage{xcolor}
\usepackage{colortbl}
\usepackage{caption}
\usepackage[labelformat=simple]{subcaption}
\usepackage[T1]{fontenc}      
\usepackage{ytableau}
\usetikzlibrary{shapes}
\usetikzlibrary{snakes}
\usetikzlibrary{arrows}
\usepackage{tikz-cd}
\usetikzlibrary{matrix, decorations.pathmorphing}
\usetikzlibrary{decorations.markings}
\usetikzlibrary{calc}
\usetikzlibrary{shapes}

\algnewcommand\algorithmicforeach{\textbf{for each}}
\algdef{S}[FOR]{ForEach}[1]{\algorithmicforeach\ #1\ \algorithmicdo}

\newif\ifdviwin

\numberwithin{equation}{section}

\theoremstyle{plain}
\newtheorem{theorem}{Theorem}[section]
\newtheorem{conjecture}[theorem]{Conjecture}

\newtheorem*{Main Theorem}{Main Theorem}

\newtheorem{proposition}[theorem]{Proposition}
\newtheorem{lemma}[theorem]{Lemma}
\newtheorem{corollary}[theorem]{Corollary}

{\Alph{theorem}}
{\Alph{theorem}}

\theoremstyle{definition}
\newtheorem{notation}[theorem]{Notation}

\newtheorem{construction}[theorem]{Construction}

\newtheorem{remark}[theorem]{Remark}
\newtheorem{definition}[theorem]{Definition}
\newtheorem{example}[theorem]{Example}

\newtheorem{question}[theorem]{Question}

\theoremstyle{remark}

\newtheorem*{Claim1}{Claim 1}
\newtheorem*{Claim2}{Claim 2}
\newtheorem*{Claim3}{Claim 3}
\newtheorem*{Claim4}{Claim 4}
\newtheorem*{Claim5}{Claim 5}

\newcommand{\facets}{\mathrm{Facets}}

\newcommand{\pd}{\mathrm{pd}}
\newcommand{\lcm}{\mathrm{lcm}}
\newcommand{\LCM}{\mathrm{LCM}}

\newcommand{\Scarf}{\mathbb{S}}

\newcommand{\supp}{\mathrm{Supp}}
\renewcommand{\exp}{\mathrm{exp}}

\newcommand{\ba}{\mathbf{a}}
\newcommand{\bb}{\mathbf{b}}
\newcommand{\bc}{\mathbf{c}}

\newcommand{\be}{\mathbf{e}}
\newcommand{\bj}{\mathbf{1}}
\newcommand{\bu}{\mathbf{u}}
\newcommand{\bv}{\mathbf{v}}
\newcommand{\bw}{\mathbf{w}}
\newcommand{\bn}{\mathbf{n}}

\newcommand{\m}{\mathbf{m}}

\newcommand{\sfk}{\mathsf k}
\newcommand{\sq}{\succcurlyeq}

\newcommand{\E}{\mathcal{E}}

\newcommand{\N}{\mathcal{N}}
\newcommand{\U}{\mathcal{U}}
\newcommand{\MS}{\mathcal{S}}

\newcommand{\LL}{\mathbb{L}}
\newcommand{\NN}{\mathbb{N}}
\newcommand{\RR}{\mathbb{R}}
\newcommand{\ZZ}{\mathbb{Z}}
\newcommand{\TT}{\mathbb{T}}
\renewcommand{\SS}{\mathbb{S}}
\newcommand{\UU}{\mathbb{U}}
\newcommand{\HH}{\mathcal{H}}
\newcommand{\bbeta}{\boldsymbol{\beta}}

\newcommand{\Lrq}{\LL^r_q}
\newcommand{\UUrq}{\UU^r_q}

\newcommand{\Erq}{{\E_q}^r}
\newcommand{\Urq}{\U^r_q}
\newcommand{\Srq}{\SS^r_q}  
\newcommand{\Sq}{\SS^3_q}  
\newcommand{\Trq}{\TT^r_q}  
\newcommand{\Nrq}{\N^r_q}

\newcommand{\Hrq}{\mathcal{H}^r_q}

\newcommand{\Uar}{\U_{\ba}^r}

\newcommand{\e}{\epsilon}
\newcommand{\pme}{{\pmb{\e}}}
\newcommand{\pmea}{\pme^{\ba}}
\newcommand{\pmeb}{\pme^{\bb}}

\newcommand{\lex}{\mathrm{lex}}

\newcommand{\qand}{\quad \mbox{and} \quad }
\newcommand{\qor}{\quad \mbox{or} \quad }
\newcommand{\qif}{\quad \mbox{if} \quad }

\newcommand{\qare}{\quad \mbox{are} \quad }
\newcommand{\qfor}{\quad \mbox{for} \quad }
\newcommand{\qwhere}{\quad \mbox{where} \quad }
\newcommand{\qwith}{\quad \mbox{with} \quad }

\newcommand{\qforsome}{\quad \mbox{for some} \quad }
\newcommand{\qforall}{\quad \mbox{for all} \quad }

\newcommand{\st}{\colon}



\title{Realizing resolutions of powers of extremal ideals}

\title{Realizing resolutions of powers of extremal ideals}

\author[T. Chau]{Trung Chau}
\address[T. Chau]
{Chennai Mathematical Institute, H1 SIPCOT IT Park, Siruseri, Kelambakkam 603103, India.}
\email{chauchitrung1996@gmail.com}

\author[A. Duval]{Art M. Duval}
\address[A. Duval]
{Department of Mathematical Sciences, University of Texas at El Paso, El Paso, TX 79968-0514, U.S.A.}
\email{aduval@utep.edu}

\author[S. Faridi]{Sara Faridi}
\address[S. Faridi]
{Department of Mathematics \& Statistics, 
Dalhousie University, 
6297 Castine Way, 
PO BOX 15000, 
Halifax, NS, 
Canada B3H 4R2 
}
\email{faridi@dal.ca}

\author[T. Holleben]{Thiago Holleben}
\address[T. Holleben]
{Department of Mathematics \& Statistics, 
Dalhousie University, 
6297 Castine Way, 
PO BOX 15000, 
Halifax, NS, 
Canada B3H 4R2}
\email{hollebenthiago@dal.ca}

\author[S.~Morey]{Susan Morey}
\address[S. Morey]{Department of Mathematics\\
Texas State University\\
601 University Dr.\\
San Marcos, TX 78666\\U.S.A.}
\email{morey@txstate.edu}

\author[L.~M.~\c{S}ega]{Liana M.~\c{S}ega}
\address[L.~M.~\c{S}ega]{Division of Computing, Analytics and Mathematics, 
University of Missouri-Kansas City, Kansas City, MO 64110 U.S.A.}
\email{segal@umkc.edu}

\date{\today}

\definecolor{DarkGreen}{RGB}{50,203,0}

\begin{document}

\maketitle


\begin{abstract}
  Extremal ideals are a class of square-free monomial ideals which dominate and determine many  algebraic invariants of powers of all square-free monomial ideals. 
  For example, the $r^{th}$ power $\Erq$ of the extremal ideal on $q$ generators has the maximum betti numbers among the $r^{th}$ power of any square-free monomial ideal with $q$ generators. 
  In this paper we study the combinatorial and geometric structure of the (minimal) free resolutions of powers of square-free monomial ideals via the resolutions of powers of extremal ideals. 
  Although the end results are algebraic, this problem has a natural interpretation in terms of polytopes and discrete geometry. 
  Our guiding conjecture is that all powers $\Erq$ of extremal ideals have resolutions  supported on their Scarf simplicial complexes, and thus their resolutions are as small as possible. 
  This conjecture is known to hold for $r \leq 2$ or $q \leq 4$. 
  In this paper we prove the conjecture holds for $r=3$ and any $q\geq 1$ by giving a complete description of the Scarf complex of ${\E_q}^3$. 
  This effectively gives us a sharp bound on the betti numbers and projective dimension of the third power of any square-free monomial ideal. 
  For large $i$ and $q$, our bounds on the $i^{th}$ betti numbers are an exponential improvement over previously known bounds. We also describe a large number of faces of the Scarf complex of $\Erq$ for any $r,q \geq 1$. 
\end{abstract}



\section{Introduction}

In Commutative Algebra, the study of relations between homogeneous polynomials, viewed as generators of an ideal, is a difficult problem that is connected to many other areas of mathematics including Algebraic Geometry,  Combinatorics, and Computational Algebra. The main tool for this purpose, going back to Hilbert's work in the 19th century,  is the construction of a {\it free resolution}: an exact sequence of free modules built recursively from the relations, called {\it syzygies}, between the polynomials.  

The problem becomes notoriously more difficult when studying powers of the ideals, or equivalently, products of the polynomials. Due to its difficulty, the best types of results one can expect for powers of arbitrary  ideals, or even for special classes such as  monomial ideals, are upper bounds for the size of a minimal resolution, which is measured  by the {\it betti numbers} of the ideal.

 Diana Taylor's thesis~\cite{T} in the 1960's, followed  some decades later by Bayer and Sturmfels~\cite{BS}  and then many others (see~\cite{BPS,Mer,L} for examples) established that a free resolution of  a monomial ideal  $I$ with $q$ generators  can be realized from the  chain complex of a cell complex on $q$ vertices. In this case, we say that the cell complex \textit{supports} a resolution of $I$. A lower bound for a minimal free resolution of $I$ using this point of view is the {\it Scarf}  complex  $\Scarf(I)$~\cite{BPS}, which is typically too small to support a resolution, but every free resolution must contain it. An upper bound is the {\it Taylor resolution}~\cite{T} coming from a full simplex on $q$ vertices, which we will denote  by $\TT(I)$. The Taylor resolution provides sharp binomial bounds for betti numbers: the $i^{th}$ betti number $\beta_i(I)$ is bounded by the number of $i$-faces of the simplex so that  $\beta_i(I) \leq \binom{q}{i+1}$, and there are ideals $I$ for which equality holds.
 
A closer inspection, however, reveals that when $I$ has more than one generator, some faces of the simplex $\TT(I^r)$ never appear in a minimal free resolution of $I^r$ when $r>1$, resulting in an increasing gap between the binomial bounds and the actual betti numbers for a given $q$ as $r$ grows. 
{\it Extremal ideals} were born from this observation by Cooper et al.\ in~\cite{L2,Lr}. The authors in those papers showed that if $\E_q$ is the extremal ideal on $q$ generators, and  $I$ is any square-free monomial ideal with $q$ generators, there is a ring homomorphism $\psi_I$ which maps $\Erq$ onto $I^r$ for all $r$, and using $\psi_I$ they showed:
\begin{theorem}[{\bf \cite[Theorem 7.9]{Lr}}]\label{t:e-bound}
If $I$ is generated by $q$ square-free monomials, then for all $i\geq 0$ and $r \geq 1$, 
$$\beta_i(I^r) \leq  \beta_i(\Erq).$$ 
\end{theorem}

It turns out that  extremal ideals also bound many other invariants for powers of square-free monomial ideals. Indeed, in \cite{Halifax} it is shown that the homomorphism $\psi_I: \Erq  \to I^r$ preserves most algebraic properties of powers of square-free monomial ideals, including many difficult invariants. This makes understanding $\Erq$ extremely important as it provides a universal object for understanding powers of square-free monomial ideals. This also means that studying $\Erq$ is deceptively difficult. Although $\Erq$ has a concrete definition and clear symmetries, the combination of all difficulties encountered when studying powers of square-free monomial ideals is, in a sense, embedded in the study of $\Erq$. Thus a clear understanding of $\Erq$, such as one that comes from having a concretely described minimal free resolution, is highly valuable but expected to be very difficult to accomplish. 

In \cite{extremal} it was speculated that $\Erq$ has a Scarf resolution.  We formalize this into a conjecture based on the additional evidence we now have, including results of this paper.
\begin{conjecture}\label{c:conjecture}
   If $r,q\geq 1$, then $\Srq=\Scarf(\Erq)$ supports  a minimal free resolution of $\Erq.$
\end{conjecture}
Establishing \cref{c:conjecture} would result in effective bounds for minimal free resolutions of powers $r\geq 1$  of any square-free monomial ideal $I$ with $q$ generators, in terms of chain complexes $\mathcal{C}$ of Scarf complexes:
\begin{equation}\label{e:MFR-sharp} 
\mathcal{C}(\Scarf(I^r)) \subseteq \mbox{ Minimal Free Resolution of } I^r \subseteq \mathcal{C}(\Scarf(\Erq)).
\end{equation}

This paper is anchored in 
understanding the simplicial complex $\Srq=\Scarf(\Erq)$.
This is a simplicial complex with dimension at least $\binom{q}{r}-1$ that can be represented in a $(q-1)$-dimensional affine subspace of $q$-dimensional space with integer lattice point vertices; see~\cref{e:q34}.

  Several findings of~\cite{extremal} provide evidence for \cref{c:conjecture}. One such result is that for all $r$ and $q$ the first betti numbers of $\Erq$ are Scarf. This was shown by finding a multigraded free resolution of $\Erq$ with only Scarf multidegrees in the first homological degree. Another result is that $\Srq$ supports a minimal free resolution of $\Erq$ when  $q\leq 4$ or $r\leq 2$. These cases were settled by providing a full characterization of $\Srq$, where the minimal nonfaces are edges. 
In the next smallest case, a triangle is among the minimal nonfaces of $\SS_5^3$ (\cref{t:minimalnonfaces3}).  Thus it is no surprise that a new type of facet appears when $r=3$ (\cref{p:weird}). This suggests that the problem becomes more difficult as $q$ or $r$ increases. Indeed, we prove that $\Srq$ inherits the complexities of all $\SS_{q'}^{r'}$ for all smaller $q'$ and $r'$.

\begin{theorem}[{\bf \cref{t:decrease powers}}] 
\label{t:inductive-intro} For positive integers  $r$ and $q$, if $\Srq$ supports a resolution for $\Erq$, then so do $\SS_{q-1}^r$ and $\SS_q^{r-1}$ for ${\E_{q-1}}^r$ and ${\E_q}^{r-1}$, respectively.
\end{theorem}

In this paper, starting with the first unknown case,
we show that the minimal free resolution of ${\E_q}^3$ comes from its Scarf complex. 
In order to determine this complex,  in~\cref{s:halfspaces,s:Scarf-faces} we use discrete geometry to develop tools 
that, for arbitrary $r$, greatly simplify determining whether or not a given set of vertices forms a face in $\SS_q^r$.  
In fact, the faces can be described as intersections of pairs of parallel half-spaces; see~\cref{t:shift}.  
This description of $\Srq$ has curious connections to  well-known structures in combinatorics such as partitions of integers or Postnikov's permutohedra~\cite{Postnikov} (\cref{r:permutohedra,r:partitions}), and is combinatorially appealing though computationally challenging.  For the specific case of $r=3$, we go on to give a complete description of every facet and minimal nonface of $\SS_q^3$ for every $q$ in \cref{t:minimalnonfaces3,thm:complete-list-of-faces}.

The final tool we need is discrete Morse theory, 
which we will use in \cref{s:morse} to
systematically eliminate the extra faces from the Taylor complex all the way down to $\Sq$, proving that the latter supports a  free resolution.  When the ideal is ${\E_q}^3$, this cell complex is the Scarf complex, so the resolution is necessarily minimal.  In other words, we show the following:

\begin{theorem}[{\bf \cref{t:r=3}}]\label{t:scarfintro}
    For any $q \geq 1$, the Scarf complex of ${\E_q}^3$, namely $\Sq$, supports a minimal free resolution of ${\E_q}^3$.
\end{theorem}

\cref{t:scarfintro} is a positive answer to the first open case of \cref{c:conjecture}. Of particular interest is enumerating the number of faces of $\Srq$, which in turn will provide a sharp upper bound for betti numbers of powers of square-free monomial ideals via \eqref{e:MFR-sharp}. 
When $r<3$, the simplicial complex $\Srq$ is well understood~\cite{L2,extremal}, but $r=3$ is a threshold where new types of faces appear, and seem to persist based on computational evidence; see~\cref{p:weird}. In~\cref{s:counting} we use \cref{t:scarfintro} and straightforward counting arguments  applied to the faces in the Scarf complex of ${\E_q}^3$ to show the following.

\begin{theorem}[{\bf \cref{t:sharpbound3}}]\label{t:sharpboundintro}
    Let $q \geq 5$ be an integer. For every $i \leq \binom{q}{3} - 1$, a polynomial $p_i$ of degree $3i + 2$ is given in \cref{t:sharpbound3} such that for every square-free monomial ideal $I$ generated by $q$ elements
    $$
        \beta_{i}(I^3) \leq \binom{\binom{q}{3}}{i+1} + p_i(q)   \qand \pd(I^3) \leq \binom{q}{3} - 1 $$
    and these bounds are sharp in that equality holds for every $i$ when $I = \E_q$. \end{theorem}

\begin{remark}[{\bf \cref{t:c-bound} and \cref{t:i-bound}}]
The asymptotic bound
$\binom{\binom{q}{3}}{i+1}$ from \cref{t:sharpboundintro} is a
more than exponential improvement, by a factor of $2^{q^2}$, over the bounds on betti numbers given by the Taylor complex.
\end{remark}

We note that sharp bounds for the cases $r \leq 2$ or $q \leq 4$ were obtained previously by El Khoury et al.~\cite{extremal}, 
and alternative bounds (for every power $r$ and every $q$) to Taylor's bound were provided in~\cite{Lr}. In~\cref{fig:sharpbound3intro}, we compare our bound from \cref{t:scarfintro}, which comes from $\Srq$ and which is sharp, to Taylor's bound and to the bound given in~\cite{Lr}, which comes from a complex denoted $\Lrq$. Due to differences in order of magnitudes of known bounds, we graph the bounds using an exponential scale. Below we compare specific values given by the bounds when $q = 6$.
$$  
\begin{array}{l|r|r|r}
    & \mbox{Taylor's bound} & \mbox{The bound in~\cite{Lr}} & \mbox{Our bound (\cref{t:sharpboundintro}}) \\ 
   \hline
        \beta_2(I^3) & \num{27720} & \num{19660} & \num{4710} \\
        \beta_3(I^3) & \num{367290} & \num{230360} & \num{19845}\\
        \beta_4(I^3) & \num{3819816}&  \num{2118790} & \num{58530}  \\
        \beta_{20}(I^3) & \num{1 346766106565880} & \num{67327446062800} & 0 \\
        \beta_{40}(I^3) & \num{41648951840265} & \num{10272278170} & 0
    \end{array}
$$
{
\begin{figure}[ht!]
    \centering
    \includegraphics[scale = 0.4]{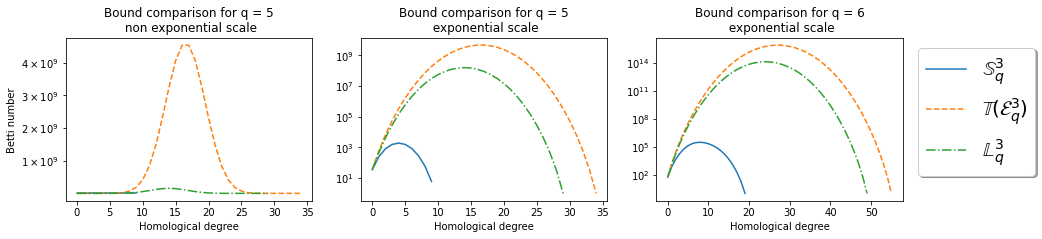}
    \caption{A comparison between known bounds for betti numbers of the third power. The $y$-axis represents the value of the betti number, while the $x$-axis represents homological degree. In the non exponential scale, our bound is not visible as a consequence of the difference in orders of magnitude. Moreover, graphs for higher $q$ look extremely similar because of the asymptotic behavior of the three bounds.}
    \label{fig:sharpbound3intro}
\end{figure}
} 

 The graphs in~\cref{fig:sharpbound3intro} illustrate that our bounds are exponentially better than Taylor's or the bounds in~\cite{Lr}.  
Our results lead us to the following questions, which can be seen as a strengthening of the original question asked in~\cite{extremal}.

\begin{question}\label{q:dominantterm}
    Let $q \gg 0$ and $r$ be a fixed positive integer.
    \begin{enumerate}
        \item\label{i:pd-bound} Does a square-free monomial ideal $I$ generated by $q \gg 0$ elements satisfy
        $
            \pd(I^r) \leq \binom{q}{r}-1?
        $
       Is the dimension of $\Srq$ equal to $\binom{q}{r} - 1$?
        \item\label{i:betti-bound} For every $i$, is there a (possibly zero) polynomial $p_i(q)$ such that for every square-free monomial ideal $I$ generated by $q$ elements: 
        $$
            \beta_i(I^r) \leq \binom{\binom{q}{r}}{i+1} + p_i(q)
        $$
      and for every $0 \leq c < 1$, if $i = c\binom{q}{r}$ then $
         \displaystyle   \lim_{q \rightarrow \infty} \frac{p_i(q)}{\binom{\binom{q}r}{i+1}} = 0?
        $
           \end{enumerate}
\end{question}

We note that even though our original motivation is to study bounds on betti numbers of powers of square-free monomial ideals, results in this paper can be stated from different perspectives. As an example, due to results in~\cite{extremal}, the definition of the Scarf complex of $\Erq$ can be given in a completely combinatorial language, in terms of a series of integer programs. Since the bounds in~\cref{t:sharpboundintro} come from the number of faces of the Scarf complex of ${\E_q}^3$, one can state \cref{t:sharpboundintro} as a computation of the $f$-vector of a simplicial complex defined via sets of lattice points that satisfy different sets of inequalities. For an exact statement of this form see~\cref{s:halfspaces}. We further note that in view of the importance of these Scarf complexes in Commutative Algebra, a systematic study of the combinatorial properties of such complexes could have deep consequences to the study of free resolutions of powers of square-free monomial ideals. 

Finally, we note that it is likely that bounds for arbitrary monomial ideals need to be larger than the ones from~\cref{t:sharpbound3}. For results regarding bounds for arbitrary monomial ideals, see~\cite{Hasan}.

The paper is organized as follows. In~\cref{s:background} we provide  a brief introduction to multigraded resolutions,  including the Scarf complex, and we introduce the $q$-extremal ideal $\E_q$.  We also highlight the doubly inductive nature of the questions in this paper by proving \cref{t:inductive-intro}. This indicates that as the power or number of generators grows, so does the complexity of the problem. In~\cref{s:halfspaces}, the connections of this work to discrete geometry are explored.  Combining discrete geometry with the inductive nature of the question, a recursive means of characterizing faces of the Scarf complex of $\Erq$ is provided in~\cref{s:Scarf-faces}. \cref{s:S3q} focuses on the case $r=3$, with all faces of $\Sq$ being completely described in the section. \cref{s:morse} develops tools using Morse theory, along with a specific order, that lead to a proof in~\cref{s:Eq3-Scarf} that ${\E_q}^3$ is Scarf for all $q$. Using this result,~\cref{s:counting} provides specific sharp bounds on the betti numbers of $I^3$ where $I$ is any square-free monomial ideal.

\section{Extremal ideals}\label{s:background}

 We start with some general background on multigraded free resolutions.
If $S=\sfk[x_1,\ldots,x_n]$ is the polynomial ring over a field $\sfk$, then a free resolution of module $S/I$ is an exact sequence of the form  
$$0\to S^{c_t} \to S^{c_{q-1}} \to \cdots \to S^{c_1} \stackrel{d}{\to} S $$
where the image of $d$ is $I$.  Thus the augmented sequence using the cokernel of $d$, which is $S/I$, is also exact. A modification of this is to replace the final free module $S$ by $I$, making $d$ surjective, and view the remainder of the sequence as a resolution of $I$. This approach simplifies numbering later on and so will be used throughout this paper. The integer $c_i$ is the rank of the $i^{th}$ free module in the resolution, which is the size of a (not necessarily minimal) generating set for the ideal of $i^{th}$ syzygies. While for a given set of monomials $m_1,\ldots,m_q$ (which is the purview of this study) $I=(m_1, \ldots ,m_q)$ has many different resolutions, the {\it minimal} free resolution, that is the resolution with smallest length and smallest values of the ranks $c_i$, is unique up to an isomorphism of complexes, and has the form
$$0\to S^{\beta_p} \to S^{\beta_{p-1}} \to \cdots \to S^{\beta_1} \to S^{\beta_0} \to I \to 0.$$
The positive integers $\beta_0, \ldots, \beta_p$ are called the {\bf betti numbers} of the ideal $I=(m_1,\ldots,m_q)$ of $S$. The length $p$ of the minimal free resolution is called the {\bf projective dimension} of $I$. The uniqueness of the minimal free resolution makes the projective dimension and the betti numbers algebraic invariants of the ideal.

Taylor \cite{T} showed one could \say{homogenize} the simplicial chain complex of a simplex to obtain a free resolution.  The homogenization process works as follows:

\begin{description}
    \item[Step 1] Label each vertex $i$ of the simplex with the monomial $m_i$.
    \item[Step 2] Label each face $\sigma$ of the simplex with the monomial $\m_\sigma$ which is the least common multiple ($\lcm$) of the monomial labels of its vertices. This labeled simplex is called the \textbf{Taylor simplex} of $I$, denoted by $\TT(I)$.
    \item[Step 3] In the simplicial chain complex of the simplex (which is acyclic as a simplex is acyclic), replace each chain group by a free $S$-module of the same rank, and \say{homogenize} each boundary map by giving it coefficients which are quotients of the labels of each face and subface. Since we will not need the details of this construction for our purposes, we refer the interested reader to the references, for example~\cite{Lr}, and \cref{e:running}.
\end{description}

The resulting exact sequence of free $S$-modules is then a free resolution of the ideal $I$.
Bayer and Sturmfels~\cite{BS} generalized Taylor's work, providing criteria for when a cell complex $\mathcal{X}$ {\bf supports} a free resolution of a monomial ideal $I$, in the sense that the cellular chain complex of $\mathcal{X}$ can be homogenized to a free resolution of $I$.

Taylor's resolution is powerful as it works for {\it any} monomial ideal, and as such, provides binomial bounds for betti numbers of monomial ideals: each 
$\beta_i$ would be bounded by the number of $i$-faces for a simplex.
However, though the Taylor resolution is minimal for select ideals, hence providing a {\it sharp bound} on betti numbers, it is too large for most ideals. 

At the other extreme, every cell complex that supports a free resolution of an ideal must contain the Scarf complex of that ideal~\cite[Theorem 59.2]{P}.
To define the Scarf complex, we first introduce some notation. Let $\Delta$ be the labeled Taylor complex corresponding to $I = (m_1, \ldots,m_q)$. Let $\sigma \subseteq \{1, \ldots,q\}$ be a face of $\Delta$ with monomial label $\m_{\sigma} = \lcm (m_i \mid i \in \sigma) \in \LCM(I)$, where $\LCM(I)$ denotes the {\bf lcm lattice} of $I$, i.e., the set of least common multiples of the monomials in the minimal generating set of $I$, ordered by divisibility. Then the {\bf Scarf complex} of $I$ is the labeled subcomplex $\Gamma$ of $\Delta$ given by
$$ \Gamma = \{ \sigma \in \Delta \mid \m_{\sigma} \neq \m_{\tau} \,\, {\mbox{\rm for all }} \tau \in \Delta \,\, {\mbox{\rm with }} \tau \ne \sigma\}.$$
In other words, the Scarf complex is precisely the subcomplex of the Taylor complex consisting of all faces with monomial labels that are unique in that no other face of the Taylor complex shares the label.

\begin{example}\label{e:running} Let $I = (xy, yz, zu)$ be a monomial ideal in the polynomial ring $S=\sfk[x,y,z,u]$.  In this case 
$m_1=xy$, $m_2=yz$ and $m_3=zu$. The Taylor complex is the labeled simplex on the left in \cref{f:Taylor}, 
\begin{figure}[ht!]
\begin{center}
\begin{tabular}{cc}
\begin{tikzpicture}[scale=1]
\tikzstyle{point}=[inner sep=0pt]
\node (a)at (0,1) {};
\node (b) at (-1,0) {};
\node (c) at (1,0) {};
\draw [fill=gray!20](a.center) -- (b.center) -- (c.center);
\draw (a.center) -- (b.center);
\draw (a.center) -- (c.center);
\draw (b.center) -- (c.center);
\node [point,label=below:{{\small $xyzu$}}] at (0,.65) {};
\node [point,label=above:{\small $xy$}] at (0,1) {};
\node [point,label=left:{\small $yz$}] at (-1,0) {};
\node [point,label=right:{\small $zu$}] at (1,0) {};
\node [point,label=right:{{\small $xyzu$}}] at (.5,.6) {};
\node [point,label=left:{{\small $xyz$}}] at (-.5,.6) {};
\node [point,label=below:{{\small $yzu$}}] at (0,0) {};

\draw[black, fill=black] (a) circle(.05);
\draw[black, fill=black] (b) circle(.05);
\draw[black, fill=black] (c) circle(.05);
\end{tikzpicture}
\qquad \qquad & \qquad \qquad 
\begin{tikzpicture}[scale=1]
\tikzstyle{point}=[inner sep=0pt]
\node (a)at (0,1) {};
\node (b) at (-1,0) {};
\node (c) at (1,0) {};
\draw (a.center) -- (b.center);
\draw (b.center) -- (c.center);
\node [point,label=above:{\small $xy$}] at (0,1) {};
\node [point,label=left:{\small $yz$}] at (-1,0) {};
\node [point,label=right:{\small $zu$}] at (1,0) {};
\node [point,label=left:{{\small $xyz$}}] at (-.5,.6) {};
\node [point,label=below:{{\small $yzu$}}] at (0,0) {};
\draw[black, fill=black] (a) circle(.05);
\draw[black, fill=black] (b) circle(.05);
\draw[black, fill=black] (c) circle(.05);
\end{tikzpicture}\\
Taylor complex of $I$  
\qquad \qquad & \qquad \qquad 
Scarf complex of $I$
\end{tabular}
\end{center}
\caption{Complexes supporting a free resolution of $I=(xy,yz,zu)$}\label{f:Taylor}
\end{figure}
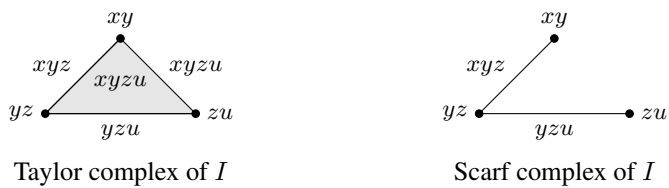
and it supports a free resolution of $I$. Now homogenizing the  simplicial chain complex
of a simplex with these labels results in a {\it multigraded} free resolution below, known as the {\bf Taylor resolution} of $I$, 
$$ 0 \longrightarrow \begin{smallmatrix} S(xyzu)\\ \end{smallmatrix}
\xrightarrow{\scriptsize \begin{bmatrix}u\\x\\-1\end{bmatrix}} 
\begin{smallmatrix}
	S(xyz)\\ \oplus \\ 
	S(yzu) \\ \oplus\\ 
    S(xyzu)\\
\end{smallmatrix}
\xrightarrow{\scriptsize \begin{bmatrix}z&0&zu\\-x&u&0\\0&-y&-xy\end{bmatrix}}
\begin{smallmatrix}
	S(xy)\\ \oplus \\ 
	S(yz) \\ \oplus\\ 
    S(zu)\\ 
\end{smallmatrix} 
\xrightarrow{\scriptsize \begin{bmatrix}xy & yz & zu\end{bmatrix}}
\begin{smallmatrix} I \\ \end{smallmatrix}
\longrightarrow 0 
\,,$$
where the notation $S(x^ay^bz^cu^d)$ refers to the $S$-free module with {\bf multi-degree} $(a,b,c,d)$, or simply $x^ay^bz^cu^d$.  The multigrading is a refinement of the usual (degree) grading. In other words, the Taylor resolution above is a regraded version of $0 \to S \to S^3 \to S^3 \to I \to 0$, where the new grading helps keep track of which face of the simplicial complex is contributing which part of the resolution.

The Scarf complex is obtained by removing the two faces with matching labels, $xyzu$. 
It is shown on the right in~\cref{f:Taylor}. 
The Scarf complex supports a minimal free resolution of $I$, which is the Taylor resolution with the two copies of $S(xyzu)$ removed.
\end{example}

 As every free resolution contains a minimal one, the Taylor resolution contains a minimal multigraded free resolution of the ideal. It follows that  each betti number $\beta_i(I)$ can be written as a sum of {\bf multigraded betti numbers} $\beta_{i,\m}(I)$, where for a monomial $\m$,  the multigraded betti number $\beta_{i,\m}(I)$ denotes the number of copies of $S(\m)$ in the $i^{th}$ homological degree in the multigraded minimal free resolution of $I$. In particular, if $\beta_{i,\m}(I) \neq 0$ then $\m \in \LCM(I)$.

When we take powers of monomial ideals, the number of faces of the Taylor complex that have monomial labels matching that of a subface increases drastically. To observe this phenomenon, consider a monomial $I$ with two minimal generators $m_1,m_2$. Then $I^2=(m_1^2,m_2^2,m_1m_2)$, and then $\TT(I^2)$ will be a triangle (simplex on $3$ vertices).  No matter what the monomials $m_1$ and $m_2$ are, the monomial label of the $2$-face of the triangle matches that of an edge, because $\lcm(m_1^2,m_2^2,m_1m_2)=\lcm(m_1^2,m_2^2)$. In other words, the Taylor resolution is never minimal for $I^2$ when $I$ has two or more generators, and it becomes further from minimal as the powers and the number of generators of $I$ grow.

By~\cref{t:e-bound}, this problem will be worst for powers of the class of ideals $\E_q$ called {\it extremal ideals} (\cite{Lr}).  Following \cite{extremal}, we define for each integer $q$ the {\bf $q$-extremal ideal}, denoted by $\E_q$, as follows.

\begin{definition}\label{d:Eq}
Fix an integer $q \geq 1$. The {\bf $q$-extremal ideal} $\E_q$ is defined to be the ideal of the polynomial ring $\MS_q = \sfk[x_A\colon \emptyset \neq A\subseteq [q]]$ generated by $\{\e_i\}_{i\in [q]}$ where for each $i\in [q]$
\[
\e_i \coloneqq \prod_{i\in A \subseteq [q]} x_A.
\]
\end{definition}

To illustrate this definition, the following example lists the generators of $\E_q$ when $q=4$.

\begin{example} For $q=4$, the ideal $\E_4$ is generated by the  monomials
  $$\begin{array}{ll} \e_1&=x_1 x_{12} x_{13} x_{14} x_{123} x_{124}            
    x_{134}  x_{1234}  ; \\ \e_2&=x_{2} x_{12} x_{23}                           
    x_{24} x_{123} x_{124} x_{234} x_{1234}                                     
    ;\\ \e_3&=x_{3} x_{13} x_{23} x_{34} x_{123} x_{134} x_{234}                
    x_{1234} ;\\ \e_4&=x_{4} x_{14} x_{24} x_{34}                               
    x_{124} x_{134} x_{234} x_{1234} \\                                         
  \end{array}$$
in the polynomial ring $\MS_4 =\sfk[x_1, x_2, x_3, x_4, x_{12}, x_{13}, x_{14}, x_{23}\
, x_{24}, x_{34}, x_{123}, x_{124}, x_{134}, x_{234}, x_{1234}]$.
\end{example}

Writing $\pmea={\e_1}^{a_1}\cdots {\e_q}^{a_q}$, the generators of $\Erq = (\E_q)^r$ are $\{ \pmea \mid a_1 + \cdots + a_q = r, a_i \geq 0\}$ where $\ba = (a_1, \ldots, a_q).$ Note that in the Taylor complex associated to $\Erq$ the monomial $\pmea$ is the label associated to the vertex $\ba.$

It was conjectured \cite{extremal} that $\Erq$ has a  Scarf resolution for any $r$ and $q$. Our first aim is to demonstrate  that as either $r$ or $q$ grows, the difficulty level of this problem can only increase: \cref{t:decrease powers} implies  that if $\Erq$ is Scarf for some $r$ and $q$, then so is ${\E_t}^s$ for all $t\leq q$ and all $s \leq r$. 

For a monomial ideal $I$ and a monomial $\m$, let $I^{\leq \m}$ denote the ideal generated by the monomials in the minimal generating set of $I$ which divide $\m$.  The following is an analog of~\cite[Lemma 2.14]{CHM24} for Scarf complexes, and generalizes a result of Faridi-H\`{a}-Hibi-Morey~\cite{FHHM24}. 

\begin{lemma}\label{lem:special-HHZ}
    If a monomial ideal $I$ is Scarf, then so is $I^{\leq \m}$ for any monomial $\m$.
\end{lemma}
\begin{proof}
    In order to show that $I^{\leq \m}$ is Scarf, we construct a Scarf resolution of $I^{\leq \m}$. Let $\mathbb{F}$ be the Scarf resolution of $I$. By~\cite[Lemma~4.4]{HHZ04}, the complex $\mathbb{F}^{\leq \m}$ obtained by only keeping multidegrees of $\mathbb{F}$ which divide $\m$ is a resolution of $I^{\leq \m}$.
      It follows from the definition of Scarf multidegrees that  a Scarf multidegree  of $I$ that divides $\m$ is a Scarf multidegree of $I^{\leq \m}$. Since a Scarf multidegree of $I^{\leq \m}$ is a Scarf multidegree of $I$ that divides $\m$, the result follows.
\end{proof}

\begin{theorem}\label{t:decrease powers} For positive integers  $r$ and $q$, if $\Erq \subset \MS_q$ is Scarf, then so are $\E_{q-1}^r$ and ${\E_q}^{r-1}$. In other words, if $\Srq$ supports a resolution, then so do $\SS_{q-1}^r$ and $\SS_q^{r-1}$.
\end{theorem}

\begin{proof}  For positive integers $q$ and $r$, let
    $$\m_1=\prod_{\tiny \begin{array}{l}A \subseteq [q]\\ A \neq \emptyset, \{q\} \end{array}} {x_A}^r \qand 
      \m_2= \epsilon_1 \prod_{\tiny \begin{array}{l}A \subseteq [q]\\ A \neq \emptyset, [q] \end{array}} {x_A}^{r-1}.$$
      Then, as ideals of $\MS_q$,
      \begin{enumerate}
          \item $(\Erq)^{\leq \m_1} = (\epsilon_1,\dots, \epsilon_{q-1})^r={\E_{q-1}}^r\MS_q$, because a generator of $\Erq$ that divides $\m_1$ cannot be divisible by $\e_q$. Since $\e_q$ is the only generator of $\E_q$ divisible by $x_{\{q\}}$, the converse holds, meaning every generator of $\Erq$ that is not divisible by $\e_q$ divides $\m_1$.
          
          \item $(\Erq)^{\leq \m_2} = \epsilon_1 {\E_q}^{r-1}$ because the generators of $\epsilon_1 {\E_q}^{r-1}$ divide $\m_2$; conversely, if $\mu=\e_1^{a_1}\cdots \e_q^{a_q}$ is a generator of $\Erq$, then set $A=\{2,\ldots,q\}$ and note that $x_A \mid \e_j$ for $j\geq 2$. Then if $\mu \mid \m_2$, examining the power of $x_A$ yields $\sum_{j=2}^q a_j \leq r-1$ so $a_1 \geq 1$ and $\mu \in \e_1 {\E_q}^{r-1}$ as desired.
      \end{enumerate}

By \cref{lem:special-HHZ},
    $ \epsilon_1 {\E_q}^{r-1}$ and ${\E_{q-1}}^r\MS_q$ are both Scarf ideals of the ring $\MS_q$. 
   Since removing a common factor in every generator does not affect the Scarf property, we conclude ${\E_q}^{r-1}$ is a Scarf ideal of $\MS_q$.   

   To show that ${\E_{q-1}}^r$ is a Scarf ideals of $\MS_{q-1}$, note that the lcm lattices of  ${\E_{q-1}}^r$ as an ideal of $\MS_q$ and as an ideal of $\MS_{q-1}$ are identical, as they only depend on the minimal generating set of the ideal.  Hence the multigraded Taylor resolution and Scarf complexes are the same too. It follows that if ${\E_{q-1}}^r$ has a Scarf resolution over $\MS_q$, then it also has a Scarf
    resolution over $\MS_{q-1}$.
   \end{proof}   

 It is known that when $q\leq 4$ or $r\leq 2$, the Scarf complex $\Srq$ supports a minimal free resolution of $\Erq$ by \cite{extremal}.  \cref{t:r=3} in this paper proves the case $r=3$. These results rely on a full characterization of the Scarf complex $\Srq$ in the aforementioned cases. \cref{t:decrease powers} suggests that finding such complete descriptions of $\Srq$ will become increasingly difficult as $q$ or $r$ increases. 
For an illustration of how the complexity can increase, note that when $q\leq 4$ or $r\leq 2$, the minimal nonfaces of $\Srq$ are edges \cite{extremal}, and when $r=3$, minimal nonfaces include a triangle, as will be seen in \cref{t:minimalnonfaces3}, and the list of facets when $r=3$ includes the appearance of a new type of facet (\cref{p:weird}). Based on \cref{t:decrease powers}, it is expected that as $r$ grows, the number of classes of facets will also increase.

\section{Realizing the Scarf complex $\Erq$ through polytopes}\label{s:halfspaces}

In this section, we expand on a result from \cite{extremal} to provide a description of the Scarf complex $\Erq$ in terms of convex polytopes and to develop new tools to identify the faces of $\Srq.$ 
We will directly translate \cite[Proposition 4.2]{extremal} to a recursive condition in the language of polytopes and integer programming. Before doing so, we collect some notation that will be used throughout the paper.
\begin{notation}\label{n:setup}
For $q$ and $r$ positive integers, $\ba=(a_1,\ldots,a_q) \in \NN^q$, and $A \subseteq [q],$  we use the following notation:
 \begin{itemize} 
 \item $|\ba|=a_1+ a_2 + \cdots + a_q;$ 
 \item $\bj = (1,1,\dots,1)$ is the vector in $\RR^q$ with entries all equal to $1$;
 \item ${\bf 0} = (0,0,\ldots,0)$ is the vector in $\RR^q$ with entries all equal to $0$;
 \item $\be_i$ is the $i^{th}$ standard unit vector in $\RR^q$;
 \item $\be_A = \sum_{i \in A} \be_i$ is the characteristic vector of $A$;
 \item $\overline{A}$ is the complement of $A$ in $[q]$;
 \item $\Nrq = \{\ba \in \NN^q \st |\ba| =r\}$;  
 \item $\Trq = \TT(\Erq)$;
 \item $\Hrq$ is the hyperplane of ${\mathbb R^q}$ with normal vector $\bj$ cut out by the equation 
\[
{\bf x} \cdot \bj = x_1+ \cdots + x_q = r.
\]
 \end{itemize}
\end{notation}

Using this notation, we have
$\ba \cdot \be_A = \sum_{i\in A} a_i$ and $\ba \cdot \bj  = |\ba|.$
Note that $\Nrq$ is the set of lattice points in the first octant of $\Hrq$.  It is also precisely the set of vectors $\ba$ such that $\e^\ba \in \Erq$, which is a minimal set of generators for $\Erq$ by \cite[Proposition~7.3]{Lr}. For this reason we can denote the vertices of $\Trq$ by elements of $\Nrq$.
Restating \cite[Proposition~4.2]{extremal} using this notation, we have: 

\begin{theorem}[{\bf The faces of $\Srq$ (see \cite[Proposition~4.2]{extremal})}]\label{t:shift}
    Let $\sigma = \{{\ba_1}, \dots, {\ba_d}\}$ be a $d$-dimensional face of $\Trq$.  Then $\sigma \in \Srq$ if and only if 
    \begin{enumerate}
        \item\label{i:subsets} $\sigma'\in \Srq$ for all proper subsets $\sigma'$ of $\sigma$; and 
        \item\label{i:first-inequality} ${\ba_1}, \dots, {\ba_d}$ are the only solutions $\bw \in \Nrq$ to the system of inequalities
        \begin{equation}
        \label{e:A}
            \bw \cdot \be_A \leq \max_i \{\ba_i \cdot \be_A\} 
        \end{equation}
        where $A$ ranges over all  $A\subseteq [q]$. 
    \end{enumerate}
\end{theorem}

\begin{remark}\label{r:permutohedra}
  Inequalities very similar to the ones in~\cref{t:shift} were used by Postnikov~\cite[after Definition 6.1]{Postnikov} to describe generalized permutohedra.
\end{remark}

The existence of a vector $\bw \ne \ba_i$ that satisfies condition \eqref{i:first-inequality} in \cref{t:shift} will be used frequently throughout the paper to demonstrate that a given face of $\Trq$ is not in the Scarf complex $\Srq$.  This will lead to \cref{d:witness} of a {\it  witness} vector $\bw$. 

 Algebraically, the conditions in \cref{t:shift} are designed to ensure there are no repeated monomial labels on the faces $\sigma$ of the Scarf complex. 
To interpret \cref{t:shift} geometrically, note that each $\ba_i \in \sigma$ and each $A \subset [q]$ yields a hyperplane and a half-space of $\Hrq$, which are, respectively,
$$
\{ {\bf x} \in \Hrq \mid {\bf x} \cdot \be_A = \ba_i \cdot \be_A\}
\qand 
\{ {\bf x} \in \Hrq \mid {\bf x}\cdot \be_A \leq \ba_i \cdot \be_A\}.
$$  
Hence for a fixed  set $A$, each point in $\sigma$ yields a  hyperplane of $\Hrq$ with normal vector $\be_A$, and so the hyperplanes associated to $\sigma$ are parallel, resulting in a set of nested half-spaces of $\Hrq.$
For a fixed $A$, the maximum value of ${\bf x} \cdot \be_A$ in $\Nrq$ corresponds to one of these half-spaces, which we call $H_{\be_A}(\sigma)$, that is defined by one or more $\ba_i$. Thus the solution set to the system of inequalities \eqref{e:A} as $A$ varies over subsets of $[q]$ is the subset of $\Hrq$ defined by the intersection of half-spaces, $\cap_{A \subseteq [q]} H_{\be_A}(\sigma)$, which by definition is a convex polytope of $\Hrq.$ 
 This means that the integral solution set of the system of linear inequalities is contained in this polytope. Therefore condition~\eqref{i:first-inequality} of \cref{t:shift} means that the lattice points contained in this polytope are precisely $\{\ba_1,\ldots,\ba_d\}$.
Although these polytopes  
are defined because of an algebraic problem, they are also combinatorially interesting; see \cref{r:hypersimplex}.

The direction vector $\be_A$ when situated at the origin lives in the first octant and so is not parallel to $\Hrq$, hence $\be_A$ does not lie in $\Hrq.$ There is another way to get to the polytope containing the integral solution set staying within $\Hrq$, allowing examples to be worked in the hyperplane, which has one lower dimension. Denote the projection of $\be_A$ onto $\Hrq$ by $\widetilde{\be_A}$.

\begin{example}
Set $q=3$ and $r=2$. In $\RR^3$, the vector $e_3= [0,0,1]$ is not in the plane containing $\mathcal{N}_3^2$, but its projection $\widetilde{\be_3}=[-1/3, -1/3, 2/3]$ shifted to be in the plane containing $\mathcal{N}_3^2$ connects the point $(2/3, 2/3, 2/3)$ to $(1/3, 1/3, 4/3)$. Once we have embedded $\widetilde{\be_3}$ in $\mathcal{N}_3^2$, we select a point of view so that the resulting diagram appears in the plane.

\begin{center}
\begin{tabular}{ccc}
 \begin{tikzpicture}[scale=0.75]
\coordinate (200) at (0, 0);
\coordinate (110) at (1.75, .75);
\coordinate (020) at (3.5,1.5);
\coordinate (101) at (.75, 1.75);
\coordinate (011) at (2.5, 2.5);
\coordinate (002) at (1.5, 3.5);
\coordinate (000) at (1.5, 1.5);
\coordinate (001) at (1.5, 2.5);
\coordinate (003) at (1.5, 4);
\coordinate (300) at (-0.5, -0.5);
\coordinate (030) at (4,1.5);
\coordinate (a) at (1.4, 2);

\coordinate (Z) at (0,-0.865);
\draw[] (Z) circle(0.0001);

\draw[black, fill=black] (200) circle(0.05);
\draw[black, fill=black] (110) circle(0.05);
\draw[black, fill=black] (020) circle(0.05);
\draw[black, fill=black] (101) circle(0.05);
\draw[black, fill=black] (011) circle(0.05);
\draw[black, fill=black] (002) circle(0.05);
\draw[black, fill=black] (000) circle(0.05);
\draw[-] (300) -- (000);
\draw[-] (030) -- (000);
\draw[-] (000) -- (003);
\draw[-,thick] (200) -- (020);
\draw[-,thick] (020) -- (002);
\draw[-,thick] (200) -- (002);
\draw[->, ultra thick] (000) -- (001);

 \node[label = {[xshift=-.6cm, yshift=-0.3cm] {\small $(2,0,0)$}}] at (200) {};
 \node[label = {[xshift=.45cm, yshift=-0.2cm] {\small $(0,2,0)$}}] at (020) {};
 \node[label = {[xshift=-.6cm, yshift=-0.4cm] {\small $(0,0,2)$}}] at (002) {};
 \node[label = {[xshift=.65cm, yshift=-0.45cm] {\small $(1,1,0)$}}] at (110) {};
 \node[label = {[xshift=-.6cm, yshift=-0.45cm] {\small $(1,0,1)$}}] at (101) {};
 \node[label = {[xshift=.6cm, yshift=-0.3cm] {\small $(0,1,1)$}}] at (011) {};
 \node[label = right: {$\be_3$}] at (a) {};
 \end{tikzpicture} 

  & \hspace{.4in} &

 \begin{tikzpicture}[scale=0.75]
\coordinate (200) at (0, 1);
\coordinate (110) at (1, 1);
\coordinate (020) at (2,1);
\coordinate (101) at (0.5, 1.865);
\coordinate (011) at (1.5, 1.865);
\coordinate (002) at (1, 2.73);
\coordinate (000) at (1, 1.577);
\coordinate (a) at (1, 2.154);
\coordinate (b) at (.9, 1.865);

\coordinate (Z) at (0,-0.865);
\draw[] (Z) circle(0.0001);

\draw[black, fill=black] (200) circle(0.05);
\draw[black, fill=black] (110) circle(0.05);
\draw[black, fill=black] (020) circle(0.05);
\draw[black, fill=black] (101) circle(0.05);
\draw[black, fill=black] (011) circle(0.05);
\draw[black, fill=black] (002) circle(0.05);
\draw[-, thick] (200) -- (020);
\draw[-, thick] (020) -- (002);
\draw[-, thick] (200) -- (002);
\draw[->, ultra thick] (000) -- (a);

 \node[label = {[xshift=-.4cm, yshift=-0.7cm] {\small $(2,0,0)$}}] at (200) {};
 \node[label = {[xshift=.4cm, yshift=-0.7cm] {\small $(0,2,0)$}}] at (020) {};
 \node[label = {[xshift=0cm, yshift=-0.15cm] {\small $(0,0,2)$}}] at (002) {};
 \node[label = {[xshift=0cm, yshift=-0.75cm] {\small $(1,1,0)$}}] at (110) {};
 \node[label = {[xshift=-.6cm, yshift=-.35cm]{\small $(1,0,1)$}}] at (101) {};
 \node[label = {[xshift=.6cm, yshift=-.35cm]{\small $(0,1,1)$}}] at (011) {};
 \node[label ={[xshift=.3cm, yshift=-0.65cm] $\widetilde{\be_3}$}] at (b) {};

 \end{tikzpicture} 
 \end{tabular}
\end{center}

\end{example}

\begin{proposition}\label{p:min-max}
Let $\sigma = \{\ba_1, \ldots, \ba_d\}$ be a $d$-dimensional face of $\Trq$, and let   $\bw \in \Nrq$. 
\begin{enumerate}
\item\label{i:tilde} if $A \subset [q]$, then $\bw$ satisfies \eqref{e:A} if and only if $\bw \cdot \widetilde{\be_A} \leq \max_i \{\ba_i \cdot \widetilde{\be_A}\}$;
\item\label{i:opposite} if $A \subset [q]$, then $\widetilde{\be_A} = -\widetilde{\be_{\overline{A}}}$; 
\item\label{i:tilde-min} if $A \subset [q]$, then $\bw \cdot \widetilde{\be_A} \leq \max_i \{\ba_i \cdot \widetilde{\be_A}\}$ if and only if $\min_i \{\ba_i \cdot {\be_{\overline{A}}}\}\leq  \bw\cdot \be_{\overline{A}}$;
\item\label{i:min-max} $\bw$ satisfies \eqref{e:A} for all $A\subseteq [q]$ if and only if
 $$ \min_i\{\ba_i \cdot \be_A\} \leq \bw \cdot \be_A \leq \max_i \{\ba_i \cdot \be_A\} 
            \qforall A \subseteq [q].$$    
\end{enumerate}  
\end{proposition}

\begin{proof}
Since $\bj$ is the normal vector to $\Hrq$, it is straightforward to verify that 
\begin{equation}\label{projection} 
\widetilde{\be_A} = \be_A - {\frac{|A|}{q}}\bj,
\end{equation}
where $|A|$ is the cardinality of the set $A$. 
Then, for any ${\bf x} \in \Nrq$,
\[
    {\bf x} \cdot \widetilde{\be_A} = {\bf x} \cdot \be_A - \frac{|A|}{q}r,
\]
so, for any ${\bf x, y} \in \Nrq$,
\[
    {\bf x} \cdot \be_A \leq {\bf y} \cdot \be_A\quad
    \text{if and only if}\quad
    {\bf x} \cdot \widetilde{\be_A} \leq {\bf y} \cdot \widetilde{\be_A}.
\]
Since $\bw, \ba_1, \ldots, \ba_d \in \Nrq$, then ~\eqref{i:tilde} readily follows.

Next note that $\be_A + \be_{\overline{A}} = \bj$.  Then
$$\widetilde{\be_A} + \widetilde{\be_{\overline{A}}} = \be_A + \be_{\overline{A}} - \left({\frac{|A|}{q}} + {\frac{|\overline{A}|}{q}}\right)\bj = \bj - \bj = 0$$ 
in $\Nrq$, by Equation~\eqref{projection}, implying~\eqref{i:opposite} .

From~\eqref{i:opposite} we deduce $\bw \cdot \widetilde{\be_{\overline{A}}} = -\bw \cdot \widetilde{\be_{A}}$ and $\max_i \{\ba_i \cdot \widetilde{\be_A}\} = - \min_i \{\ba_i \cdot \widetilde{\be_{\overline{A}}}\}$, implying~\eqref{i:tilde-min}.
Finally, as $A$ runs over all possible subsets of $[q]$, so does $\overline{A}$, so~\eqref{i:min-max} follows from~\eqref{i:tilde} and~\eqref{i:tilde-min}.
\end{proof}

Part~\eqref{i:tilde} of~\cref{p:min-max} means that we can define the polytopes using $\widetilde{\be_A}$ in place of $\be_A.$ Geometrically, $\widetilde{\be_A}$ defines a hyperplane $\{ {\bf x} \in \Hrq \mid {\bf x} \cdot \widetilde{\be_A} = \ba_i \cdot \widetilde{\be_A}\}$, and so $\{ {\bf x} \in \Hrq \mid {\bf x}\cdot \widetilde{\be_A} \leq \ba_i \cdot \widetilde{\be_A}\}$ defines a half-space of $\Hrq$ for each $A$ and each $\ba_i.$ The advantage here is that $\widetilde{\be_A}$ lies in $\Hrq.$

As an immediate consequence of part~\eqref{i:opposite}, for each face $\sigma\in \Trq$, the hyperplanes in $\Hrq$ corresponding to $\widetilde{\be_A}$ and $\widetilde{\be_{\overline{A}}}$ are parallel, even though the hyperplanes in $\RR^q$ corresponding to $\be_A$ and $\be_{\overline{A}}$ are not.
Thus when we draw in $\Hrq$ a region described by part~\eqref{i:min-max}, we form a ``sandwich'' between two parallel hyperplanes.

\begin{example}\label{e:sandwiches}
Let $r=2$ and $q=3.$
The hyperplane corresponding to $A=\{3\}$ and $\ba_i = (0,0,2)$ is the dashed horizontal line at the top of the diagram below. Since $(0,0,2) \cdot \be_A = (0,0,2) \cdot (0,0,1) = 2$ is the maximum value of $\ba_i \cdot \be_A$ as $\ba_i$ runs through the points in ${\mathcal N}_3^2,$ then for $\gamma$ the maximal face of the Taylor complex, which contains all $\ba_i \in {\mathcal N}_3^2$, we have $H_{\be_A}(\gamma)$ is the half-space below the dashed line. Now $\overline{A} = \{1,2\}$ and $H_{\be_{\overline{A}}}(\gamma)$ is the half-space lying above the dashed horizontal line at the bottom of the diagram, which corresponds to $\ba_i = (1,1,0), (2,0,0),$ and $(0,2,0),$ since $\ba_i \cdot \be_{\overline{A}} = \ba_i \cdot (1,1,0) = 2$ for each of these three points. Notice that for each of these three points $\bw \cdot \be_A =0$, which is the minimum value over all points $\bw$ in $\gamma$ of $\bw \cdot \be_A,$ so the half-space lying above the lower dashed horizontal line can also be described using the minimum value, as in \cref{p:min-max}.

 To see that $\sigma = \{(0,0,2),(1,1,0)\}$ is not an edge in $\mathbb{S}_3^2$, consider the convex region for $\sigma$. As above, for $A=\{3\}$ and $\overline{A} = \{1,2\}$ the region is the "sandwich" bounded by the two dashed lines. For $A = \{2\}$ and $\overline{A} = \{1,3\}$ the region is bounded by the lines with mixed dashes and dots, and for $A = \{1\}$ and $\overline{A} = \{2,3\}$ the region is bounded by the dotted lines. The intersection of these regions is a parallelogram containing $(1,0,1)$ and $(0,1,1)$, which are not in $\sigma$. Therefore by \cref{t:shift}, $\sigma$ is not a face of $\SS_3^2.$
    
    Similarly, $\tau = \{(1,1,0),(1,0,1),(0,1,1)\}$ is a face in $\mathbb{S}_3^2$. Note that the hyperplane corresponding to $A=\{3\}$ that maximizes the value of $\ba_i \cdot \be_A$ over $\tau$ now passes through $(1,1,0)$ and $(0,1,1).$ The intersection of the regions for $\tau$ is the central triangle in the figure, whose only lattice points lie in $\tau.$

Continuing in this manner, one can see that the facets of the Scarf complex $\SS_3^2$ are represented by the four shaded triangles in the right hand diagram, recovering a result from \cite{extremal}.

\begin{center}
\begin{tabular}{ccc}
 
 \begin{tikzpicture}[scale=0.75]
\coordinate (200) at (0, 0);
\coordinate (110) at (1, 0);
\coordinate (020) at (2,0);
\coordinate (101) at (0.5, 0.865);
\coordinate (011) at (1.5, 0.865);
\coordinate (002) at (1, 1.73);

\coordinate (TL) at (0,1.73);
\coordinate (TR) at (2,1.73);
\coordinate (LT) at (-0.5,0.865);
\coordinate (LB) at (0.5,-0.865);
\coordinate (RT) at (2.5,0.865);
\coordinate (RB) at (1.5,-0.865);

\draw [fill=gray!20](101) -- (002) -- (011) -- (110);

\draw[dashed, thick, -] (TL) -- (TR);
\draw[dashed, thick, -] (200) -- (020);
\draw[densely dotted, thick, -] (RB) -- (TL);
\draw[densely dotted, thick, -] (020) -- (002);
\draw[dashdotdotted, thick, -] (LB) -- (TR);
\draw[dashdotdotted, thick, -] (200) -- (002);

\draw[black, fill=black] (200) circle(0.05);
\draw[black, fill=black] (110) circle(0.1);
\draw[black, fill=black] (020) circle(0.05);
\draw[black, fill=black] (101) circle(0.05);
\draw[black, fill=black] (011) circle(0.05);
\draw[black, fill=black] (002) circle(0.1);

\node[label = {[xshift=-.6cm, yshift=-0.4cm] {\tiny $(2,0,0)$}}] at (200) {};
 \node[label = {[xshift=.6cm, yshift=-0.4cm] {\tiny $(0,2,0)$}}] at (020) {};
 \node[label = {[xshift=0cm, yshift=-0.15cm] {\tiny $(0,0,2)$}}] at (002) {};
 \node[label = {[xshift=.65cm, yshift=-0.65cm] {\tiny $(1,1,0)$}}] at (110) {};
 \node[label = {[xshift=-.55cm, yshift=-0.4cm] {\tiny $(1,0,1)$}}] at (101) {};
 \node[label = {[xshift=.55cm, yshift=-0.4cm]{\tiny $(0,1,1)$}}] at (011) {};
 
 \end{tikzpicture} 

 &
 
 \begin{tikzpicture}[scale=0.75]
\coordinate (200) at (0, 0);
\coordinate (110) at (1, 0);
\coordinate (020) at (2,0);
\coordinate (101) at (0.5, 0.865);
\coordinate (011) at (1.5, 0.865);
\coordinate (002) at (1, 1.73);

\coordinate (TL) at (0,1.73);
\coordinate (TR) at (2,1.73);
\coordinate (LT) at (-0.5,0.865);
\coordinate (LB) at (0.5,-0.865);
\coordinate (RT) at (2.5,0.865);
\coordinate (RB) at (1.5,-0.865);

\draw [fill=gray!20](101) -- (110) -- (011);
\draw[dashed, thick, -] (LT) -- (RT);
\draw[dashed, thick, -] (200) -- (020);
\draw[densely dotted, thick, -] (RB) -- (TL);
\draw[densely dotted, thick, -] (020) -- (002);
\draw[dashdotdotted, thick, -] (LB) -- (TR);
\draw[dashdotdotted, thick, -] (200) -- (002);

\draw[black, fill=black] (200) circle(0.05);
\draw[black, fill=black] (110) circle(0.1);
\draw[black, fill=black] (020) circle(0.05);
\draw[black, fill=black] (101) circle(0.1);
\draw[black, fill=black] (011) circle(0.1);
\draw[black, fill=black] (002) circle(0.05);

\node[label = {[xshift=-.6cm, yshift=-0.4cm] {\tiny $(2,0,0)$}}] at (200) {};
 \node[label = {[xshift=.6cm, yshift=-0.4cm] {\tiny $(0,2,0)$}}] at (020) {};
 \node[label = {[xshift=0cm, yshift=-0.15cm] {\tiny $(0,0,2)$}}] at (002) {};
\node[label = {[xshift=.65cm, yshift=-0.65cm] {\tiny $(1,1,0)$}}] at (110) {};
 \node[label = {[xshift=-.55cm, yshift=-0.2cm] {\tiny $(1,0,1)$}}] at (101) {};
 \node[label = {[xshift=.55cm, yshift=-0.2cm]{\tiny $(0,1,1)$}}] at (011) {};
 
 \end{tikzpicture} 

&

\begin{tikzpicture}[scale=0.75]
\coordinate (200) at (0, 0);
\coordinate (110) at (1, 0);
\coordinate (020) at (2,0);
\coordinate (101) at (0.5, 0.865);
\coordinate (011) at (1.5, 0.865);
\coordinate (002) at (1, 1.73);

\coordinate (Z) at (0,-0.865);
\draw[] (Z) circle(0.0001);

\draw [fill=gray!20](200) -- (110) -- (101);
\draw [fill=gray!20](101) -- (110) -- (011);
\draw [fill=gray!20](110) -- (011) -- (020);
\draw [fill=gray!20](101) -- (002) -- (011);

\draw[black, fill=black] (200) circle(0.05);
\draw[black, fill=black] (110) circle(0.05);
\draw[black, fill=black] (020) circle(0.05);
\draw[black, fill=black] (101) circle(0.05);
\draw[black, fill=black] (011) circle(0.05);
\draw[black, fill=black] (002) circle(0.05);
\draw[-] (200) -- (020);
\draw[-] (020) -- (002);
\draw[-] (200) -- (002);
\draw[-] (110) -- (011);
\draw[-] (101) -- (011);
\draw[-] (110) -- (101);

\node[label = {[xshift=-.6cm, yshift=-0.4cm] {\tiny $(2,0,0)$}}] at (200) {};
 \node[label = {[xshift=.6cm, yshift=-0.4cm] {\tiny $(0,2,0)$}}] at (020) {};
 \node[label = {[xshift=0cm, yshift=-0.15cm] {\tiny $(0,0,2)$}}] at (002) {};
 \node[label = {[xshift=0cm, yshift=-0.7cm] {\tiny $(1,1,0)$}}] at (110) {};
 \node[label = {[xshift=-.55cm, yshift=-0.4cm] {\tiny $(1,0,1)$}}] at (101) {};
 \node[label = {[xshift=.55cm, yshift=-0.4cm]{\tiny $(0,1,1)$}}] at (011) {};

 \end{tikzpicture} 

\end{tabular}
\end{center}

\end{example}                                                      

Note that for $q \leq 4$, the techniques above suffice to determine the Scarf complex $\Srq$ for all $r$.

\begin{example}[{\bf The case where $q \leq 3$}]\label{e:q23}
When $q=1$, there is only a single point, the vector $(r)=r\be_1$.
When $q=2$ the hyperplane $\HH_2^r$ is a line for all $r$ and when $q=3$ the hyperplane $\HH_3^r$ is two-dimensional and so can be depicted in a plane as in \cref{e:sandwiches}. In these two cases, it is straightforward to explicitly describe the Scarf complex for all $r.$ For $q=2$, all lattice points in ${\mathcal N}_2^r$ are necessarily colinear, so if $\sigma \in \TT_2^r$ contains two points that are not adjacent, then $\sigma$ is not in the Scarf complex. For $q=3$ one can verify using the same technique that the Scarf complex is a collection of triangles as depicted below when $r=4$. The pattern holds for all $r.$

\begin{figure}
\hfill
\begin{center} 
\begin{subfigure}{0.4\textwidth}
\begin{center}
\begin{tikzpicture}[scale=.5]
\coordinate (00) at (0, 0);
\coordinate (0r) at (0,6);
\coordinate (1r-1) at (1,5);
\coordinate (2r-2) at (2,4);
\coordinate (dots) at (3,3);
\coordinate (r-22) at (4,2);
\coordinate (r-11) at (5,1);
\coordinate (r0) at (6,0);

\draw[black, fill=black] (00) circle(0.05);
\draw[black, fill=black] (0r) circle(0.05);
\draw[black, fill=black] (1r-1) circle(0.05);
\draw[black, fill=black] (2r-2) circle(0.05);
\draw[black, fill=black] (r-22) circle(0.05);
\draw[black, fill=black] (r-11) circle(0.05);
\draw[black, fill=black] (r0) circle(0.05);
\draw[-] (00) -- (r0);
\draw[-] (00) -- (0r);
\draw[-,thick] (r0) -- (r-11);
\draw[-,thick] (r-11) -- (r-22);
\draw[-,thick] (0r) -- (1r-1);
\draw[-, thick] (1r-1) -- (2r-2);

 \node[label = right: {\small $(r,0)$}] at (r0) {};
 \node[label = right: {\small $(1,r-1)$}] at (1r-1) {};
 \node[label = right: {\small $(2,r-2)$}] at (2r-2) {};
 \node[label = right: {\small $(0,r)$}] at (0r) {};
 \node[label = right: {\small $(r-1,1)$}] at (r-11) {};
 \node[label = right: {\small $(r-2,2)$}] at (r-22) {};
 \node at (dots) {$\ddots$};
 \end{tikzpicture} 
 \end{center}
\caption{$q=2$ any $r$ as a hyperplane in $\RR^2$}
\label{f:q2}
\end{subfigure}
\hspace{.4in}
\begin{subfigure}{0.4\textwidth}
\begin{center}
 \begin{tikzpicture}[scale=0.75]
\coordinate (1111) at (0, 0);
\coordinate (1112) at (1, 0);
\coordinate (1122) at (2,0);
\coordinate (1222) at (3, 0);
\coordinate (2222) at (4, 0);

\coordinate (1113) at (0.5, 0.865);
\coordinate (1123) at (1.5, 0.865);
\coordinate (1223) at (2.5,0.865);
\coordinate (2223) at (3.5, 0.865);

\coordinate (1133) at (1, 1.73);
\coordinate (1233) at (2, 1.73);
\coordinate (2233) at (3, 1.73);

\coordinate (1333) at (1.5, 2.595);
\coordinate (2333) at (2.5, 2.595);

\coordinate (3333) at (2, 3.46);

\draw [fill=gray!20](1111) -- (1112) -- (1113);
\draw [fill=gray!20](1123) -- (1112) -- (1113);
\draw [fill=gray!20](1123) -- (1133) -- (1113);
\draw [fill=gray!20](1123) -- (1133) -- (1233);
\draw [fill=gray!20](1333) -- (1133) -- (1233);
\draw [fill=gray!20](1333) -- (2333) -- (1233);
\draw [fill=gray!20](1333) -- (2333) -- (3333);

\draw [fill=gray!20](1112) -- (1122) -- (1123);
\draw [fill=gray!20](1223) -- (1122) -- (1123);
\draw [fill=gray!20](1123) -- (1223) -- (1233);
\draw [fill=gray!20](2233) -- (1223) -- (1233);
\draw [fill=gray!20](2233) -- (2333) -- (1233);

\draw [fill=gray!20](1122) -- (1222) -- (1223);
\draw [fill=gray!20](2223) -- (1222) -- (1223);
\draw [fill=gray!20](2223) -- (1223) -- (2233);
\draw [fill=gray!20](2223) -- (1222) -- (2222);

\draw[black, fill=black] (1111) circle(0.05);
\draw[black, fill=black] (1122) circle(0.05);
\draw[black, fill=black] (1112) circle(0.05);
\draw[black, fill=black] (1222) circle(0.05);
\draw[black, fill=black] (2222) circle(0.05);
\draw[-] (1111) -- (2222);

\node[label = {[xshift=-0.6cm, yshift=-0.4cm]{\small $(4,0,0)$}}] at (1111) {};
\node[label = {[xshift=-.34cm, yshift=-0.7cm] {\small $(3,1,0)$}}] at (1112) {};
\node[label = {[xshift=0cm, yshift=-0.7cm] {\small $(2,2,0)$}}] at (1122) {};
\node[label = {[xshift=0.34cm, yshift=-0.7cm] {\small $(1,3,0)$}}] at (1222) {};   
\node[label = {[xshift=0.6cm, yshift=-0.4cm]{\small $(0,4,0)$}}] at (2222) {};      

\draw[black, fill=black] (1113) circle(0.05);
\draw[black, fill=black] (1123) circle(0.05);
\draw[black, fill=black] (1223) circle(0.05);
\draw[black, fill=black] (2223) circle(0.05);
\draw[-] (1113) -- (2223);
\node[label = {[xshift=-0.6cm, yshift=-0.4cm] {\small $(3,0,1)$}}] at (1113) {}; 
\node[label = {[xshift=0.6cm, yshift=-0.4cm] {\small $(0,3,1)$}}] at (2223) {};   

\draw[black, fill=black] (1133) circle(0.05);
\draw[black, fill=black] (1233) circle(0.05);
\draw[black, fill=black] (2233) circle(0.05);
\draw[-] (1133) -- (2233);
\node[label = {[xshift=-0.6cm, yshift=-0.4cm] {\small $(2,0,2)$}}] at (1133) {}; 
\node[label = {[xshift=0.6cm, yshift=-0.4cm] {\small $(0,2,2)$}}] at (2233) {};  

\coordinate (1333) at (1.5, 2.595);
\coordinate (2333) at (2.5, 2.595);
\draw[black, fill=black] (1333) circle(0.05);
\draw[black, fill=black] (2333) circle(0.05);
\draw[-] (1333) -- (2333);
\node[label = {[xshift=-0.6cm, yshift=-0.4cm] {\small $(1,0,3)$}}] at (1333) {};
\node[label = {[xshift=0.6cm, yshift=-0.4cm] {\small $(0,1,3)$}}] at (2333) {};

\draw[black, fill=black] (3333) circle(0.05);
\node[label = {[xshift=0.cm, yshift=-0.2cm]{\small $(0,0,4)$}}] at (3333) {};

\draw[-] (1111) -- (3333);
\draw[-] (1112) -- (2333);
\draw[-] (1122) -- (2233);
\draw[-] (1222) -- (2223);
\draw[-] (1112) -- (1113);
\draw[-] (1122) -- (1133);
\draw[-] (1222) -- (1333);
\draw[-] (2222) -- (3333);

 \end{tikzpicture} 
 \end{center}
\caption{$q=3$ and $r=4$ in the hyperplane ${\mathcal H}_3^4$}
\label{f:q3}
\end{subfigure}

 \caption{\cref{e:q23}}\label{f:two-together}
 \end{center}
 \hfill
\end{figure}

\end{example}

It turns out that $q=4$ (see \cref{t:Urq-Srq} and \cref{e:q34}) is the limit for using just the techniques above to describe all of the faces of $\Srq$. When $q\ge 5$ unexpected behavior occurs and additional faces appear. As $r$ grows, the number of classes of vertices in $\Nrq$ that are distinct in the sense of not being equivalent under permutation grows significantly, increasing the complexity of types of potential faces of $\Srq$. Computations show that the unexpected behavior persists as $q$ and $r$ increase.
Indeed, as $q$ and $r$ grow, the problem becomes increasingly difficult, which can also be seen from~\cref{t:char-face,t:decrease powers}.

\section{The faces of $\Srq$} \label{s:Scarf-faces}

In this section, we describe sets of faces that are guaranteed to appear in the Scarf complex, and show how to use the recursive nature of some such faces to visualize higher dimensional  complexes in some settings. \cref{t:shift} gives a recursive condition for checking if a subset of $\Nrq$ is in the Scarf complex, 
which leads, in turn, to other characterizations of the faces of the Scarf complex. Several of these are summarized in \cref{t:char-face}. Throughout the rest of the paper, we will use the following notation in addition to \cref{n:setup}:

\begin{notation}\label{n:setup2}
For positive integers $p$ and $q$, and $\ba=(a_1,\ldots,a_q), \bb=(b_1,\ldots,b_q) \in \ZZ^q$,  use the following  language.
 \begin{itemize} 
 \item $\ba$ is {\bf square-free} if $a_i \in \{0,1\}$ for all $i \in [q]$;
 \item $\ba \leq \bb$ means $a_i \leq b_i$ for all $i \in [q]$;
 \item $\max(\ba)=\max\{a_1,\ldots,a_q\}$;
 \item $\supp(\ba)=\{i \in [q] \st a_i \neq 0 \}$;
 \item $\ba \cap \bb=(\min(a_1,b_1), \min(a_2,b_2),\ldots, \min(a_q,b_q))$;
 \item $\ba 0^p, 0^p \ba \in \ZZ^{p+q}$ 
 are vectors obtained by appending $p$ $0$'s to the right and left of $\ba$, respectively;
 \item $\pi(\ba)=(a_{\pi(1)},\ldots,a_{\pi(q)})$ for a permutation $\pi \in S_q$. 
  \end{itemize}
\end{notation}

The following theorem makes precise what happens to the faces of the Scarf complex when the vertices are translated or are embedded in a higher dimension. It also shows that permuting the coordinates preserves Scarf faces.

\begin{theorem}[{\bf The faces of $\Srq$}]\label{t:char-face}
    Let $p$, $q$, and $r$ be positive integers, and $\sigma = \{\ba_1, \dots, \ba_d\} \in \Trq$.  The following are equivalent.
    
        \begin{enumerate}
     \item \label{i:sigma} $\sigma \in \Srq$;
      \item \label{i:extremal}
      for every $C \subseteq [d]$  the only solutions $\bw \in \Nrq$ to the system of inequalities
        $$
            \bw \cdot \be_A \leq \max_{i \in C} \{\ba_i \cdot \be_A\} 
            \qforall A \subseteq [q]
        $$
        are $\{\ba_i \mid i \in C\}$;
    \item \label{i:min-max-witness} 
    for every $C \subseteq [d]$  the only solutions $\bw \in \Nrq$ to the system of inequalities
        $$
            \min_{i \in C} \{\ba_i \cdot \be_A\} \leq
            \bw \cdot \be_A \leq \max_{i \in C} \{\ba_i \cdot \be_A\} 
            \qforall A \subseteq [q]
        $$
        are $\{\ba_i \mid i \in C\}$;      
    \item \label{i:0-left} $\sigma 0^p=\{\ba_1 0^p, \dots, \ba_d 0^p\} \in \SS_{q+p}^r$ for every  $p \in \NN$;       
    \item \label{i:0-right} $0^p \sigma=\{0^p\ba_1, \dots, 0^p\ba_d\} \in \SS_{q+p}^r$ for every  $p \in \NN$;       
    \item \label{i:permute} $\pi \sigma = \{\pi\ba_1, \dots, \pi\ba_d\} \in \Srq$ for every permutation $\pi \in S_q$;
    \item \label{i:shift} $\sigma +\bv = \{\ba_1+\bv, \dots, \ba_d+\bv\} \in \SS_q^{r+|\bv|}$ for every $\bv \in \ZZ^q$ when $\ba_i + \bv \geq {\bf 0}$ for each $i \in [d].$     \end{enumerate}
\end{theorem}

\begin{proof}
Items \eqref{i:sigma} and \eqref{i:extremal} are equivalent by \cref{t:shift} by noting that $C$ represents a subface of $\sigma.$ Item \eqref{i:min-max-witness} is also equivalent to \eqref{i:extremal} by~\cref{p:min-max}.

 To see that \eqref{i:sigma} implies \eqref{i:0-left} suppose $\bw \in {\mathcal N}_{q+p}^r$ satisfies   
 $$\bw\cdot \be_{A} \leq \max_{i \in C}\{\ba_i0^p \cdot \be_{A}\}$$
 for all $A \subseteq [p+q]$ and all $C \subseteq [d].$ Write $\bw = \bw_10^p + 0^q\bw_2$ where $\supp(\bw_1) \subseteq [q]$ and $\supp(\bw_2) \subseteq \{ q+1, \ldots, q+p\}.$ If $\bw_2 \ne 0$, selecting $A = \supp(\bw_2)$ and $C=[d]$ gives a contradiction since $\ba_i0^p \cdot \be_A=0$ for all $i$. Thus $\supp(\bw) \subseteq [q].$
 
 For $A \subseteq [p+q],$ write $A = A_1 \cup A_2$ where $A_1 \subseteq [q]$ and $A_2 \subseteq \{ q+1, \ldots , q+p\}.$ Note that $\ba_i0^p \cdot \be_A = \ba_i \cdot \be_{A_1}$ and $\bw \cdot \be_A= \bw_1 \cdot \be_{A_1}.$
 Thus
 $$\bw_1 \cdot \be_{A_1} \leq \max_{i \in C}\{\ba_i \cdot \be_{A_1}\}.$$ 
So $\bw_1 = \ba_i$ for some $i\in C$ and $\bw = \ba_i0^p$ as desired. 

To see that \eqref{i:0-left} implies \eqref{i:sigma}, notice that if $\bw \in \Nrq$ satisfies 
 $$\bw \cdot \be_A \leq \max_{i \in C} \{\ba_i \cdot \be_A\} \qforall A \subseteq [q]$$
 then $\bw0^p$ satisfies 
 $$\bw0^p \cdot \be_{A} \leq \max_{i \in C}\{\ba_i0^p \cdot \be_{A}\} \}$$
 so by \eqref{i:0-left}, we have $\sigma 0^p \in \SS_{q+p}^r$, so $\bw0^p = \ba_i 0^p$ for some $i$ by \eqref{i:extremal}. Thus $\bw = \ba_i.$

 The equivalence of \eqref{i:sigma} and \eqref{i:0-right} is similar. 

 Next, if $A \subseteq [q]$ and $\pi \in S_q$ define $\pi(A) = \{ \pi(i) \mid i \in A\}$ and notice that 
 $$\bw \cdot \be_A \leq \max_{i \in C}\{\ba_i \cdot \be_{A}\} \Leftrightarrow \pi(\bw) \cdot \be_{\pi(A)} \leq \max_{i \in C}\{\pi(\ba_i)\cdot \be_{\pi(A)}\}$$
 since for ${\bf x} \in \Nrq$ and $A \subseteq [q]$ we have ${\bf x}\cdot \be_A = \pi({\bf x})\cdot \be_{\pi(A)}.$
 Then the equivalence of \eqref{i:sigma} and \eqref{i:permute} follows~immediately.

 The equivalence of \eqref{i:sigma} and \eqref{i:shift} follows from \eqref{i:extremal}. To see this, note that 
$$(\bw + \bv) \cdot \be_A - (\ba_i + \bv) \cdot \be_A =\bw \cdot \be_A - \ba_i \cdot \be_A.$$ 
Thus $\bw \cdot \be_A - \ba_i \cdot \be_A \leq 0$ if and only if 
$(\bw + \bv) \cdot \be_A - (\ba_i + \bv) \cdot \be_A \leq 0.$
Fix $C \subseteq [d].$ Assume $\bw \in \Nrq$ is a solution to the system
 $$
            \bw \cdot \be_A \leq \max_{i \in C} \{\ba_i \cdot \be_A\} 
            \qforall A \subseteq [q].
        $$
        Then $\bw$ lies in the convex polytope associated to $\sigma = \{ \ba_i \mid i \in C\}.$ Since $\ba_i + \bv \geq {\bf 0}$ for all $i$, the convex polytope associated to $\sigma + \bv = \{ \ba_i + \bv \mid i \in C\}$ lies in $\N_q^{r + |\bv|}.$ Since $\bw + \bv$ is in this polytope, $\bw + \bv \in \N_q^{r + |\bv|}.$ Thus if $\sigma \not\in \Srq$ then there exists a $\bw \in \Nrq$ with $\bw$ a solution to the system but $\bw \ne \ba_i$ for all $i \in C$, which implies $\bw + \bv \ne \ba_i + \bv$ for all $i \in C.$ The converse also follows using $-\bv$ for $\bv.$
 \end{proof}

\begin{example}
The labels in~\cref{f:q2}, setting $r=4$, correspond to the right-hand side of the full triangle in~\cref{f:q3}, 
illustrating how faces can be embedded in higher dimensions as in~\cref{t:char-face}\eqref{i:0-left}. 
In the diagram below, note that $\sigma = \{ (4,0,0), (3,1,0), (3,0,1)\}$ and $\tau = \{(2,0,2), (1,1,2), (1,0,3)\}$ satisfy $\sigma + \bv = \tau$ where $\bv = (-2,0,2)$. Moreover, all triangles in this orientation are translations of $\sigma.$ Note that $\rho = \{(2,0,0), (1,0,1), (1,1,0)\} \in \SS_3^2$ (see \cref{e:sandwiches}) so if $\bv' = (2,0,0)$ we have $\sigma = \rho + \bv' \in \SS_3^4$, so translations in \cref{t:char-face}\eqref{i:shift} can change $r.$ Note that adding $\bv'$ to every face of $\SS_3^2$ yields a copy of $\SS_3^2$ embedded in the lower left corner of the picture of $\SS_3^4$ below. This type of embedding can be generalized so that one can see multiple copies of $\SS_q^{r'}$ in $\Srq$ whenever $r' < r.$

Finally, if $\pi \in S_3$ is given by $\pi = (132),$ then $\pi \tau = \gamma$ where $\gamma = \{ (0,2,2), (1,2,1), (0,3,1)\}.$ 

\begin{center}
    \begin{tikzpicture}[scale=0.75]
\coordinate (1111) at (0, 0);
\coordinate (1112) at (1, 0);
\coordinate (1122) at (2,0);
\coordinate (1222) at (3, 0);
\coordinate (2222) at (4, 0);

\coordinate (1113) at (0.5, 0.865);
\coordinate (1123) at (1.5, 0.865);
\coordinate (1223) at (2.5,0.865);
\coordinate (2223) at (3.5, 0.865);

\coordinate (1133) at (1, 1.73);
\coordinate (1233) at (2, 1.73);
\coordinate (2233) at (3, 1.73);

\coordinate (1333) at (1.5, 2.595);
\coordinate (2333) at (2.5, 2.595);

\coordinate (3333) at (2, 3.46);

\draw [fill=gray!20](1111) -- (1112) -- (1113);
\draw [fill=gray!20](1333) -- (1133) -- (1233);
\draw [fill=gray!20](2223) -- (1223) -- (2233);

\draw[black, fill=black] (1111) circle(0.05);
\draw[black, fill=black] (1122) circle(0.05);
\draw[black, fill=black] (1112) circle(0.05);
\draw[black, fill=black] (1222) circle(0.05);
\draw[black, fill=black] (2222) circle(0.05);
\draw[-] (1111) -- (2222);

\node[label = {[xshift=-0.6cm, yshift=-0.45cm]{\small $(4,0,0)$}}] at (1111) {};
\node[label = {[xshift=-0.35cm, yshift=-0.7cm]{\small $(3,1,0)$}}] at (1112) {};
\node[label = {[xshift=0cm, yshift=-0.7cm]{\small $(2,2,0)$}}] at (1122) {};
\node[label = {[xshift=0.35cm, yshift=-0.7cm] {\small $(1,3,0)$}}] at (1222) {};   
\node[label = {[xshift=0.6cm, yshift=-0.45cm]{\small $(0,4,0)$}}] at (2222) {};      

\draw[black, fill=black] (1113) circle(0.05);
\draw[black, fill=black] (1123) circle(0.05);
\draw[black, fill=black] (1223) circle(0.05);
\draw[black, fill=black] (2223) circle(0.05);
\draw[-] (1113) -- (2223);
\node[label = {[xshift=-0.6cm, yshift=-0.45cm]{\small $(3,0,1)$}}] at (1113) {};
\node[label = {[xshift=0.6cm, yshift=-0.45cm]{\small $(0,3,1)$}}] at (2223) {};   

\draw[black, fill=black] (1133) circle(0.05);
\draw[black, fill=black] (1233) circle(0.05);
\draw[black, fill=black] (2233) circle(0.05);
\draw[-] (1133) -- (2233);
\node[label = {[xshift=-0.6cm, yshift=-0.45cm]{\small $(2,0,2)$}}] at (1133) {};
\node[label = {[xshift=0.6cm, yshift=-0.45cm]{\small $(0,2,2)$}}] at (2233) {};  

\coordinate (1333) at (1.5, 2.595);
\coordinate (2333) at (2.5, 2.595);
\draw[black, fill=black] (1333) circle(0.05);
\draw[black, fill=black] (2333) circle(0.05);
\draw[-] (1333) -- (2333);
\node[label = {[xshift=-0.6cm, yshift=-0.45cm]{\small $(1,0,3)$}}] at (1333) {};
\node[label = {[xshift=0.6cm, yshift=-0.45cm] {\small $(0,1,3)$}}] at (2333) {};

\draw[black, fill=black] (3333) circle(0.05);
\node[label = {[xshift=0cm, yshift=-0.15cm]{\small $(0,0,4)$}}] at (3333) {};

\draw[-] (1111) -- (3333);
\draw[-] (1112) -- (2333);
\draw[-] (1122) -- (2233);
\draw[-] (1222) -- (2223);
\draw[-] (1112) -- (1113);
\draw[-] (1122) -- (1133);
\draw[-] (1222) -- (1333);
\draw[-] (2222) -- (3333);

 \end{tikzpicture} 
 \end{center}
\end{example}

As $r$ grows, \cref{t:char-face} provides a way to use inductive arguments algebraically on the vertices to identify faces of $\Srq$.

\begin{example}\label{e:using-char-face}
    We can use \cref{t:char-face} to show that $\{(3,0,0,0),(1,0,1,1)\}$ is not a face of $\mathbb{S}_4^3$:  
    $$\begin{array}{lcll}
    \{(3,0,0,0),(1,0,1,1)\} \in \mathbb{S}_4^3 & \iff & \{(2,0,0,0),(0,0,1,1)\} \in \mathbb{S}_4^2 & \mbox{by \cref{t:char-face}\eqref{i:shift}}\\
                                               & \iff & \{(0,0,0,2),(0,1,1,0)\} \in \mathbb{S}_4^2 & \mbox{by \cref{t:char-face}\eqref{i:permute}}\\
                                               & \iff & \{(0,0,2),(1,1,0)\}     \in \mathbb{S}_3^2 & \mbox{by \cref{t:char-face}\eqref{i:0-left}},
    \end{array}$$ 
    but this last statement is false, by~\cref{e:sandwiches}. 

    Similarly, we can show $\{(2,1,1,0),(2,0,1,1),(1,1,1,1)\}$ is a face of $\mathbb{S}^4_4$.  
      $$\begin{array}{cll}
         & \{(2,1,1,0),(2,0,1,1),(1,1,1,1)\} \in \mathbb{S}^4_4  & \\
    \iff & \{(1,1,0,0),(1,0,0,1),(0,1,0,1)\} \in \mathbb{S}_4^2  & \mbox{by \cref{t:char-face}\eqref{i:shift}}\\
    \iff & \{(1,1,0,0),(1,0,1,0),(0,1,1,0)\}  \in \mathbb{S}_4^2 & \mbox{by \cref{t:char-face}\eqref{i:permute}}\\
    \iff & \{(1,1,0),(1,0,1),(0,1,1)\} \in \mathbb{S}^2_3        & \mbox{by \cref{t:char-face}\eqref{i:0-right}},
    \end{array}$$ 
    and this last statement is true, by~\cref{e:sandwiches}. 
\end{example}

    \cref{t:char-face} lays the foundation for our current methods of finding faces of $\SS^r_q$ for arbitrary $q, r$. \cref{t:Urq-Srq} below, quoted from~\cite{extremal}, gives us a backbone for building $\Srq$ in terms of square-free vectors. 
    See also~\cref{r:partitions}.

\begin{definition}[{\bf $\Urq$, $\Uar$, and $\UUrq$}~\cite{extremal}] \label{d:Urq}
For $r,q\geq 1$, $\ba=(a_1,\ldots, a_q) \in \NN^q$ let 
\begin{itemize}
\item $\Urq$ denote the set of all square-free
   $q$-tuples with $r$ nonzero entries, i.e.,
$$\Urq = \{\be_A \st A \subseteq [q], \, |A|=r\} \,\, = \,\, \{\ba \in \Nrq,\ a_i \in \{0, 1\} {\text{ for all }}i\in [q]\};$$

\item $\Uar = 
\ba + \U_q^{r-|\ba|}= 
\{( c_1 ,\ldots , c_q) \in \Nrq \, \mid \,
a_i\leq c_i \leq a_i + 1 \qforall i \in [q] \}$. 

\item $\UUrq$ is  the simplicial
complex described by its facets, as follows:
$$\UUrq = \langle \Uar \st \ba \in \NN^q , \quad  r - q < |\ba| < r \rangle.$$ 

\end{itemize}
\end{definition}

\begin{remark}\label{r:hypersimplex}
The realization of $\Urq$ as a polytope in $\NN^q$ is a {\bf hypersimplex} \cite{GGL}, i.e., the realization is the convex hull of every $0/1$-vector in $\NN^q$ whose coordinates add up to a fixed integer. See also~\cite{DST} and~\cite[Example~3.3]{Ziegler}.  Many invariants of a hypersimplex count interesting combinatorial objects; for example, its volume counts permutations by number of ascents (Eulerian numbers)~\cite{Laplace,Stanley}. 
More generally, a polytope that is the convex hull of $0/1$-vectors is called a $0/1$-polytope~\cite{Ziegler-lectures}.
Many well-known polytopes are the convex hull of sets of $0/1$-vectors corresponding to combinatorial objects.  These include permutation polytopes~\cite{permutation-polytopes}, matching polytopes~\cite{matching-polytopes}, and matroid polytopes~\cite{matroid-polytopes}.   The matroid polytope of a uniform matroid is a hypersimplex and vice versa.
\end{remark}

The simplicial complex $\UUrq$ has a nice and inductive structure, and it turns out that a large part of $\Srq$ consists only of $\UUrq$, as we state below.

\begin{theorem}[{\cite[Theorem~5.9~and~Corollary~6.4]{extremal}}]\label{t:Urq-Srq} For $r, q \geq 1$, $\facets(\UUrq) \subseteq  \facets(\Srq)$. Moreover $\Srq=\UUrq$ when $q\leq 4$ and $r\geq 1$. \end{theorem}

To better understand these important facets, we provide an example.

\begin{example}\label{e:q34}
We illustrate \cref{d:Urq} and \cref{t:Urq-Srq} starting with small $r$ and $q.$ For $q=3$ and $r=1$, $\TT$ is a triangle with square-free vertex label so ${\mathcal U}_3^1 = \SS_3^1 = \TT_3^1$.  Now $r-q = -2$ and $r=1$ so if $r-q < |\ba| < r$ we have $|\ba|=0.$ Thus $\UU_3^1 = {\mathcal U}_3^1 = {\mathcal U}_{\bf 0}^1.$  

Now for $q=3$ and $r=2$, we have $-1 < |\ba| < 2$ in the definition of $\UU_3^2.$ If $|\ba| = {0}$ we have ${\mathcal U}_3^2 = {\mathcal U}_{\bf 0}^2$ consists of three square-free vectors forming a central triangle in the figure below. Note that the remaining three triangles are $\be_1 + {\mathcal U}_3^1= {\mathcal U}_{\be_1}^2,$ $\be_2 + {\mathcal U}_3^1= {\mathcal U}_{\be_2}^2,$ and $\be_3 + {\mathcal U}_3^1 = {\mathcal U}_{\be_3}^2.$ So $\UU_3^2 = \langle {\mathcal U}_3^2, {\mathcal U}_{\be_1}^2, {\mathcal U}_{\be_2}^2, {\mathcal U}_{\be_3}^2 \rangle = \SS_3^2.$

\begin{center}
    \begin{tabular}{cc}
\begin{tikzpicture}[baseline=(current bounding box.south)]
\coordinate (100) at (0, 0);
\coordinate (010) at (1, 0);
\coordinate (001) at (.5,.865);

\draw [fill=gray!20](100) -- (001) -- (010);

\draw[black, fill=black] (100) circle(0.05);
\draw[black, fill=black] (010) circle(0.05);
\draw[black, fill=black] (001) circle(0.05);
\draw[-] (100) -- (010);
\draw[-] (010) -- (001);
\draw[-] (100) -- (001);

 \node[label = below: {\scriptsize{$(1,0,0)$}}] at (100) {};
 \node[label = below: {\scriptsize{$(0,1,0)$}}] at (010) {};
 \node[label = above: {\scriptsize{$(0,0,1)$}}] at (001) {};

 \end{tikzpicture} 
 
&
    
\begin{tikzpicture}[baseline=(current bounding box.south)]
\coordinate (200) at (0, 0);
\coordinate (110) at (1, 0);
\coordinate (020) at (2,0);
\coordinate (101) at (0.5, 0.865);
\coordinate (011) at (1.5, 0.865);
\coordinate (002) at (1, 1.73);

\coordinate (Z) at (0,0);
\draw[] (Z) circle(0.0001);

\draw [fill=gray!20](200) -- (110) -- (101);
\draw [fill=gray!20](101) -- (110) -- (011);
\draw [fill=gray!20](110) -- (011) -- (020);
\draw [fill=gray!20](101) -- (002) -- (011);

\draw[black, fill=black] (200) circle(0.05);
\draw[black, fill=black] (110) circle(0.05);
\draw[black, fill=black] (020) circle(0.05);
\draw[black, fill=black] (101) circle(0.05);
\draw[black, fill=black] (011) circle(0.05);
\draw[black, fill=black] (002) circle(0.05);
\draw[-] (200) -- (020);
\draw[-] (020) -- (002);
\draw[-] (200) -- (002);
\draw[-] (110) -- (011);
\draw[-] (101) -- (011);
\draw[-] (110) -- (101);

 \node[label = below: {\scriptsize{$(2,0,0)$}}] at (200) {};
 \node[label = below: {\scriptsize{$(0,2,0)$}}] at (020) {};
 \node[label = above: {\scriptsize{$(0,0,2)$}}] at (002) {};
 \node[label = below: {\scriptsize{$(1,1,0)$}}] at (110) {};
 \node[label = left: {\scriptsize{$(1,0,1)$}}] at (101) {};
 \node[label = right: {\scriptsize{$(0,1,1)$}}] at (011) {};

 \end{tikzpicture} 
\end{tabular}
\end{center}

Next we consider $q=4.$ If $r=1$, then $\TT^1_4$ is a tetrahedron with vertices labeled $\be_1, \be_2, \be_3, \be_4.$ Again $\SS_4^1 = \TT_4^1 = {\mathcal U}_4^1 = \UU_4^1$. For $q=4$ and $r=2$ we again have $\SS_4^2 = \UU_4^2$ by \cref{t:Urq-Srq}. To understand $\UU_4^2$ notice that once again $-2 < |\ba| < 2$ so $\UU_4^2 = \langle {\mathcal U}_4^2, {\mathcal U}_{\be_1}^2,{\mathcal U}_{\be_2}^2,{\mathcal U}_{\be_3}^2,{\mathcal U}_{\be_4}^2\rangle.$ Now ${\mathcal U}_{\be_i} = \be_i + {\mathcal U}_4^1$ is a shift of a tetrahedron. These four facets are the four corner tetrahedra of the subdivided tetrahedron in \cref{f:q4r2}. The octahedron that remains after truncating these four corner tetrahedra is ${\mathcal U}_4^2$, whose vertices are the six points $(1,1,0,0), (1,0,1,0),(1,0,0,1),(0,1,1,0),(0,1,0,1),(0,0,1,1)$. Note that algebraically these six points form an abstract five-dimensional simplex in $\SS_4^2$, so every subset of those points is a face of $\SS_4^2$.  But the embedding of those points in $\N^2_4$, a three-dimensional hyperplane in $\NN^4$, represents this five-dimensional simplex degenerately as a three-dimensional octahedron, a hypersimplex. In other words, every subset of vertices in this octahedron forms a face of $\SS_4^2$, even if we do not draw those faces in the diagram.  The advantages of drawing $\SS_4^2$ this way in the diagram include: 
it illustrates the adjacencies to other facets of $\SS_4^2$; it allows us to see symmetries and transformations, such as those illustrated in \cref{t:char-face};   
and finally it is easier to see a three-dimensional picture than a five-dimensional picture.  

\begin{figure}
\hfill
\begin{center}
\begin{subfigure}{0.3\textwidth}
\begin{center}
\begin{tikzpicture}[scale=0.4]

\coordinate (raise) at (0,-4.5);
\draw[black, fill=black] (raise) circle(0.001);

\coordinate (2000) at (0, 0);
\coordinate (1100) at (3.5, 0);
\coordinate (1010) at (2, 1.25);
\coordinate (1001) at (2,3);
\coordinate (0200) at (7,0);
\coordinate (0110) at (5.5, 1.25);
\coordinate (0101) at (5.5, 3);
\coordinate (0020) at (4,2.5);
\coordinate (0011) at (4, 4.25);
\coordinate (0002) at (4,6);

   \coordinate (2000a) at (0, .1);
\coordinate (1100a) at (3.5, .1);
\coordinate (1010a) at (2.15, 1.25);
\coordinate (1001a) at (2.1,3);
\coordinate (0200a) at (7,.1);
\coordinate (0110a) at (5.35, 1.25);
\coordinate (0101a) at (5.4, 3);
\coordinate (0020a) at (4,2.34);
\coordinate (0011a) at (3.85, 4.25);
\coordinate (0002a) at (4,5.85);

\draw[black, fill=black] (2000) circle(0.075);
\draw[black, fill=black] (0200) circle(0.075);
\draw[black, fill=black] (0020) circle(0.075);
\draw[black, fill=black] (0002) circle(0.075);
\draw[black, fill=black] (1100) circle(0.075);
\draw[black, fill=black] (1010) circle(0.075);
\draw[black, fill=black] (1001) circle(0.075);
\draw[black, fill=black] (0110) circle(0.075);
\draw[black, fill=black] (0101) circle(0.075);
\draw[black, fill=black] (0011) circle(0.075);

\draw[-, very thick] (2000) -- (0200);
\draw[-, semithick] (2000) -- (0020);
\draw[-, very thick] (2000) -- (0002);
\draw[-, semithick] (0200) -- (0020);
\draw[-, very thick] (0200) -- (0002);
\draw[-, semithick] (0020) -- (0002);
\draw[-] (1100) -- (1001);
\draw[-] (1100) -- (0101);
\draw[-] (1001) -- (0101);
\draw[-, very thin] (1010) -- (0110);
\draw[-, very thin] (1010) -- (0011);
\draw[-, very thin] (0110) -- (0011);
\draw[-, very thin] (1010) -- (1100);
\draw[-, very thin] (1010) -- (1001);
\draw[-, very thin] (0110) -- (1100);
\draw[-, very thin] (0110) -- (0101);
\draw[-, very thin] (0011) -- (1001);
\draw[-, very thin] (0011) -- (0101);

 \node[label = below: {\scriptsize{$(2,0,0,0)$}}] at (2000a) {};
 \node[label = below: {\scriptsize{$(0,2,0,0)$}}] at (0200a) {};
 \node[label = above: {\scriptsize{$(0,0,2,0)$}}] at (0020a) {};
 \node[label = above: {\scriptsize{$(0,0,0,2)$}}] at (0002a) {};
 \node[label = below: {\tiny{$(1,1,0,0)$}}] at (1100a) {};
 \node[label = left: {\tiny{$(1,0,1,0)$}}] at (1010a) {};
 \node[label = right: {\tiny{$(0,1,1,0)$}}] at (0110a) {};
 \node[label = left: {\tiny{$(1,0,0,1)$}}] at (1001a) {};
 \node[label = right: {\tiny{$(0,1,0,1)$}}] at (0101a) {};
 \node[label = right: {\tiny{$(0,0,1,1)$}}] at (0011a) {}; 
 \end{tikzpicture} 
\end{center}
\caption{$\SS^2_4 = \UU^2_4$}
\label{f:q4r2}
\end{subfigure}
\hspace{0.2in}
\begin{subfigure}{0.4\textwidth}
\begin{center}
 \begin{tikzpicture}[scale=0.4]
\coordinate (3000) at (0, 0);
\coordinate (2100) at (3.5, 0);
\coordinate (2010) at (2, 1.25);
\coordinate (2001) at (2,3);
\coordinate (1200) at (7,0);
\coordinate (1110) at (5.5, 1.25);
\coordinate (1101) at (5.5, 3);
\coordinate (1020) at (4,2.5);
\coordinate (1011) at (4, 4.25);
\coordinate (1002) at (4,6);
\coordinate (0003) at (6,9);
\coordinate (0300) at (10.5,0);
\coordinate (0210) at (9, 1.25);
\coordinate (0201) at (9, 3);
\coordinate (0120) at (7.5,2.5);
\coordinate (0102) at (7.5,6);
\coordinate (0030) at (6, 3.75);
\coordinate (0021) at (6, 5.5);
\coordinate (0012) at (6, 7.25);
\coordinate (0111) at (7.5, 4.24);

   \coordinate (3000a) at (0, .1);
\coordinate (2100a) at (3.5, .1);
\coordinate (2010a) at (2.15, 1.25);
\coordinate (2001a) at (2.1,3);
\coordinate (1200a) at (7,.1);
\coordinate (1110a) at (5.35, 1.25);
\coordinate (1101a) at (5.4, 3);
\coordinate (1020a) at (4,2.34);
\coordinate (1011a) at (3.85, 4.25);
\coordinate (1002a) at (4.1,6);
\coordinate (0003a) at (6,8.85);
\coordinate (0300a) at (10.4,0);
\coordinate (0210a) at (9, 1.15);
\coordinate (0201a) at (8.8, 3);
\coordinate (0120a) at (7.4,2.5);
\coordinate (0102a) at (7.4,6);
\coordinate (0030a) at (6, 3.75);
\coordinate (0021a) at (6, 5.5);
\coordinate (0012a) at (6, 7.25);
\coordinate (0111a) at (7.7, 4.7);

\draw[black, fill=black] (3000) circle(0.05);
\draw[black, fill=black] (1200) circle(0.05);
\draw[black, fill=black] (1020) circle(0.05);
\draw[black, fill=black] (1002) circle(0.05);
\draw[black, fill=black] (2100) circle(0.05);
\draw[black, fill=black] (2010) circle(0.05);
\draw[black, fill=black] (2001) circle(0.05);
\draw[black, fill=black] (1110) circle(0.11);
\draw[black, fill=black] (1101) circle(0.11);
\draw[black, fill=black] (1011) circle(0.11);
\draw[black, fill=black] (0003) circle(0.05);
\draw[black, fill=black] (0300) circle(0.05);
\draw[black, fill=black] (0210) circle(0.05);
\draw[black, fill=black] (0201) circle(0.05);
\draw[black, fill=black] (0120) circle(0.05);
\draw[black, fill=black] (0102) circle(0.05);
\draw[black, fill=black] (0030) circle(0.05);
\draw[black, fill=black] (0021) circle(0.05);
\draw[black, fill=black] (0012) circle(0.05);
\draw[black, fill=black] (0111) circle(0.11);

\draw[-, very thick] (3000) -- (0300);
\draw[-, semithick] (3000) -- (0030);
\draw[-, very thick] (3000) -- (0003);
\draw[-, semithick] (0300) -- (0030);
\draw[-, very thick] (0300) -- (0003);
\draw[-, semithick] (0030) -- (0003);

\draw[-] (2001) -- (0201);
\draw[-] (1002) -- (0102);
\draw[-] (2001) -- (2100);
\draw[-] (1002) -- (1200);
\draw[-] (1200) -- (0201);
\draw[-] (2100) -- (0102);

\draw[-, very thin] (2010) -- (0210);
\draw[-, very thin] (1020) -- (0120);
\draw[-, very thin] (1020) -- (1200);
\draw[-, very thin] (2100) -- (0120);
\draw[-, very thin] (1200) -- (0210);
\draw[-, very thin] (2010) -- (2100);
\draw[-, very thin] (2001) -- (2010);
\draw[-, very thin] (1002) -- (1020);
\draw[-, very thin] (2001) -- (0021);
\draw[-, very thin] (1002) -- (0012);
\draw[-, very thin] (2010) -- (0012);
\draw[-, very thin] (1020) -- (0021);

\draw[-, very thin] (0210) -- (0201);
\draw[-, very thin] (0120) -- (0102);
\draw[-, very thin] (0021) -- (0201);
\draw[-, very thin] (0012) -- (0102);
\draw[-, very thin] (0012) -- (0210);
\draw[-, very thin] (0021) -- (0120);

\draw[-, very thin] (1110) -- (1101);
\draw[-, very thin] (1110) -- (1011);
\draw[-, very thin] (1110) -- (0111);
\draw[-, very thin] (1011) -- (1101);
\draw[-, very thin] (1011) -- (0111);
\draw[-, very thin] (1101) -- (0111);

 \node[label = below: {\tiny{$(3,0,0,0)$}}] at (3000a) {};
 \node[label = below: {\tiny{$(1,2,0,0)$}}] at (1200a) {};
 \node[label = left: {\tiny{$(1,0,0,2)$}}] at (1002a) {};
 \node[label = below: {\tiny{$(2,1,0,0)$}}] at (2100a) {};
 \node[label = left: {\tiny{$(2,0,0,1)$}}] at (2001a) {};
 \node[label = above: {\tiny{$(0,0,0,3)$}}] at (0003a) {};
 \node[label = below: {\tiny{$(0,3,0,0)$}}] at (0300a) {}; 
 \node[label = right: {\tiny{$(0,2,0,1)$}}] at (0201a) {};
 \node[label = right: {\tiny{$(0,1,0,2)$}}] at (0102a) {}; 
 \end{tikzpicture} 
\end{center}
\caption{$\SS^3_4 = \UU^3_4$}
\label{f:q4r3}
\end{subfigure}
\caption{$\SS^r_4 = \UU^r_4$ for $r = 2, 3$.  The top of $\UU^3_4$ is the same as all of $\UU^2_4$ with $\be_4$ added to each vertex, and similarly for the other three corners of $\UU^3_4$ with $\be_1, \be_2, \be_3$, respectively.  See~\cref{e:q34}. The complexes $\SS^r_q$ and $\UU^r_q$ are no longer equal for $q \geq 5$ and $r \geq 3$; see~\cref{p:weird,e:picture-q5-r3}.}\label{f:q4}
 \end{center}
\hfill
\end{figure}

When $q=4$ and $r=3$, we again see, in \cref{f:q4r3}: ten translated copies of $\U_4^1$ (tetrahedra with a triangular base at the bottom), one at each of the four corners, and one along each of the six edges; four translated copies of $\U_4^2$ (octahedra); and one copy of $\U_4^3$ (tetrahedron near the bottom of the diagram with a triangular base at the top). The (square-free) vertices of $\U_4^3$, which are located at the centroid of each triangular face, are shown larger to make this new face easier to identify.  Note that $\U_4^3$ is a reflection of each of the translated copies of $\U_4^1$.  We also again see four translated copies of $\SS_4^2$, one in each corner of the diagram. 

The pictures above are reminiscent of edge-wise subdivisions of simplicial complexes; we refer the interested reader to \cref{log-concave}.
\end{example}

The equivalence of items \eqref{i:extremal} and \eqref{i:min-max-witness} of \cref{t:char-face} leads to the definition of a witness, allowing for two formulations that can be useful. The existence of a witness tracks the failure of a face of the Taylor complex to be Scarf. In light of condition \eqref{i:subsets} of \cref{t:shift}, the terminology allows for either condition of \cref{t:shift} to fail via the use of a subface $\sigma'$ of $\sigma$ defined by a subset $C$ of the indices of $\ba_i$ defining $\sigma.$

\begin{definition}\label{d:witness}
A vector $\bw \in \Nrq$ is a {\bf witness} to $\sigma = \{\ba_1, \ldots,\ba_d\} \notin \Srq$ if there exists a $C \subseteq [d]$ such that $\bw \ne \ba_i$ for all $i \in C$, and one of the following equivalent conditions holds:
\begin{itemize}
    \item[(i)] for all $A \subseteq [q]$
    $$\bw \cdot \be_A \leq \max_{i\in C}\{\ba_i \cdot \be_A\};$$
    \item[(ii)] for all $A \subseteq [q]$
    $$\min_{i \in C} \{\ba_i \cdot \be_A\} \leq \bw \cdot \be_A \leq \max_{i\in C} \{\ba_i \cdot \be_A\}.$$
\end{itemize}
\end{definition}
Note that by \cref{t:char-face}, the two conditions in \cref{d:witness} are equivalent when considered over all $A \subseteq [q]$. Either condition can be used, with the optimal choice depending on the situation at hand.

As an example, consider \cref{e:sandwiches}, where it was shown that $\bw = (1,0,1)$ is a witness to $\sigma = \{(0,0,2),(1,1,0)\} \notin \SS_3^2.$

 \cref{d:witness} says in particular that faces of $\Srq$ correspond to polytopes defined by special equations that have no interior lattice point. It is an interesting combinatorial question to study which polytopes can arise in this way. 

In light of \cref{t:Urq-Srq}, our main goal is understanding the faces of $\Srq$ which are not in $\UUrq$ for $q \ge 5.$ We start by identifying a large number of edges in $\Srq$ not in $\UUrq.$

\begin{lemma}\label{l:all-1-easy}
    Let $\bu, \bv \in \Nrq$.  Let $J \subseteq [q]$ and assume for all $j \in J$ that $u_j = 1$.  Also assume for some $i \not\in J$ that $1 \leq v_i < |J|$ and that $u_i=0$.  Then if $\bw \in \Nrq$ is a witness to $\{\bu, \bv\} \notin \Srq$
    then either $w_j = 0$ for all $j \in J$, or $w_j = 1$ for all $j \in J$. 
\end{lemma}
\begin{proof}
    Choosing $A = J$, \cref{p:min-max} implies $0 \leq w_j \leq 1$ for all $j \in J$.  We will show that $w_{j'} = w_{j''}$ for every pair of distinct $j', j'' \in J$, from which the result follows immediately.
    
    Pick $J_0 \subset J$ such that $|J_0| = v_i - 1$ and such that $j', j'' \in J \setminus J_0$.  It is possible to pick such a $J_0$ because $1 \leq v_i < |J|$.  Let $J' = J_0 \cup j'$ and $J'' = J_0 \cup j''$, and then let $A' = i \cup J'$ and $A'' = i \cup J''$.  It is easy to see that $\be_{A'} \cdot \bu = \be_{A'} \cdot \bv = \be_{A''} \cdot \bu = \be_{A''} \cdot \bv = v_i$.  \cref{p:min-max} then implies that $\be_{A'} \cdot \bw = \be_{A''} \cdot \bw = v_i$.  Then
    \[
        0 = (\be_{A'} - \be_{A''}) \cdot \bw 
            = (\be_{j'} - \be_{j''}) \cdot \bw
            = w_{j'} - w_{j''},
    \]
    as desired.
\end{proof}

\begin{proposition}[\textbf{Edges incident to a square-free vertex}]\label{p:square-free} 
Let $r,q >1.$ If $\bu \in \Urq$ and $\bv \in \Nrq$ are such that $\supp(\bu) \cap \supp(\bv)=\emptyset$,
then
\begin{center} $\{\bu,\bv\} \in \Srq$ if and only if $\bv \neq r\be_i$ for all $i \in [q]$.
\end{center}
\end{proposition}

\begin{proof}
    Assume $\bv \neq r\be_i$ for all $i \in [q]$, and assume $\bw$ is a potential witness to $\{\bu, \bv\} \not\in \Srq$,
    Let $J = \supp(\bu)$, and let $i \in \supp(\bv)$.  Since $|J|=r$ and $\bv \neq r\be_i$, then $1 \leq v_i < r$, and \cref{l:all-1-easy} applies. Then either: $\bw_j = 1$ for all $j \in \supp(\bu)$, in which case, $\bw = \bu$; or $\bw_j = 0$ for all $j \in \supp(\bu)$, in which case $\bw=\bv$.  Therefore $\bw$ is not a witness.  Since no witness exists, $\{\bu,\bv\} \in \Srq$.

    Conversely, the only edges of $\Srq$ containing $r\be_i$ are of the form $\{r\be_i, (r-1)\be_i + \be_j\}$ by \cite{Lr}, which contradicts $\supp(\bu) \cap \supp(\bv) = \emptyset$.  Alternatively, it is straightforward to check that $\bw=(r-1)\be_i + \be_j$ is a witness to $\{r\be_i,\bv\} \not\in \Srq$, where $j \in \supp(\bv)$.
\end{proof}

\begin{corollary}\label{c:reduce to square-free} 
Let $r,q >1.$ If $\ba,\bb \in \Nrq$ such that $\ba-(\ba\cap\bb)$ is square-free, then 
$\{\ba,\bb\} \in \Srq$ if and only if
\begin{enumerate}
\item $(r-|\ba \cap \bb|) =1$ or 
\item $(r-|\ba \cap \bb|) \ge 2$ and $\bb - (\ba \cap \bb) \neq (r-|\ba \cap \bb|) \be_i$ for all $i \in [q].$
\end{enumerate}
\end{corollary}

\begin{proof}
 This follows immediately from \cref{p:square-free}  and \cref{t:char-face}\eqref{i:shift}.
 \end{proof}

\begin{remark}\label{r:partitions}
    Finding faces for arbitrary values of $q, r$ is a very hard problem. 
    In order to check whether two vertices $\ba$ and $\bb$  
    share an edge
    in $\Srq$, \cref{t:char-face}\eqref{i:permute},\eqref{i:shift} imply we may assume $\ba$ and $\bb$ have disjoint support, and that the non-zero entries of each vertex are a weakly decreasing sequence of positive integers (also known as a {\bf partition}). See \cref{e:using-char-face}. 
    Equivalently, to fully describe edges of $\Srq$ for $q \geq 2r$, it is enough to understand the graph $G(q,r)$ with vertices the partitions of $r$, and two partitions $\lambda = (\lambda_1, \dots, \lambda_n)$ and $\mu = (\mu_1, \dots, \mu_m)$ are adjacent if and only if the vectors $\lambda 0^m$ and $0^n \mu$ are adjacent in $\Srq$.

    This equivalence of problems then implies that the number of edges of $\Srq$ can be computed by an expression in terms of the edges of $G(q,r)$.  
    In particular, this approach makes the problem of finding edges more tractable.
     We have so far found every edge for $r \leq 8$ and arbitrary $q$.

    We note that for the case of edges specifically, in~\cite[Theorem 7.2]{extremal} it is shown that the number of edges of $\Srq$ is exactly  the first betti number of $\Erq$. As a consequence, counting the number of edges of $\Srq$ is equivalent to bounding the first betti number of powers of all square-free monomial ideals.
\end{remark}

\begin{remark}
\cref{t:char-face}\eqref{i:shift} may be restated as saying that translations preserve whether or not a set of vertices is a face of $\Srq$.  Similarly, other rigid motions, i.e., rotations and reflections, also preserve this property.  The corresponding algebraic interpretations are more involved, and they are not helpful until $r \geq 4$, so we do not explore these further here.
   
\end{remark}

\section{The facets and minimal nonfaces  of  $\Sq$}\label{s:S3q}

When $r = 3$, there are only three types of vectors $\ba \in \N^3_q$ up to permutation:
$$
   \ba = 3\be_i, \quad \ba = 2\be_i + \be_j, \qor \ba = \be_i + \be_j + \be_k, \qfor i, j, k \in [q]. 
$$
\cref{t:char-face} allows us to  describe all faces of $\Sq$ for any $q$.

We start by recalling the characterization of the edges.
\begin{lemma}[{\bf The edges of $\Sq$ (see~{\cite[Lemma~8.1]{extremal}})}]
\label{l:types}
Let  $q\ge 2$ and $\ba\ne \bb$ in $\N^3_q$. Then $\{\ba, \bb\}\in \Sq$ if and only if $\{\ba, \bb\}$ is of one of the following forms for distinct indices $u,v,w,i,j,k \in [q]$: 
$$
\begin{array}{lll}
\{3\be_u, 2\be_u+\be_j\}&& \{2\be_i+\be_j, \be_u+\be_v+\be_w \}\\
\{2\be_u+\be_i, 2\be_u+\be_j\}&& \{\be_i+\be_j+\be_k, \be_u+\be_v+\be_w\}\\
\{2\be_u+\be_i, 2\be_i+\be_u\}&&\{\be_u+\be_i+\be_j, \be_u+\be_v+\be_w\}\\
\{2\be_u+\be_i, \be_u+\be_i+\be_j\} && \{\be_u+\be_v+\be_i, \be_u+\be_v+\be_j\}\\
\{2\be_u+\be_i, \be_u+\be_j+\be_k\}&&
\end{array}
$$
\end{lemma}
Notice that all but two of the types of edges listed in \cref{l:types} come from edges of $\Srq$ for $r<3$ that have been shifted in the sense of \cref{t:char-face}\eqref{i:shift}. For instance, the first type listed has the form $\sigma + \bv$ where $\sigma = \{\be_u, \be_j \} \in \SS_q^1$ and $\bv = 2\be_u.$ When $q \ge 6$, edges of the form $\{\be_i+\be_j+\be_k, \be_u+\be_v+\be_w\}$ are not shifted but they are in $\U_q^3.$ That leaves edges of the form $\{2\be_i+\be_j, \be_u+\be_v+\be_w \}$ for $q \ge 5$, which are of particular interest since they are neither shifted nor in $\U_q^3.$ 

Many of the higher dimensional faces of $\Sq$ are also either shifted from $\Srq$ for lower values of $r$ or are in $\U_q^3.$ However, there is one class of faces that is new in the sense that it is not formed in either of these two ways.

\begin{proposition}[\textbf{A new face(t) of $\Sq$}]\label{p:weird}
 Let $q$ be a positive integer, and suppose 
 $i \in P \subseteq [q]$.
 Then the simplex $W_{P,i}$ defined as 
 \begin{align*}
W_{P,i}= \{2\be_{i}+\be_{j} \st j \in P \setminus \{i\}\} 
 &\cup 
\{\be_j + \be_k + \be_l \st   j < k <l\leq q, \quad j,k,l \in {[q]\setminus P}\}\\
& \cup   
\{\be_i + \be_j + \be_k \st  1 \leq j < k \leq q, \quad i \notin \{j,k\}\} 
\end{align*}
is \
\begin{enumerate}
\item\label{i:P1} the facet $\U^3_q$ when $|P|=1$;
\item\label{i:P2} a facet of $\Sq$ that is not in $\UU_q^3$ when $2 \leq  |P| \leq  q-3$;
\item\label{i:Pbig} a face of $\Sq$  contained in the facet $\U_{\be_i}^3$  when $|P|>q-3$.
\end{enumerate}
Moreover, if $W_{P,i}=W_{Q,j}$ for $2\le |P|, |Q|\le q-3$, then $P=Q$ and $i=j$.
 \end{proposition}

\begin{proof}  Part~\eqref{i:P1} follows from the definition of $W_{P,i}$ by noting that in this case, $P=\{i\}$ and the first set in the definition is empty. The remaining two sets include all vertices in $\U_q^3.$ Similarly,~\eqref{i:Pbig} follows from noting that the second set in the definition of $W_{P,i}$ is empty and the remaining vertices are in $\U_{\be_i}^3$. If $P \ne [q]$, the containment will be proper.

Assume now $2\le |P|\le q-3$. Note that in this scenario, $q \geq 5.$ We prove first $W_{P,i}\in \Sq$.  By \cref{t:char-face}\eqref{i:permute}, it is enough to prove the case $P=[p]$ and $i=1$ for some $p$ with  $2\le p\le q-3$. In this case, 
 $W=W_{[p],1}$ is defined as
 \begin{equation}\label{e:weird}
\begin{array}{lcl}
\{2\be_{1}+\be_{j} \st j=2,\ldots,p\}
& \cup & 
\{\be_j + \be_k + \be_l \st p+1 \leq j < k < l \leq q\} 
\\
&\cup & 
\{\be_1 + \be_j + \be_k \st  1 < j < k \leq q\}.
\end{array}
\end{equation}
Assume $W\notin \Sq$. Let $W'\subseteq W$ be a minimal subset of $W$ (under inclusion) such that $W'\notin \Sq$. 

By \cref{t:char-face} there exists a witness $\bw \in \N_{q}^3 \setminus W'$ such that 
    \begin{equation}
    \label{e:w}
    \bw\cdot \be_A \leq \max\{\ba\cdot \be_A \st \ba \in W'\} \qforall A \subseteq [q] 
    \end{equation}
    Observe that 
    \begin{align}
    \begin{split}
    \label{e:w2}
    &\max\{\ba\cdot \be_1 \st \ba \in W'\}\le \max\{\ba\cdot \be_1 \st \ba \in W\} =2\qand\\
    &\max\{\ba\cdot \be_i \st \ba \in W'\}\le \max\{\ba\cdot \be_i \st \ba \in W\}=1 \qforall i\in [q]\smallsetminus \{1\}
    \end{split}
    \end{align}
 and hence \eqref{e:w}, with $A=\{1\}$ and $A=\{i\}$ respectively,  implies 
 \begin{equation}
 \label{e:w3}
w_1\le 2  \qand w_i\le 1 \qforall i\in [q]\smallsetminus \{1\}.
\end{equation}
 
 If $w_1 = 0$, then \eqref{e:w3} implies $\bw=\be_j+\be_k+\be_l$ with $1<j<k < l\le q$. Then  \eqref{e:w}  gives 
    $$
        3=\bw\cdot \be_{\{j,k,l\}}  \le   \max\{\ba\cdot \be_{\{j,k,l\}} \st \ba \in W'\}\le 3.
    $$
There exists thus $\bv\in W'$ such that $\bv\cdot \be_{\{j,k,l\}}=3$. Since $j,k,l>1$ and $W'\subseteq W$, this implies $\bv=\bw$, a contradiction.

If $w_1 = 1$, then \eqref{e:w3} implies $\bw=\be_1+\be_j+\be_k$ for some $1<j<k\le q$. As above, this implies that there exists $\bv\in W'$ such that $\bv\cdot \be_{\{1,j,k\}}=3$.  It follows that $$\bv=2\be_1+\be_j \qor \bv=2\be_1+\be_k \qor \bv=\be_1+\be_j + \be_k.$$  Since $\bv\neq \bw$, the only options are $\bv=2\be_1+\be_\ell$ with $\ell\in \{j,k\}$. Moreover, since $\bv\in W'\subseteq W$, the definition of $W$ implies $\ell\le p$; in particular $j\leq p$. 

Now  $\bw\cdot \be_{\{j,k\}}=2$ and \eqref{e:w} shows there must exist $\bv'\in W'$ with $\bv'\cdot \be_{\{j,k\}}\ge 2$. 
It follows that $\bv'$ is a square-free vector $$\bv'=\be_j+\be_k+\be_l \qwith p+1 \leq j,k,l   \qor  \bv'=\be_1+\be_j+\be_k .$$
The first scenario is impossible because we know $j\leq p$ and the second one is impossible because $\bv \neq \bw$, ending our argument for the case $w_1=1$.

If $w_1 = 2$, then    $\bw=2\be_1+\be_j$ with $j\in [q]$. 
Then using  \eqref{e:w} we have then 
        $$
            \bw\cdot \be_{\{1,j\}} = 3\le  \max\{\ba\cdot \be_{\{1,j\}} \st \ba \in W'\}\leq 3.
        $$
There exists thus $\bv\in W'$ such that $\bv\cdot \be_{\{1,j\}}=3$, so our only option is $\bv=2\be_1+\be_j=\bw$,  a contradiction.

Now we show that $W_{[p],1}$ is a facet of $\Sq$. Recall that $q \geq 5$. Assume that $\ba \in \Nrq \setminus W_{[p],1}$. Below we use \eqref{e:weird} to consider each possible $\ba$, and in each case we use the list in \cref{l:types}, or equivalently \cref{t:minimalnonfaces3}, to show that there is a vertex $\bb$ of $W_{[p],1}$ such that $\{\ba,\bb\} \notin \Sq$, showing that $W_{[p],1} \cup \{\ba\} \notin \Sq$.  

\begin{itemize}
\item $\ba=3\be_a$ with $a\in [q]$. Then let $\bb=\be_1+\be_2+\be_3$.
\item $\ba=2\be_a+\be_b$ with  $1$, $a$, $b$ distinct. Then let $\bb=\be_1+\be_b+\be_c$ with  $1$, $a$, $b$, $c$ distinct.
\item $\ba=2\be_a+\be_1$  with  $1$, $a$  distinct. Then let $\bb=\be_1+\be_b+\be_c$  with  $1$, $a$, $b$, $c$ distinct.
\item $\ba=2\be_1+\be_a$ with $a>p$. Then let $\bb=\be_a+\be_b+\be_c$  when $p< b < c \leq q$, and  $a$, $b$, $c$ distinct.
\item $\ba=\be_a+\be_b+\be_c$ for some $a<b<c$ where $1<a \leq p$. Then let $\bb=2\be_1+\be_a$. 
\end{itemize}
We conclude that $W_{[p],1}$ is a facet of $\Sq$.

Finally, notice that $W_{P,i}$ is not in $\U_q^3$ since it contains a vertex that is not square-free. Moreover, since $q\geq 5$, $W_{P,i}$ contains vertices $\bb=2\be_i+e_j$ with $j\in P\setminus\{i\}$ and $\bc = \be_k + \be_l + \be_u$ with $k,l,u \in [q]\setminus P$. Since $\bb, \bc \in W_{P,i}$ and $\supp(\bb) \cap \supp (\bc) = \emptyset$, then $W_{P,i}\neq \U_{\ba}^3$ for any $\ba \in \NN^q.$ Thus $W_{P,i} \notin \UU_q^3$. 

For the final statement, note that $2 \leq |P|, |Q| \leq q - 3$ implies all three sets in the definition of $W_{P, i} = W_{Q, j}$ are non-empty. In particular, a nonsquare-free vector in $W_{P, i} = W_{Q, j}$ implies $i = j$, and the set $\{\be_j + \be_k + \be_l \st j < k < l \leq q, \quad j,k,l \in [q]\setminus P\}$ implies $P = Q$.
\end{proof}

\begin{theorem}[{\bf The facets of $\Sq$}]\label{thm:complete-list-of-faces}
    The facets of $\Sq$ are of one of the following types: 
    \begin{enumerate}
    \item\label{i:2ea3}  (when $q\geq 1$)\, $\mathcal U_{2\be_a}^3=\{2\be_a+\be_i\colon i\in [q]\}$\, for some $a\in [q]$; 
    \item\label{i:eab3}  (when $q\geq 2$)\, $\mathcal U_{\be_a+\be_b}^3=\{\be_a+\be_b+\be_i\colon i\in [q]\}$\, for some $a,b\in [q]$ with $a\ne b$;
    \item\label{i:ea3}  (when $q\geq 3$)\, $\mathcal U_{\be_a}^3=\{\be_a+\be_i+\be_j\colon i,j\in [q], i<j\}$\, for some $a\in [q]$;
     \item\label{i:q3} (when $q\geq 4$)\, $\U_q^3=\{\be_i+\be_j+\be_k\colon 1\le i<j<k\le q \}$;    \item\label{i:WPa} (when $q\geq 5$)\, The simplex $W_{P,a}$ for some $P\subseteq [q]$ with $2\le |P|\le q-3$ and $a\in P$. 
    \end{enumerate}
In particular, 
$$\Sq=\UU^3_q \cup 
\langle W_{P,a} \st P\subseteq [q], \  2\le |P|\le q-3, \  a\in P \rangle.$$

\end{theorem}

\begin{proof}
    By \cref{t:Urq-Srq} and \cref{p:weird}, every item on the list above is a facet of $\Sq$. 
    
    Assume now that $F$  is a facet of $\Sq$. If $F\subseteq \U_q^3$, then it follows that $F=\U_q^3$, so we may assume  $F\not\subseteq \U_q^3$. Then $F$ must contain a nonsquare-free vector.

Assume first $3\be_i\in F$ for some $i\in [q]$.  If $F$ contains a vertex $\bv$ other than $3\be_i$, then $\{3\be_i,\bv\}\in \Sq$ and \cref{l:types} implies $\bv\in \U_{2\be_i}^3$ for all $\bv \in F.$ Thus $F\subseteq \U_{2\be_i}^3$, and hence $F=\U_{2\be_i}^3$. 

It remains to consider the case when  $3\be_i\notin F$ for all $i\in [q]$, and $F$ contains an element of the form $2\be_i+\be_j$ with $i\ne j$. Without loss of generality, assume $2\be_1+\be_2\in F$ and set
\[
P=\{u\in [q]\colon 2\be_1+\be_u\in F, u\ne 1\}\cup\{1\}\,.
\]
Without loss of generality, we may assume $P=[p]$, where $1<p\le q$. 
Suppose $F$ contains another nonsquare-free vector $2\be_a+\be_b$ where $a \neq b$ and $a \neq 1$. Then for every $u \in \{2,\ldots,p\}$ 
$$\{2\be_1+\be_u, 2\be_a+\be_b\}\in \Sq$$ and \cref{l:types} implies that $a=u$ and $b=1$.  This means, in particular
that either $2\be_1+\be_u$,with $1 < u \leq p $ are the only nonsquare-free vectors in $F$, or otherwise $p=2$. In other words
\begin{equation} \label{e:2cases}
F\subseteq  \{2\be_1+\be_u\colon u\in [p]\}\cup \U_q^3 \qor
F\subseteq  \{2\be_1+\be_2, 2\be_2+\be_1\}\cup \U_q^3.
\end{equation}

Now suppose $F$ contains a square-free vector $\be_a+\be_b+\be_c$ with $a<b<c$. Then we have 
$$\{2\be_1+\be_u, \be_a+\be_b+\be_c\}\in \Sq  \qwhere 2\le u\le p.$$ 
By \cref{l:types}, we must have 
\begin{equation}
\label{e:u}
a=1 \qor a \geq p+1.
\end{equation}
Now, going back to \eqref{e:2cases}, we will consider each case separately.

\begin{enumerate}
    \item If $F\subseteq  \{2\be_1+\be_u\colon u\in [p]\}\cup \U_q^3 $
then we see from \eqref{e:u} and  \eqref{e:weird} that $F\subseteq W_{[p],1}$. Recall $p\ge 2$. So by \cref{p:weird}, and since we assumed $F$ is a facet, we must have $F=W_{[p],1}$ 
or $F=\U_{\be_1}^3$. 
    \item If $F\subseteq  \{2\be_1+\be_2, 2\be_2+\be_1\}\cup \U_q^3$ then 
take $\be_a+\be_b+\be_c\in \U_q^3$, where $a<b<c$. By \eqref{e:u} we must have $a \neq 2$.
If $a\geq 3$, then 
\begin{equation}
\label{e:nontriangle}
\tau =\{2\be_1+\be_2, 2\be_2+\be_1, \be_a+\be_b+\be_c\} \notin \Sq
\end{equation}
because we can pick a witness $\bw = \be_1 + \be_2 + \be_a \notin \tau$, and for all $A\subseteq [q]$ we have  
    \begin{equation*}
        \bw \cdot \be_A =
        |A\cap \{1,2,a\}| \leq  
        \max \{(2\be_1 + \be_2)\cdot \be_A, (\be_1 + 2\be_2)\cdot \be_A, (\be_a + \be_b + \be_c)\cdot \be_A\}. 
    \end{equation*}
 Thus we must have $a=1$. On the other hand, since $\{ 2\be_2+\be_1, \be_1+\be_b+\be_c\}\in \Sq$, \cref{l:types} shows that $b=2$. 
 We have so far proved 
 \[
F\subseteq  \{2\be_1+\be_2, 2\be_2+\be_1\}\cup \{\be_1+\be_2+\be_c\colon 2<c\le q\}\,.
\]
This implies $F\subseteq \U_{\be_1+\be_2}^3$, and hence $F=\U_{\be_1+\be_2}^3$. 
\end{enumerate}

We have proved that in all possible scenarios $F$ is one of the facets listed above, and this ends our argument. 
\end{proof}

\cref{thm:complete-list-of-faces} provides a formulaic description of all the individual facets of $\Srq.$ To better understand this list, we discuss the geometry of the facets and how they fit together in the following example. 

\begin{example}\label{e:picture-q5-r3} {\bf (Picturing $\SS_5^3$.)}

The facets described in~\cref{thm:complete-list-of-faces}\eqref{i:2ea3}-\eqref{i:q3} are precisely those of $\UU_5^3$, which has five overlapping shifted copies of $\UU_5^2=\SS_5^2$, one in each corner,
just as we found four shifted copies of all of $\SS_4^2$ in $\SS_4^3$, in \cref{e:q34}.

We also have the facets described in \cref{thm:complete-list-of-faces}\eqref{i:WPa}. In this case, $|P|=2$ and $a\in P$, so we have $20$ simplices of the form $W_{\{i,j\},i}$, one for every choice of an ordered pair $(i,j)$ of distinct $i,j \in [q]$.  Each such facet is a $7$-dimensional simplex whose vertices consist of: one vertex of the form $2\be_i + \be_j$; six square-free vertices containing $\be_i$, and one square-free vertex consisting of $\be_a + \be_b + \be_c$ where $\{i,j,a,b,c\} = \{1,2,3,4,5\}.$  The realization of each of these facets is a polytope combinatorially equivalent to a $4$-cross-polytope, with minimal non-faces $\{2\be_i+\be_j,\ \be_a + \be_b + \be_c\}, \{\be_i+\be_j+\be_a,\ \be_i+\be_b+\be_c\}, \{\be_i+\be_j+\be_b,\ \be_i+\be_a+\be_c\}, \{\be_i+\be_j+\be_c,\ \be_i+\be_a+\be_b\}$.
To see this, note that the midpoint between any two vertices in a minimal non-face is the same: $\be_i + (\be_j + \be_a + \be_b + \be_c)/2$; and it is easy to see that the four direction vectors, each pointing from one vertex in a minimal non-edge to the other vertex in that minimal non-edge, are linearly independent.
\end{example}

\medskip

We are now ready to describe all minimal nonfaces of $\Sq$. As a consequence, \cref{t:minimalnonfaces3} shows that $\Srq$ is not a flag complex for $r \geq 3,\ q \geq 5$.

\begin{theorem}[{\bf Minimal nonfaces of $\Sq$}] \label{t:minimalnonfaces3}
     A non-empty subset  $\tau$ of $\N_q^3$ is a minimal nonface of $\Sq$ if and only if $\tau$ is of one of the following forms 
\begin{enumerate}     
\item\label{i:nonface-edges} {\bf (edges)} for distinct indices $a,b,c,d$ in $[q]$
$$
\begin{array}{lll}
\{3\be_a, 3\be_b\} &&\{2\be_a+\be_b, 2\be_b+\be_c\}\\    
\{3\be_a, \be_a + 2 \be_b\} && \{2\be_a+\be_b, \be_b+2\be_c\}\\     
\{3\be_a, 2\be_b+\be_c\} &&\{2\be_a+\be_b, 2\be_c+\be_d\}\\     
\{3\be_a, \be_a+\be_b+\be_c\}&& \{2\be_a+\be_b,\be_b+\be_c+\be_d\}\\   
\{3\be_a, \be_b+\be_c+\be_d\} &&\\
\end{array}
$$
\item\label{i:nonface-triangles} {\bf (triangles)} for distinct $a,b,c,d,e \in [q]$
$$\{2\be_a + \be_b, \be_a + 2\be_b, \be_c + \be_d + \be_e\}.$$
\end{enumerate}
\end{theorem}  

\begin{proof}
\eqref{i:nonface-edges}: The listed edges are precisely the edges that are not of the forms described in \cref{l:types}.  

\eqref{i:nonface-triangles}: $(\Longleftarrow)$ Assume, without loss of generality,  
$\tau=\{2\be_1 + \be_2, \be_1 + 2\be_2, 
\be_3 + \be_4 + \be_5\}$. The conclusion follows from \eqref{e:nontriangle}. 

$(\Longrightarrow)$ Now assume $|\tau|>2$.
Since $\tau$ is a minimal nonface with $|\tau|>2$, every pair of vectors in $\tau$ must form an edge of $\Sq$. In other words,
\begin{equation}\label{e:tau-edge}
\{\bu,\bv\}\in \Sq \qif \bu,\bv \in \tau.
\end{equation}

\begin{Claim1}
 $\tau$ cannot consist of only square-free vectors. 
\end{Claim1}
Claim~1 follows directly from \cref{t:Urq-Srq}.

\begin{Claim2} $3\be_i \notin \tau$ for any $i \in [q]$
\end{Claim2}
    Assume without loss of generality that $3\be_1 \in \tau$. Then by \eqref{e:tau-edge},  $\{3\be_1, \bv\}$ is an edge of $\Sq$, which by \cref{l:types} is equivalent to $\bv = 2\be_1 + \be_j$. 
         Therefore  every other vector in $\tau$ is of the form $2\be_1 + \be_j$. This implies $\tau$ is contained in the facet $\U_{2\be_1}$, and thus $\tau$ is a face of $\Sq$, a contradiction.  This settles Claim~2.

By Claim~1 and Claim~2 we can assume, without loss of generality, that $$2\be_1+\be_2 \in \tau.$$
Set
    \begin{equation}
    \label{e:TheP} 
    P=\{k\in [q]\smallsetminus \{1\}\colon 2\be_1+\be_k\in \tau\}\cup\{1\}\,.
    \end{equation}
By \cref{l:types} and in view of~\eqref{i:nonface-edges}, we have that if $\bb\in \tau$ is not of the form listed in \eqref{e:TheP}, then it can have one of the following forms. 

\begin{itemize}
\item $\bb=\be_1+\be_u+\be_v$ for distinct $1,u,v$.

\item $\bb=\be_u+\be_v+\be_w$ for distinct $1,u,v,w$.
 By \cref{l:types}, $\{2\be_1+\be_l,\be_u+\be_v+\be_w\}\in \Sq$ if and only if $l \notin\{u,v,w\}$, and we see that 
\begin{equation}
\label{e:uvw}
\bb=\be_u+\be_v+\be_w\in \tau \implies u,v,w\notin P.
\end{equation}
\item  $\bb=\be_1+2\be_2$.
 \end{itemize}
    
\begin{Claim3}
$\be_1+2\be_2 \in \tau$. 
\end{Claim3}
 
If $\be_1+2\be_2\notin \tau$, then we use \eqref{e:uvw} to see that $\tau\subseteq W_{P,1}$, and \cref{p:weird} implies $\tau\in \Sq$, a contradiction.

\begin{Claim4}
$\tau\smallsetminus \{2\be_1+\be_2, 2\be_2+\be_1\}\subseteq \mathcal U_q^3$. 
\end{Claim4}
By~\eqref{e:tau-edge}  $\{\be_1+2\be_2,2\be_1+\be_k\}\in \Sq$ for all $k\in P$. However,  \cref{l:types} implies  $P=\{1,2\}$ and this gives $\tau\smallsetminus \{2\be_1+\be_2, 2\be_2+\be_1\}\subseteq \mathcal U_q^3$. This finishes the proof of Claim~4.

\begin{Claim5} If $\bb \in 
\tau\smallsetminus \{2\be_1+\be_2, 2\be_2+\be_1\}$, then $\bb=\be_u+\be_v+\be_w$ where $1,2,u,v,w$ are distinct. 
\end{Claim5}

By our previous arguments, $\bb$ must be square-free. Assume 
$\bb=\be_u+\be_v+\be_w$ where $u<v<w$ and $u,v,w \notin P.$ If $u=1$, then
by~\eqref{i:nonface-edges} we know $\{\bb, 2\be_2+\be_1\}\notin \Sq$, contradicting the minimality of $\tau$. So $u>1$.

By~\eqref{e:uvw} and the fact that $P=\{1,2\}$ established in Claim~4,
$u>2$. This finishes our proof of Claim~5.

Now choose $\bb=\be_u+\be_v+\be_w\in \tau$ with $2<u<v<w$, and note that $\sigma=\{2\be_1+\be_2, \be_1+2\be_2, \bb\}$ is a nonface of $\Sq$ by the reverse implication that we just proved. Since $\tau$ is a minimal nonface of $\Sq$ containing $\sigma$, we conclude that $\tau=\sigma$. 
\end{proof}

\section{A general Morse Theorem}\label{s:morse}

Recall that our goal is to show the subcomplex $\Srq$ of $\Trq$ supports a free resolution of $\Erq$. 
To achieve this goal, we will use a homogenized version of Morse matchings due to Batzies and Welker~\cite{BW}. In this section we develop a general matching strategy that uses the minimal nonfaces of a simplicial complex. We will then use this strategy on the minimal nonfaces of $\Sq$ that were described in previous sections. Specifically, \cref{p:nice-Morse} can be used to show that under certain conditions, all the nonfaces of $\Srq$ can be eliminated from $\Trq$ via a homogeneous acyclic matching, and as a  result the remaining faces, which form  $\Srq$, support a resolution of $\Erq$.

We start by setting up the notation and language for discrete Morse theory, which was first developed by Forman~\cite{For} in a topological setting, and later adapted to the study of free resolutions of monomial ideals by Batzies and Welker~\cite{BW} by using a translation of Chari~\cite{Char}.

Let $V$ be a finite set, and $Y \subseteq 2^V$ a  set of subsets of $V$. We call the elements of $Y$ {\bf cells}. A cell of cardinality $n$ is called an $(n-1)${\bf -cell}. 
We define the directed graph $G_Y$ as a graph with vertex  set $\{\sigma \mid \sigma \in Y\}$ and
with directed edges $E_Y$ consisting of $$\sigma\to \sigma' \qwith
\sigma'\subseteq \sigma \qand |\sigma|=|\sigma'|+1.$$

A {\bf matching} of $G_Y$ is a set $A\subseteq E_Y$ of edges of $G_Y$
with the property that each cell of $Y$ occurs in at most one edge of
$A$.  Given a matching $A$, let the directed graph $G_Y^A$ be the same graph as $G_Y$ except that the direction of the arrows in $A$ are reversed.
The matching $A$ is said to be {\bf acyclic} if $G^A_Y$ contains no directed cycles. The cells of $Y$ that do not appear in the edges of the matching are called $A${\bf -critical cells}.

Given an ideal $I$, minimally generated by $t$ monomials, we consider $Y$ to be labeled  Taylor complex $\TT$ of $I$. Recall that the label of $\sigma\in \TT$ is denoted $\m_\sigma$, and it is equal to the lcm of the monomial labels of the vertices in $\sigma$. In this context, we say that a matching $A$ on $G_Y$ is {\bf homogeneous} if $\m_{\sigma}=\m_{\sigma'}$ for all directed edges $\sigma\to \sigma'\in A$. The following special case of a result of Batzies and Welker \cite[Proposition~1.2, Theorem~1.3]{BW} translates information given by a homogeneous acyclic matching on the Taylor complex of an ideal into a CW-complex supporting a free resolution of the ideal. 

\begin{theorem}[{\bf Resolutions from acyclic matchings}, see \cite{BW}] 
\label{t:BW}Let $I$ be a monomial ideal and let $\TT$ be the Taylor complex of $I$. If $A$ is a homogeneous acyclic matching on $G_\TT$, then there is a CW-complex $\mathcal{X}_A$ supporting a free resolution of $I$ whose $i$-cells are in one to one correspondence with the $A$-critical $i$-cells of the Taylor complex of $I$.
\end{theorem}

\begin{remark}
\label{r:BW}
If the set $\Delta$ of $A$-critical cells of $\TT$ is a simplicial complex, then $\Delta$ itself supports a free resolution of $I$. 
Indeed, in this case, it can be seen that the differential of the CW-complex $\mathcal {X}_A$ described in \cite[Lemma~7.7]{BW} coincides with the differential of the homogenized cell complex of $\Delta$, up to the identifying the cells of the two complexes; see also \cite[Proposition~5.2]{CK24}.
\end{remark}

\begin{lemma}[{\bf Cluster Lemma} {\cite[Lemma~4.2]{Jo}}]\label{clusterlemma} 
Let $V$ be a set and $Y\subseteq 2^V$, $(Q,\sq)$ a poset and $\{Y_q\}_{q\in Q}$ a partition
of $Y$ such that
$$
\text{If $\sigma\in Y_q$ and $\sigma'\in Y_{q'}$ satisfy $\sigma\supseteq \sigma'$, then $q\sq q'$.}
$$
Let $A_q$ be an acyclic matching on $G_{Y_q}$ for each $q$. Then $A\coloneqq \bigcup_{q\in Q}A_q$ is an acyclic matching on $G_{Y}$. 
\end{lemma}

\begin{lemma}{\cite[Lemmas~3.2, 3.3]{Morse}} \label{matching-lemma}
Let $Y\subseteq 2^V$. 
\begin{enumerate}[\quad\rm(1)] 
\item\label{i:matchingY} If  $Y \subseteq Y'\subseteq 2^V$ and  $A\subseteq E_Y$, then $A$ is an acyclic matching of $G_{Y'}$ if and only if $A$ is an acyclic matching of $G_Y$. 
\item\label{i:matchingA} If $A$ is an acyclic matching of $G_Y$ and $A'\subseteq A$, then $A'$ is an acyclic matching of $G_Y$ as well.  
\item\label{i:matchingv}  If $v\in V$, then 
\begin{equation*}
A_Y^{v} = \big\{\sigma\to \sigma'\in E_Y\st v\in \sigma   {\text{ and }} \sigma'=\sigma\smallsetminus \{v\}\big\}
\end{equation*}
is an acyclic matching on $G_Y$. 
\end{enumerate}
\end{lemma}

\begin{construction}[{\bf A partition of $Y$}]\label{c:nice}
    Let $V$ be a finite set, and $N \subseteq Y \subseteq 2^V$
be such that for every $\gamma\in Y$ there exists $\sigma\in N$ such that $\sigma\subseteq \gamma$, and $N$ is given a total order
$\sq$. 

For each $\sigma\in N$, define
\begin{equation}
\label{eq:nice}
Y_\sq(\sigma)=\{\gamma\in Y\colon \sigma= \max_\sq\{\tau\in N\colon \tau\subseteq \gamma\}\}.
\end{equation}
In other words, $Y_\sq(\sigma)$ consists of all cells $\gamma$ of $Y$ for which $\sigma$ is the largest cell of $N$ that is contained in $\gamma$. 

Observe that $\{Y_\sq(\sigma)\}_{\sigma\in N}$ is a partition of $Y$ with the property that for $\sigma,\sigma'\in N$, if $\gamma\in Y_\sq(\sigma)$ and $\gamma'\in Y_\sq(\sigma')$, then $\gamma\supseteq \gamma'$ implies $\sigma\sq \sigma'$. 
\end{construction}

\begin{proposition}[{\bf A Morse matching induced by a special  order}]\label{p:nice-Morse}
 Let $V$ be a finite set, and $N \subseteq Y \subseteq 2^V$
be such that for every $\gamma\in Y$ there exists $\sigma\in N$ such that $\sigma\subseteq \gamma$, and $N$ is given a total order
$\sq$. 
With notation as in \eqref{eq:nice}, assume that $$\omega\colon N\to V$$ is a function such that
\begin{enumerate}
\item\label{i:nice-Morse-sigma} $\omega(\sigma)\notin \sigma$  for each $\sigma \in N$;
\item\label{i:nice-Morse-gamma} if  $\gamma\in Y_\sq(\sigma)$  then $\gamma\cup \omega(\sigma)\in Y_\sq(\sigma)$.
\end{enumerate}
For each $\sigma\in N$, construct a set of directed edges 
\[
A_\sigma=\{\gamma\to \gamma\smallsetminus \{\omega(\sigma)\} \st  \gamma\in Y_\sq(\sigma), \quad \omega(\sigma)\in \gamma\}.
\]
Then $A_\sigma$ is an acyclic matching of $G_{Y_\sq(\sigma)}$, and $A\coloneqq \bigcup_{\sigma\in N}A_\sigma$ is an acyclic matching of $G_{2^V}$ such that the set of $A$-critical cells equal to $2^V\smallsetminus Y$. 
\end{proposition}

\begin{proof}
   Let $\sigma\in N$. First, observe that if $\gamma\in Y_\sq(\sigma)$ and $\omega(\sigma)\in \gamma$, then $\gamma \setminus \{\omega(\sigma)\} \in Y_\sq(\sigma)$ as well. Indeed, since $\omega(\sigma) \notin \sigma$, we have $\sigma\subseteq \gamma\setminus \{\omega(\sigma)\}$.  If $\tau \in N$ and $\tau \subseteq \gamma \setminus \{\omega(\sigma\}$, then $\tau \subseteq \gamma$ trivially. Thus $\sigma \sq \tau$, and hence $\gamma \setminus \{\omega(\sigma)\} \in Y_\sq(\sigma)$. 
   
   The fact that $A_\sigma$ is an acyclic matching is then a consequence of \cref{matching-lemma}\eqref{i:matchingv}. Observe  that all cells $\gamma\in Y_\sq(\sigma)$ appear in an edge of $A_\sigma$. Indeed, if $\omega(\sigma)\in \gamma$ then $\gamma\to \gamma\smallsetminus \{\omega(\sigma)\}$ is  in $A_\sigma$. If $\omega(\sigma)\notin \gamma$, then $\gamma\cup \omega(\sigma)\in Y_\sq(\sigma)$ by~\eqref{i:nice-Morse-gamma} and $\gamma\cup \omega(\sigma)\to \gamma$ is in $A_\sigma$. Thus, the set of $A_\sigma$-critical cells is empty.
   
  Then, \cref{clusterlemma} gives that  $A=\bigcup_{\sigma\in N}A_\sigma$ is also an acyclic matching on $G_Y$ and its set of critical cells is empty. Using \cref{matching-lemma}\eqref{i:matchingY}, we conclude that $A$ is also an acyclic matching on $G_{2^V}$, with set of critical cells equal to $2^V\smallsetminus Y$. 
\end{proof}

We apply~\cref{p:nice-Morse} below with $V$ being the set of monomial labels of the vertices of the Taylor complex corresponding to a particular ideal $I$, recalling that we identify vertices with their monomial labels.

\begin{theorem}[{\bf A Morse resolution from the minimal nonfaces of a simplicial complex}]\label{t:nice-Morse-resolution}
Let $\TT$ be the Taylor complex of a monomial ideal $I$ minimally generated by monomials $\m_1,\ldots,\m_t$ , and let $\Delta$ be a subcomplex of $\TT$. Let $N$ denote the set of minimal (with respect to inclusion) nonfaces of $\Delta$ equipped with a total order $\sq$, and let $Y$ denote the set of all nonfaces of $\Delta$ partitioned by $Y=\bigcup_{\sigma \in N} Y_\sq(\sigma)$ as in \eqref{eq:nice}.

Assume that $$\omega\colon N\to \{\m_1,\ldots,\m_t\}$$ is a function such that for each $\sigma \in N$
\begin{enumerate}
\item\label{i:resolution-sigma} $\omega(\sigma)\notin \sigma$;
\item\label{i:resolution-lcm} $\omega(\sigma)\mid \lcm(\sigma)$;
\item\label{i:resolution-tau} if  $\tau\in Y_\sq(\sigma)$  then $\tau\cup \omega(\sigma)\in Y_\sq(\sigma)$.
\end{enumerate}
Then $\Delta$ supports a free resolution of $I$. 
\end{theorem}

\begin{proof} By \cref{p:nice-Morse} there is an acyclic matching $A$ on the faces  of $\Trq$ whose critical cells  are exactly the  faces of $\Delta$. The condition $\omega(\sigma)\mid \lcm(\sigma)$ 
ensures that this matching is homogeneous. It now follows from \cref{t:BW}, in view of \cref{r:BW}, that $\Delta$ supports a free resolution of $I$. \end{proof}

The next result offers a path towards proving condition~\eqref{i:resolution-tau} in \cref{t:nice-Morse-resolution}. In our discussions below, for a finite set $V$ and a total order $\sq$ of $2^V$, we will use the total order induced by $\sq$ on $V$ itself, defined by 
\begin{equation}\label{e:induced-order}
v \sq u
\iff
\{v\} \sq \{u\} \qfor v,u\in V.
\end{equation}
Also, by $v \succ w$ we mean $v \sq w$ and $v \neq w$, and for  $\sigma \in 2^V$, we denote  $\max_{\sq}(\sigma)$ by $\max(\sigma)$.

\begin{lemma}
\label{l:max}
Let $V$ be a finite set whose power set  $2^V$ is endowed with a total order $\sq$ such that   
\begin{equation}
\label{e:order}
\sigma \sq \tau \Longrightarrow \{ \max(\sigma)\} \sq \{\max(\tau)\} \qquad \text{for $\sigma, \tau \in 2^V$.  }
\end{equation}
Let $\emptyset\ne N \subseteq 2^V$ and  $$Y:=\{\tau \in 2^V \st \sigma \subseteq \tau \qforsome \sigma \in N\}.$$ 
Consider a function $\omega\colon N\to V$, and let
$\sigma\in N$ satisfying the following property:  
\begin{enumerate}
    \item[(\#)] if $\alpha\in N$ is such that $\omega(\sigma)\in \alpha$ and $\alpha\succ \sigma$, then $\max(\alpha)\ne \max(\sigma)$, and there exist $\sigma' \subseteq \sigma$ and $\sigma''\subseteq \alpha\smallsetminus\{\omega(\sigma)\}$ such that 
$$
\sigma'\cup \sigma''\in N \qand \max(\alpha)\in \sigma''\,.
$$
\end{enumerate}
If $\tau\in Y_\sq(\sigma)$, then $\tau\cup\{\omega(\sigma)\}\in  Y_\sq(\sigma)$. 
\end{lemma}

\begin{proof}
Let $\tau\in Y_\sq(\sigma)$. If $\omega(\sigma) \in \sigma$ then since $\sigma \subseteq \tau$ we have $\tau \cup \{\omega(\sigma)\} = \tau$, so there is nothing to prove. Assume $\omega(\sigma) \notin \sigma.$ Since $\sigma\subseteq \tau$, we also have $\sigma\subseteq \tau\cup\{\omega(\sigma)\}$, hence $\tau\cup\{\omega(\sigma)\}\in Y$. Then  $\tau\cup\{\omega(\sigma)\} \in Y_\sq(\alpha)$ for some $\alpha \in N$. Assume $\tau\cup\{\omega(\sigma)\}\notin  Y_\sq(\sigma)$, therefore 
$$\alpha \succ \sigma \qand \omega(\sigma)\in \alpha.$$

By property (\#) in the hypothesis, we have 
 \begin{equation*}\gamma:=\sigma'\cup\sigma''\in N \qforsome \sigma'\subseteq \sigma, \quad  \sigma''\subseteq\alpha\smallsetminus\{\omega(\sigma)\} \qwith \max(\alpha)\in \sigma''.
 \end{equation*}

Since $\alpha\subseteq \tau \cup \{\omega(\sigma)\}$, we have 
$\alpha\smallsetminus\{\omega(\sigma)\}\subseteq \tau$. Further, since $\sigma'\subseteq \sigma\subseteq \tau$ and $\sigma''\subseteq \alpha\smallsetminus\{\omega(\sigma)\}$, we conclude 
\begin{equation}
\label{e:D}
    \gamma\subseteq \tau.
\end{equation} 

But $\tau\in Y_\sq(\sigma)$ and $\gamma\in N$, and from \eqref{e:D} we conclude $\sigma \sq \gamma$, and hence 
\begin{equation}\label{e:E}\max(\sigma)\succcurlyeq \max(\gamma).\end{equation} 

 Since $\alpha\succ\sigma$ and $\sigma'\subseteq \sigma$, 
 we must have $\max(\alpha)\sq\max(\sigma)\sq \max(\sigma')$. Also $\max(\alpha)\in\sigma''$ and $\sigma''\subseteq \alpha$  implies $\max(\alpha)=\max(\sigma'')$.  In particular, since $\gamma=\sigma'\cup\sigma''$, we have
\begin{equation}
\label{e:F}
\max(\gamma)=\max(\alpha)\succcurlyeq\max(\sigma).
\end{equation}
 
Then \eqref{e:E} and \eqref{e:F} together imply that 
$$
\max(\alpha)=\max(\gamma)= \max(\sigma),
$$
contradicting (\#). This ends our argument.
\end{proof}

Next we define an order in the case of interest, with $V=\mathcal N^r_q$ for integers $r,q$, for which $2^V$ is the Taylor complex $\Trq$.  We first define an order on $V$, and we use this order to induce an order on $2^V$. 

If $\ba=(a_1,a_2,\ldots, a_q) \in \Nrq$, then we let $\lambda(\ba)$ denote the vector $\ba$ written in partition form, obtained by ordering the integers $a_i$ in non-increasing order. For example, if $\ba=(2,3,3,1)$, then $\lambda(\ba)=(3,3,2,1)$. 

\begin{definition}[\bf A total order on the cells of $\Trq$] \label{d:total-order} We define a total order  $\sq$ on the cells of $\Trq$ starting from a lexicographic order on the partition forms of their vertices.

\begin{enumerate}  
\item \label{i:V-order} For 
$\ba=(a_1,\cdots,a_q), \ \bb = (b_1,\cdots,b_q) \in \Nrq$,  
\begin{itemize}
    
\item The lexicographic order $\geq _\lex$ is defined as follows: $\ba \geq_\lex \bb$ means  $\ba=\bb$ or $\ba \neq \bb$ and $\ba >_\lex \bb$ where  
$$\ba >_\lex \bb 
 \iff
a_1=b_1, a_2=b_2, \ldots , a_k=b_k, \qand a_{k+1}> b_{k+1}
\qforsome  k \geq 1.
 $$
\item
The total order $\sq$ is defined as follows: 
$$\ba \sq  \bb \qif 
 \lambda(\ba) >_\lex \lambda(\bb), \qor
 \lambda(\ba)=\lambda(\bb)        \qand 
 \ba \geq_\lex \bb.$$
\end{itemize}

By $\ba \succ \bb$ we mean $\ba \sq \bb$ and $\ba \neq \bb$.
\item\label{i:E-order}   We order the minimal generators of $\Erq$ as follows:  given $\pmea$, $\pmeb$ two minimal generators of $\Erq$, we write 
$$\pmea \sq \pmeb  \qif 
    \ba \sq \bb.$$  

\item\label{i:T-order}   For $\sigma=\{\ba_1,\ldots,\ba_s\}$  and $\tau=\{\bb_1,\ldots,\bb_t\}$ in $\Trq$   such that 
$$\ba_1 \sq \ba_2 \sq \cdots \sq \ba_s
\qand 
\bb_1 \sq \bb_2 \sq \cdots \sq \bb_t,$$
we say $\sigma \sq  \tau$ if $\sigma=\tau$ or for some $k \geq 1$,
$$\ba_1=\bb_1, \ba_2=\bb_2, \ldots , \ba_k=\bb_k, 
\qand s>k=t \qor \ba_{k+1} \succ \bb_{k+1}.
$$
\end{enumerate}
\end{definition}

In \cref{d:total-order} with $V=\Nrq$ as the vertex set of $\Trq$, we first defined  an order on $V$, and then we used this order to define an order on  $2^V=\Trq$. Observe that the order on $\Trq$ in~\eqref{i:T-order}  above induces an order on $\Nrq$ in the sense of \eqref{e:induced-order}, which  coincides with the one in \eqref{i:V-order}.

The order on the cells of $\Trq$ introduced in \cref{d:total-order} is the one we will use to show that ${\E_q}^3$ is Scarf in \cref{s:Eq3-Scarf}.

\section{${\E_q}^3$ is Scarf}\label{s:Eq3-Scarf}

We are now ready to show that the Scarf complex $\SS_q^3$ supports a resolution of ${\E_q}^3$ by eliminating all other faces of the Taylor complex $\TT_q^3$ using the Morse matching from \cref{t:nice-Morse-resolution}.

\begin{construction}[{\bf A function $\omega\colon N\to \N_q^3$}]
\label{list}
In \cref{t:minimalnonfaces3} we found all elements of the set $N$ of  minimal nonfaces of $\Sq$. For the convenience of the reader, we list the witness $\omega(\sigma)$ for each non-edge $\sigma$ of $\Sq$ in the table below (these choices come from~\cite[Lemma 6.1, Remark 6.5]{extremal}). 
For a subset $A\subseteq [q]$ and $i\in [q]$ let
$$t_i=|A \cap\{i\}|.$$ Then the table below helps the reader verify  that for each $\sigma=\{\ba,\bb\}$, $$\omega(\sigma)\cdot \be_A \leq \max\{\ba\cdot \be_A, \bb \cdot \be_A\}.$$
$$
\begin{array}{l|l|l|l|l}
\sigma=\{\ba,\bb\} 
& \ba\cdot \be_A &
\bb\cdot \be_A  & \omega(\sigma) & \omega(\sigma)\cdot \be_A \\
\hline
\{3\be_a, 3\be_b\} &
3t_a & 3t_b & 
2\be_a+\be_b &2t_a+t_b\\
\{3\be_a, \be_a + 2 \be_b\} &
3t_a & t_a + 2t_b& 
2\be_a+\be_b &2t_a+t_b\\
\{3\be_a, 2\be_b+\be_c\} &
3t_a & 2t_b+t_c&
2\be_a+\be_b &2t_a+t_b\\
\{3\be_a, \be_a+\be_b+\be_c\}&
3t_a & t_a+t_b+t_c&
2\be_a+\be_b &2t_a+t_b\\
\{3\be_a, \be_b+\be_c+\be_d\} &3t_a & t_b+t_c+t_d&
2\be_a+\be_b &2t_a+t_b\\
\{2\be_a+\be_b, 2\be_b+\be_c\}& 
2t_a+t_b & 2t_b+t_c&
\be_a+\be_b+\be_c&t_a+t_b+t_c\\    
\{2\be_a+\be_b, \be_b+2\be_c\} &
2t_a+t_b & t_b+2t_c&
\be_a+\be_b+\be_c&t_a+t_b+t_c\\  
\{2\be_a+\be_b, 2\be_c+\be_d\}&
2t_a+t_b & 2t_c+t_d&
\be_a+\be_b+\be_c&t_a+t_b+t_c\\  
\{2\be_a+\be_b,\be_b+\be_c+\be_d\}&
2t_a+t_b & t_b+t_c+t_d&
\be_a+\be_b+\be_c&t_a+t_b+t_c
\end{array}
$$

We list them below, using our $\epsilon$ notation rather than vectors as our next arguments will be easier in this language, together with a choice of $\omega(\sigma)\in {\E_q}^3$  such that $\omega(\sigma)\notin \sigma$ and  $\omega(\sigma)\mid \lcm(\sigma)$ for all $\sigma\in N$.  The choice of $b$  recorded below comes from the proof of \cref{t:minimalnonfaces3} for the triangles. Note that when $q< 5$ some of these faces may not exist. (For example, one needs $q\ge 5$ in order to have faces of Type~III.)

$$\begin{array}{c|c|c}
\mbox{Type} & \sigma \in N & \omega(\sigma)\\
\hline
\begin{array}{l}
\mbox{\small I} \\ 
\mbox{\small II} \\
\mbox{\small III} \\
\end{array}
&
\begin{array}{ll}
 \{\e_i^3, \bn\}, & 
 \bn \not \in \{ \e_i^2\e_j \st j \in [q]\}, \e_i^3\succ \bn  \\ 
 \{\e_i^2\e_j,\e_i\e_j^2, \bn\},  & i<j, \quad \bn \in \U_{[q]\smallsetminus\{i,j\}}^3  \\
 \{\e_i^2\e_j, \bn\},& i\ne j, \quad \bn\in \mathcal Q_{i,j} \\
\end{array}
&
\begin{array}{ll} 
\e_i^2\e_j, &
j=\min\left(  \supp(\bn) \smallsetminus \{i\}\right)\\
\e_i\e_j\e_k, &
k=\min \left( \supp(\bn) \right)\\
\e_i\e_j\e_k, &
k=\min \left( \supp(\bn)\smallsetminus \{i,j\} \right)
\end{array}
 \end{array}
$$ 
where 
 \[
 \mathcal Q_{i,j}=\left\{\e_a^2\e_b\colon a\ne b,a\ne i, \e_i^2\e_j\succ \e_a^2\e_b\ne \e_i\e_j^2 \right\}\cup \mathcal \e_j\mathcal U_{[q]\smallsetminus\{i,j\}}^2\,.
 \]
\end{construction}

Below we use the total order $\sq$ defined in \cref{d:total-order} and use the notation in \eqref{eq:nice} with this order. 

 \begin{lemma}
 \label{l:r=3}
In the setup of \cref{list}, let $\sigma \in N$ and 
$\tau\in Y_\sq(\sigma)$. Then  $\tau\cup\{\omega(\sigma)\}\in Y_\sq(\sigma)$. 
 \end{lemma}

 \begin{proof} To prove the statement, we will prove~(\#)  in \cref{l:max}. Let $\alpha\in N$ with $\omega(\sigma)\in \alpha$ and $\alpha\succ\sigma$. To verify (\#) of \cref{l:max},  it suffices to show \begin{equation}
 \label{e:(2)}
 \max(\sigma)\ne\max(\alpha)\ne \omega(\sigma) \qquad\text{and}\qquad \{\max(\alpha), \m\}\in N \quad\text{for some  $\m\in \sigma$.}
 \end{equation}
 Indeed, if \eqref{e:(2)} holds, then (\#) holds with $\sigma'=\{\m\}$ and $\sigma''=\{\max(\alpha)\}$.

Since $\sigma, \alpha \in N$, they are both  of the types introduced in \cref{list}.

Assume $\sigma=\{\e_i^3,\bn\}$ is of Type~I.  Then $\omega(\sigma)=\e_i^2\e_j$ and since $\alpha\succ \sigma$, $\alpha$ is of  Type~I as well. Therefore $\alpha=\{e_k^3, \e_i^2\e_j\}$ for some $k\in[q]$. If $k=i$, then $\alpha=\{\e_i^3,\e_i^2\e_j\}\notin N$, a contradiction. Hence $k\ne i$, implying $\max(\alpha)=\e_k^3\ne\e_i^3=\max(\sigma)$ and also $\max(\alpha)\ne \omega(\sigma)$, and we  know  $\{\e_k^3,\e_i^3\}$ is a Type~I face of  $N$, 
and thus \eqref{e:(2)} holds with $\m=\e_i^3$. 

Assume $\sigma=\{\e_i^2\e_j,\e_i\e_j^2, \bn\}$ or $\sigma=\{\e_i^2\e_j,\bn\}$ is of Type~II or III, hence $\omega(\sigma)=\e_i\e_j\e_k$ is square-free. Assume  $\alpha$ is of Type~I, then $\max(\alpha)=\e_l^3$ for some $l\in [q]$, implying $\max(\alpha)\ne \max(\sigma)$ and $\max(\alpha)\ne \omega(\sigma)$. If $l\ne i$, then $\{\e_l^3,\e_i^2\e_j\}$ is a Type~I face of $N$,  which proves  \eqref{e:(2)} with $\m=\e_i^2\e_j$. If $l=i$, then $\{\e_i^3,\bn\}$ is a Type~I face of $N$, thus  \eqref{e:(2)} holds with $\m=\bn$.   

Assume $\alpha$ is of Type~II. Since $\omega(\sigma)=\e_i\e_j\e_k$ and $\omega(\sigma)\in \alpha$, we must have 
$$
\alpha=\{\e_a^2\e_b,\e_a\e_b^2, \e_i\e_j\e_k\}\qwhere \{a,b\}\cap \{i,j,k\}=\emptyset \qand a<b.
$$
Then $\max(\alpha)=\e_a^2\e_b\ne \e_i^2\e_j= \max(\sigma)$ and $\max(\alpha)\ne \omega(\sigma)$. 
Observe that 
$\{\e_a^2 \e_b, \e_i^2 \e_j\}$ is a Type~III face of $N$, hence  \eqref{e:(2)} holds with $\m=\e_i^2\e_j$. 

Assume $\alpha$ is of Type~III. Since $\omega(\sigma)=\e_i\e_j\e_k$ and $\omega(\sigma)\in \alpha$, we must have 

$$\alpha=\{\e_a^2\e_b, \e_i\e_j\e_k\} \qwhere
\{a,b\}\cap \{i,j,k\}=\{b\} \qand b\ne a.
$$
In particular, $a\ne i$, or else $a\in \{a,b\}\cap \{i,j,k\}$, implying $a=b$, a contradiction. We have then  $\max(\alpha)=\e_a^2\e_b\ne \e_i^2\e_j=\max(\sigma)$ and $\max(\alpha)\ne \omega(\sigma)$. Also, 
$\{\e_a^2 \e_b, \e_i^2 \e_j\}$ is a face of Type~III, hence \eqref{e:(2)} holds with $\m=\e_i^2\e_j$. 
\end{proof}

 We now have the tools in place to show the main result of the paper. When combined with results from prior sections, the theorem below provides a complete combinatorial description of a simplicial complex that supports a minimal free resolution of ${\E_q}^3$.

\begin{theorem}[{\bf ${\E_q}^3$ is Scarf}]
\label{t:r=3}
For any $q\ge 1$, $\mathbb S_q^3$ supports a minimal free resolution of ${\E_q}^3$. 
\end{theorem}

\begin{proof} Let  $N$ be the set of minimal non-faces of the $\Sq$ and let $\omega\colon N \to \N^3_q$ be defined  as in \cref{list}.
    We use \cref{t:nice-Morse-resolution} with $\Delta=\Sq$ and the order defined in \cref{d:total-order}. Conditions~\eqref{i:resolution-sigma} and~\eqref{i:resolution-lcm} of \cref{t:nice-Morse-resolution} are clear from \cref{list}, and Condition~\eqref{i:resolution-tau} is proved in \cref{l:r=3}. 
\end{proof}

\section{An upper bound for projective dimension and betti numbers of the third power of square-free monomial ideals}\label{s:bounds}\label{s:counting}

Now we are ready to go back to where we started: bounding betti vectors of powers of square-free monomial ideals by $f$-vectors of simplicial complexes. Recall that for a monomial ideal $I$ of projective dimension $p$ and a simplicial complex $\Delta$ of dimension $d$, the betti-vector and $f$-vector can be defined, respectively, as
$${\bbeta}(I)=(\beta_0(I),\beta_1(I),\ldots, \beta_p(I)) 
\qand
{\mathbf f}(\Delta)=(f_0(\Delta),f_1(\Delta),\ldots, f_{d}(\Delta)),
$$
where $f_i(\Delta)$ denotes the number of $i$-faces of $\Delta$.
The three simplicial complexes
we consider are  
\begin{equation}\label{e:complexes}
\Srq \subseteq \Lrq \subseteq \Trq.
\end{equation}

The complex $\SS^r_q=\Scarf({\E_q}^r)$ supports a minimal free resolution of ${\E_q}^r$ when $r \leq 2$ by~\cite{L2} and when $r=3$ by \cref{t:r=3} above. The complex $\Lrq$ was defined in~\cite{Lr}, and supports a free resolution of $I^r$ where $I$ is minimally generated by any $q$ square-free monomials. The complex $\Trq$ is the Taylor simplex of ${\E_q}^r$ 
with $\binom{q+r-1}{r}$ vertices.
As a result, we can write 
\begin{equation}\label{e:inequalities}
{\bbeta}(I^r) \leq 
{\bbeta}({\E_q}^r) =
{\mathbf f}(\Srq) \leq
{\mathbf f}(\Lrq) \leq
{\mathbf f}(\Trq)
\end{equation}
where $I$ is any square-free monomial ideal and $r \leq 3$. The first inequality holds for all $r$ by \cref{t:e-bound}~(\cite{Lr}). 
The equality is conjectured (\cref{c:conjecture}) to be true for all $r$. 
The  last two inequalities hold for all $r$ and are strict when $r > 1$ and $q > 2$. When $r=1$, the three complexes in~\eqref{e:complexes} are equal.
When $r\le 2$, $\SS_q^2 = \LL_q^2$ and the corresponding $f$-vector is already known:

\begin{theorem}[{See {\cite[Theorem 4.1]{L2}, \cite[Theorem 7.9]{Lr}}}]\label{t:bound}
 Let $I$ be an ideal of a polynomial ring  minimally generated by
  $q\geq 1$ square-free monomials. For any $i\ge 0$ we have
  $$
\beta_i(I) \leq \beta_i({\E_q}) = {\textstyle \binom{q}{i+1}}
\qand 
 \pd (I) \leq \pd(\E_q)= q-1
    $$    and 
$$
\beta_i(I^2) \leq \beta_i({\E_q}^2) = {\textstyle \binom{\binom{q}{2}}{i+1} + q\binom{q-1}{i} }
\qand 
 \pd (I^2) \leq \pd({\E_q}^2)= 
  \begin{cases}\textstyle{\binom{q}{2}}-1 & \qif q\geq 3\\
      q-1 & \qif  1\le q\le 2. 
      \end{cases}
    $$    
  \end{theorem}

An algebraic consequence of~\cref{t:r=3,t:e-bound} is that the number of faces of each dimension of $\Sq$ gives a sharp upper bound to the betti number of the third power of square-free monomial ideals.

\begin{theorem}[{\bf A sharp bound for the betti numbers of the third power of a square-free monomial ideal}]\label{t:sharpbound3}
    Let $I$ be a square-free monomial ideal generated by $q \geq 1$ generators. Then
    for any $i\geq 1$, we have
            \begin{multline}\label{eq:sharpbound}
             \beta_{i}(I^3) \leq
           \beta_{i}({\E_q}^3)=
\binom{\binom{q}{3}}{i + 1}  + q\binom{q-1}{i} + \binom{q}{2}\binom{q-2}{i-1} + \sum_{s=2}^{q-3} q\binom{q-1}{s-1} \binom{\binom{q-s}{3} + \binom{q-1}{2}}{i-s+2}\\
    + q\binom{q-1}{2}\binom{\binom{q-1}{2}}{i-q+4}
         + q(q-1)\binom{\binom{q-1}{2}}{i-q+3}
        + q\binom{\binom{q-1}{2}}{i-q+2}       
        \end{multline}
   and    
     $$
      \pd (I^3) \leq \pd ({\E_q}^3) = 
      \begin{cases}\textstyle{\binom{q}{3}}-1 & \qif q\geq 5\\
      \textstyle{\binom{q}{2}}-1 & \qif 3\le q\le 4\\
      q-1 &\qif 1\le q\le 2
      \end{cases}
    $$   
    Furthermore, if $q>5$ then $\beta_t(I^3)\le 1$ for $t=\pd({\E_q}^3)$. 
    
    All the inequalities turn to equalities when $I=\E_q$.
\end{theorem}

\begin{proof}
    By~\cref{t:bound}, it suffices to prove the equalities when $I={\E_q}^3$. In view of the fact that ${\E_q}^3$ is a  Scarf ideal with resolution supported on $\Sq$ (\cref{t:r=3}), the projective dimension $\pd({\E_q}^3)$ is equal  to the dimension  of $\Sq$, and the betti number $\beta_i({\E_q}^3)$ is equal to the number of $i$-faces of $\Sq$. Using the list of facets of $\Sq$ in \cref{thm:complete-list-of-faces}, we compute the sizes of the facets of $\Sq$ in the following table, for distinct $a, b \in [q]$  and $P \subseteq [q]$ with $2\leq |P|\leq q-3$.
$$  
\begin{array}{l|c|ll|l}
   \mbox{Facet} & q \geq & \mbox{Elements} && \mbox{Size} \\ 
   \hline
        \U_q^3 & 4 &\be_i+\be_j+\be_k & 1\le i<j<k\le q  & \binom{q}{3}\\
\U_{\be_a}^3&3&\be_a+\be_i+\be_j & 1 \leq i<j \leq q &  \binom{q}{2}  \\
\U_{2\be_a}^3 &1&2\be_a+\be_i & 1 \leq i\leq  q&q\\
\U_{\be_a+\be_b}^3&3&\be_a+\be_b+\be_i & i\in [q]  & q \\
W_{P,a} &5&  \mbox{ See \cref{p:weird}}  &&
         |P|-1  +\binom{q-1}{2}+ \binom{q-|P|}{3}             
    \end{array}
$$
 Observe that
       $$\begin{array}{ll}
        \binom{q}{3} &= \binom{q-1}{2} + \binom{q-1}{3} 
        = \binom{q-1}{2} +\binom{q-2}{3}  +\binom{q-2}{2} 
        =  \binom{q-1}{2} +\binom{q-2}{3}  + \frac{(q-2)(q-3)}{2}\\ 
        &\\
        &\geq \binom{q-1}{2} +\binom{q-|P|}{3}  + \frac{2|P|}{2}
        > \binom{q-1}{2} +\binom{q-|P|}{3}  + |P|-1.\\
    \end{array}$$
    When $q\geq 5$, we always have 
    \[{\textstyle
    \binom{q}{3}\geq \max ( \binom{q}{2}, q )
       }\]
    thus $\dim(\Sq)=\binom{q}{3}-1=\dim(\UU^3_q)$.
Moreover when $q\leq 4$ we have $\Sq=\UU^3_q$.
Hence 
$$
\pd({\E_q}^3)
= \dim(\UU^3_q)=
\begin{cases}\textstyle{\binom{q}{3}}-1 & \qif q\geq 5\\
      \textstyle{\binom{q}{2}}-1 & \qif 3\le q\le 4\\
      q-1 &\qif 1\le q\le 2.
      \end{cases}
$$
    
   We now proceed to count the number of $i$-faces of $\Sq$ for $i>0$.
   We first show that this number equals
\begin{multline}\label{eq:intermediate}
    \binom{\binom{q}{3}}{i+1} 
    + q\Bigg{(}\binom{\binom{q}{2}}{i + 1} - \binom{\binom{q-1}{2}}{i+1}\Bigg{)} 
    + q\binom{q-1}{i} 
    + \binom{q}{2}\binom{q-2}{i-1}\\ 
    + \sum_{s = 2}^{q - 3} q \textstyle{\binom{q-1}{s-1} (\binom{\binom{q-s}{3}+\binom{q-1}{2}}{i-s+2} - \binom{\binom{q-1}{2}}{i-s+2}) }
\end{multline}
   
\begin{itemize}
    \item The number of $i$-faces in the facet $\U_q^3$ is 
    $\binom{\binom{q}{3}}{i+1}$. This accounts for the first term in  \eqref{eq:intermediate}. 

    \item The number of $i$-faces in a facet of type $\U_{\be_a}^3$ is $\binom{\binom{q}{2}}{i+1}$ and there are $q$ such facets. Out of all the $i$-faces of a facet of this type, $\binom{\binom{q-1}{2}}{i+1}$ are also faces of $\U_q^3$ and hence have been counted before. Moreover, $\U_{\be_a}^3 \cap \U_{\be_b}^3 \subset \U_q^3$ for $a \neq b$. This takes care of the second term in~\eqref{eq:intermediate}
    
    \item We next count the number of $i$-faces in a facet of type $\U_{2\be_a}^3$ which were not previously counted as faces of facets of type $\U_q^3$ or $\U_{\be_a}^3$. Such a face $\sigma$  must contain a vector of the form $3\be_a$, otherwise it was already counted as a face of $\U_{\be_a}^3$. There are $\binom{q-1}{i}$ such faces. Since there are $q$ facets of type $\U_{3\be_a}^3$ and $\U_{3\be_a}^3 \cap \U_{3\be_b}^3 = \emptyset$ for $a \neq b$, this gives us the third term in~\eqref{eq:intermediate}. 
    
    \item We now count the number of $i$-faces in a facet of type $\U_{\be_a+\be_b}^3$ that have not been previously counted as faces of facets of types $\U_q^3$, $\U_{2\be_a}^3$ or $\U_{\be_a}^3$. For a fixed $a$ and $b$ such a face must contain both vectors $2\be_a+\be_b$ and $\be_a+2\be_b$. Thus there are exactly 
    \[{\textstyle
    \binom{q-2}{i-1}
    }\]
    such faces. Finally, there are $\binom{q}{2}$ facets of type $\U_{\be_a+\be_b}^3$, and moreover, if $\be_a+\be_b\ne \be_c+\be_d$, 
    \[
    \U_{\be_a+\be_b}^3\cap \U_{\be_c+\be_d}^3\subseteq \U_q^3\cup\bigcup_{t=1}^q\left(\U_{\be_t}\cup\U_{2\be_t} \right)\,.
    \]
This accounts for the fourth term in \eqref{eq:intermediate}. 

\item Finally, we count the number of $i$-faces $\sigma$ in a facet  of type $W_{P,a}$ that  have not been previously counted as faces of facets of the types $\U_q^3$, $\U_{2\be_a}^3$, $\U_{\be_a}^3$, or $\U_{\be_a+\be_b}^3$.

For any such $\sigma$, if $2\be_a+ \be_{b_1},\dots, 2\be_a+\be_{b_{s-1}}$ are all the nonsquare-free vectors in $\sigma$,  then there exists a unique set $P=\{a, b_1, \dots, b_{s-1}\}\subseteq [q]$ and a unique index $a\in P$ such that $\sigma$ can be written as a disjoint union
\begin{align}
\sigma&=\sigma_1\cup \sigma_2\cup \sigma_3\qwith \nonumber \\
\label{e:sigma1} \sigma_1&\coloneqq \{ 2\be_a+\be_b \st b\in P\smallsetminus \{a\}\}\\
\label{e:sigma2}\sigma_2&\subseteq S_2:= \{ \be_j+\be_k+\be_l \st 1\le j<k<l\le q,\,\, j,k,l \not\in P  \}\\
\label{e:sigma3} \sigma_3&\subseteq S_3:= \{ \be_a+\be_j+\be_k \st 1\le j<k\le q,\,\,  a\notin \{j,k\}\}.   
\end{align}
Since $\sigma$ is not a face of  $\U_{\be_a}^3$, we must have $|\sigma_2|\ge 1$.

Conversely, for any integer $s$ with 
$2\le s\le q-3$ and for any choice of
\begin{itemize}
    \item an index $a \in [q]$;
    \item a set $P$ satisfying $a \in P \subseteq [q]$ and $|P|=s$;
    \item a non-empty set $\sigma_2$ satisfying \eqref{e:sigma2} and a set $\sigma_3$ satisfying \eqref{e:sigma3}, such that $|\sigma_2 \cup \sigma_3|=i-s+2$;
\end{itemize}
    and with $\sigma_1$ defined as in \eqref{e:sigma1}, the set $\sigma:=\sigma_1\cup \sigma_2\cup \sigma_3$ is an $i$-face of the facet $W_{P,a}$, and is not a face of any facet of the types $\U_q^3$, $\U_{2\be_a}^3$, $\U_{\be_a}^3$, or $\U_{\be_a+\be_b}$. 
    There are $q$ choices for $a$ and $\binom{q-1}{s-1}$ choices for $P$ containing $a$.  There are $\binom{\binom{q-s}{3}+\binom{q-1}{2}}{i-s+2} - \binom{\binom{q-1}{2}}{i-s+2}$ choices for $\sigma_2 \cup \sigma_3$ because $|S_2| = \binom{q-s}{3}$ and $|S_3|=\binom{q-1}{2}$, and $S_2$ and $S_3$ are disjoint, but we cannot pick all the elements of $\sigma_2 \cup \sigma_3$ from $S_3$, since $|\sigma_2| > 1$.  Therefore, the number of possible faces $\sigma$ is equal to
\[
\textstyle{q\binom{q-1}{s-1} \left(\binom{\binom{q-s}{3}+\binom{q-1}{2}}{i-s+2} - \binom{\binom{q-1}{2}}{i-s+2}\right)}.
\]
This explains the last term in \eqref{eq:intermediate}.
\end{itemize}
\smallskip

Using the Chu-Vandermonde identity (see, e.g.~\cite[p.~54]{Merris}),
\begin{equation}\label{e:comb-sum-identity}
    \binom{a+b}{k} = \sum_{j=0}^k \binom{a}{j}\binom{b}{k-j},
\end{equation}
we can remove one of the terms from the summation in~\eqref{eq:intermediate} as follows:

\begin{align*}
    &\quad \sum_{s=2}^{q-3} q\binom{q-1}{s-1} \binom{\binom{q-1}{2}}{i-s+2}\\
    &= q\left(\sum_{s=1}^{q} \binom{q-1}{s-1} \binom{\binom{q-1}{2}}{i-s+2}
        - \sum_{s \in \{1, q, q-1, q-2\}} \binom{q-1}{s-1} \binom{\binom{q-1}{2}}{i-s+2}\right)\\
    &= q\left(\binom{(q-1)+\binom{q-1}{2}}{i+1}
         - \binom{\binom{q-1}{2}}{i+1}
        - \binom{q-1}{2}\binom{\binom{q-1}{2}}{i-q+4}
         - (q-1)\binom{\binom{q-1}{2}}{i-q+3}
        - \binom{\binom{q-1}{2}}{i-q+2}\right)\\
    &= q\left(\binom{\binom{q}{2}}{i+1}
        - \binom{\binom{q-1}{2}}{i+1}
        - \binom{q-1}{2}\binom{\binom{q-1}{2}}{i-q+4}
        - (q-1)\binom{\binom{q-1}{2}}{i-q+3}
        - \binom{\binom{q-1}{2}}{i-q+2}\right).
\end{align*}
Plugging this back into~\eqref{eq:intermediate}, some cancellation of terms gives~\cref{eq:sharpbound}.
\qedhere
\end{proof}

\cref{t:sharpbound3} gives us a sharp bound for the betti numbers of the third power of a square-free monomial ideal $I$ with $q$ generators. In~\cref{fig:sharpbound3} we compare the bounds from~\cref{t:sharpbound3} 
to the other bounds in~\eqref{e:inequalities}, namely, $f_i(\TT_q^3)$ and  $f_i(\LL_q^3)$.
Counting faces of the Taylor complex $\TT_q^3$, and restating~\cite[Corollary~6.2]{Lr} for ease of reference, we have:
\begin{align}
    f_i(\TT_q^3) &= \binom{\binom{q + 2}{3}}{i+1} \label{e:Tbound} \\
    f_i(\LL_q^3) &= q\binom{q - 1}{i} + \binom{\binom{q}{3} + 2 \binom{q}{2}}{i + 1} \label{e:Lbound}
\end{align}
{
\begin{figure}[h!]
    \centering
    \includegraphics[scale = 0.4]{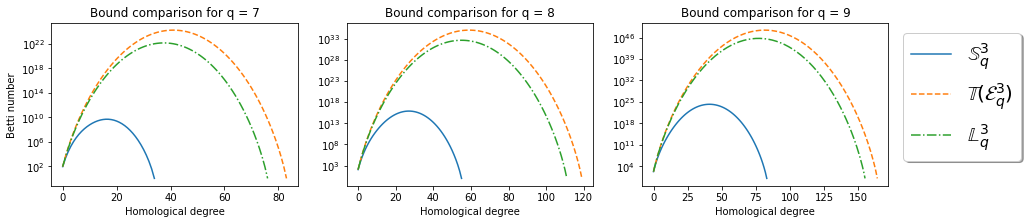}
    \caption{A comparison between known bounds for betti numbers}
    \label{fig:sharpbound3}
\end{figure}
}

We now give a more algebraic analysis of the comparison of these three bounds.  
We will show in \cref{t:c-bound} that the pattern demonstrated in 
\cref{fig:sharpbound3} continues to hold and becomes more pronounced as $q$ grows: the bound given by $f_i(\Sq)$ is significantly better than the bounds given by $f_i(\TT_q^3)$ and $f_i(\LL_q^3)$.  
A close examination of \cref{fig:sharpbound3} shows that the improvement in the bound becomes more significant as $i$ grows relative to the projective dimension of $\E^3_q$. In contrast, if we fix an $i$, for $q$ large relative to $i$ the improvement is not apparent, as seen in \cref{t:i-bound}.

\begin{theorem}\label{t:i-bound}
    Fix $i \geq 0$.  Then the dominant term of $f_i(\Sq)$ is $\binom{\binom{q}{3}}{i+1}$. In particular, $f_i(\TT_q^3), f_i(\LL_q^3)$, and  $f_i(\Sq)$ are all $O(q^{3i+3})$. 
\end{theorem}
\begin{proof}
First note that the degree in $q$ of the binomial coefficient 
    $$
        \binom{\binom{q - t}{b}}{c}
    $$
is $bc$.  We will use this to see that the dominant term of $f_i(\SS_q^3)$ is $\binom{\binom{q}{3}}{i+1}$, a polynomial of degree $3i+3$.  Each of the other binomial terms in the sum is a polynomial of lower degree.  The only term remaining is 
$\sum_{s=2}^{q-3} p_s(q)$, where $p_s(q)= q\binom{q-1}{s-1}\binom{\binom{q-s}{3}+\binom{q-1}{2}}{i-s+2}$.  Now $\binom{q-s}{3}+\binom{q-1}{2}$ is a polynomial of degree $3$, so $p_s(q)$ is a polynomial of degree $1 + (s-1) + 3(i-s+2) = 3i - 2s + 6$.  In particular, $p_2(q)$ has degree $3i+2$.  Each of the other terms in the summation has degree at most $3i - 1$, since $s \geq 3$ for those terms; adding up $q-5$ such polynomials increases the degree by at most 1, giving a total degree for the summation of the $s \geq 3$ terms of at most $3i$.  Therefore, the entire sum $\sum_{s=2}^{q-3} p_s(q)$ has degree $3i+2$.

Using the same tools, it is apparent that $f_i(\TT_q^3)$ and $f_i(\LL_q^3)$ are also polynomials in $q$ of degree $3(i + 1)$.
\end{proof}

The bound on betti numbers given by $f_i(\SS_q^3)$ is sharp, since it is achieved by $\beta_i({\E_q}^3)$. Thus \cref{t:i-bound} for $r=3$, and \cite[Theorem~5.9]{extremal} more generally, show that for fixed $i$, if $q$ is allowed to grow then Taylor's bound is asymptotically sharp for betti numbers of powers of ideals. 
Yet \cref{fig:sharpbound3} shows that, when we fix $q$ instead of $i$, our bound $f_i(\SS_q^3)$ is exponentially smaller than the other bounds $f_i(\TT_q^3)$ and $f_i(\LL_q^3)$, and that the ratios
$f_i(\TT_q^3)/f_i(\SS_q^3)$ and $f_i(\LL_q^3)/f_i(\SS_q^3)$
increase rapidly as $i$ grows. 
We will verify and quantify this observation in \cref{t:c-bound}.

Our point of view is to treat $f_i(\TT_q^3), f_i(\LL_q^3)$, and  $f_i(\SS_q^3)$ as functions of $i$ for a fixed large value of $q$.  In fact, each one is, asymptotically (for large $q$), a binomial function 
$\binom{u}{i}$ in $i$, for some $u$.  We compare these functions by first comparing their peak values at $i=u/2$, which is where the binomial function is maximum.
We then analyze the growth of the functions as $i$ ranges from $0$ to $\dim \SS_q^3+1$, the maximum value of $i$ for which $f_{i-1}(\SS_q^3) \neq 0$.
We focus, however, not on $i$ itself but on the proportion $c=i/(\dim \SS_q^3+1)$.  We do this by fixing $c$, and letting $q$ go to infinity.  The resulting functions of $c$ show the limit behavior of $f_i(\TT_q^3), f_i(\LL_q^3)$ and $f_i(\SS_q^3)$.
\cref{t:c-bound} shows that as $q$ grows, our bounds on the higher betti numbers given by the Scarf complex are more than exponentially better than those given by the Taylor complex.

\begin{theorem}\label{t:c-bound}
The following statements hold:
\begin{enumerate}
    \item\label{i:bounds-peak} Asymptotically in $q$, the ratios
        $$\frac{\max_i f_i(\TT_q^3)}{\max_i f_i(\SS_q^3)}
        \qand
        \frac{\max_i f_i(\LL_q^3)}{\max_i f_i(\SS_q^3)} \qare 
        \frac{2^{q^2+O(q)}}{\sqrt{1+6/q}}.$$
    \item\label{i:bounds-all} 
   Fix $0 \leq c < 1$. Define $i = c(\dim(\SS_q^3)+1)$.  Then, asymptotically in $q$, the ratios
    \[
        \frac{f_{i-1}(\TT_q^3)}{f_{i-1}(\SS_q^3)} \qand
        \frac{f_{i-1}(\LL_q^3)}{f_{i-1}(\SS_q^3)}
        \qare
        \left(\frac{1}{1-c}\right)^{q^2 + O(q)}O\left(\exp\left(\frac{-3cq}{1-c} \right)\right).  
    \]
\end{enumerate}
\end{theorem}
Since $1/(1-x) \geq e^x$ for $x$ in the range $0 \leq x < 1$, this ratio grows faster than exponential in $i$.

\begin{proof}\mbox{}
By~\cref{t:i-bound}, $f_{i-1}(\SS_q^3), f_{i-1}(\TT_q^3), f_{i-1}(\LL_q^3)$ are all polynomials in $q$ of degree $3i$, so for large $q$, we may consider only the dominant term of each: 
\[
    S(q,i) = \binom{s(q)}{i},\
    T(q,i) = \binom{t(q)}{i},\
    L(q,i) =\binom{\ell(q)}{i},
\]
respectively, where
\begin{align*}
    s(q) = \binom{q}{3} &= \frac{1}{6}q^3 - \frac{1}{2}q^2 + O(q) \\ 
    t(q) = \binom{q+2}{3} &= \frac{1}{6}q^3 + \frac{1}{2}q^2 + O(q)\\
    \ell(q) = \binom{q}{3}+2\binom{q}{2} &= \frac{1}{6}q^3 + \frac{1}{2}q^2 + O(q).
\end{align*}
  Since $t(q)$ and $\ell(q)$ agree up to $O(q)$ terms, we will only analyze $t(q)$.  Calculations with $\ell(q)$ are similar.

It will be helpful to note these easy calculations: 
\begin{equation}\label{e:fqgq}
    t(q) - s(q) = q^2 + O(q),\qand
    \frac{t(q)}{s(q)} = 1 + \frac{6}{q} + O(q^{-2}).
\end{equation}
Also recall Stirling's approximation,
\[
    n! = \sqrt{2\pi n}(n/e)^n(1+O(n^{-1})),
\]
and the product limit definition of $e^x$,
\[
    (1 + x/n)^n = e^{x - (x^2/n) + O(x^3/n^2)}.
\]

Proof of~\eqref{i:bounds-peak}:
The maximum value of any binomial function $\binom{n}{i}$ is achieved at $i = n/2$.  
\begin{align*}
    \frac{\binom{t(q)}{t(q)/2}}{\binom{s(q)}{s(q)/2}}
        &= \frac{t(q)!}{s(q)!}\left(\frac{(s(q)/2)!}{(t(q)/2)!} \right)^2\\
        &= \frac{\sqrt{2\pi t(q)}(t(q)/e)^{t(q)}(1+O(q^{-3}))}{\sqrt{2\pi s(q)}(s(q)/e)^{s(q)}(1+O(q^{-3}))}
        \left( \frac{\sqrt{\pi s(q)}(s(q)/2e)^{s(q)/2}(1+O(q^{-3}))}{\sqrt{\pi t(q)}(t(q)/2e)^{t(q)/2}(1+O(q^{-3}))} \right)^2\\
        &= \sqrt{\frac{s(q)}{t(q)}}2^{t(q)-s(q)}(1+O(q^{-3}))
        = \frac{2^{q^2+O(q)}(1+O(q^{-3}))}{\sqrt{1 + 6/q + O(q^{-2})}}\\
        &=2^{q^2 + O(q)}/\sqrt{1 + 6/q},
\end{align*}
where  the second equality is by Stirling's formula, and the fourth equality is by~\cref{e:fqgq}.

Proof of~\eqref{i:bounds-all}:
First note that $i = c(\dim (\Sq) + 1) = c\binom{q}{3} =cs(q)$.
Then $O(s(q)-i) = O(s(q)(1-c)) = O(s(q)) = O(q^3)$, since $c$ is constant.  Similarly, $O(t(q)-i) = O(q^2 + s(q) - i) = O(q^2) + O(s(q) - i) = O(q^3)$.  Therefore,
\begin{align*}
    \frac{T(q,i)}{S(q,i)} &= \frac{\binom{t(q)}{i}}{\binom{s(q)}{i}}
            = \frac{t(q)!}{s(q)!} \frac{(s(q)-i)!}{(t(q)-i)!} \\
        &= \frac{\sqrt{2 \pi t(q)}(t(q)/e)^{t(q)}(1+O(q^{-3}))}{\sqrt{2 \pi s(q)}(s(q)/e)^{s(q)}(1+O(q^{-3}))}
            \frac{\sqrt{2 \pi (s(q)-i)}((s(q)-i)/e)^{s(q)-i}(1+O(q^{-3}))}{\sqrt{2 \pi (t(q)-i)}((t(q)-i)/e)^{t(q)-i}(1+O(q^{-3}))}\\
        &= \frac{t(q)^{t(q)}}{s(q)^{s(q)}}
            \frac{(s(q)-i)^{s(q)-i}}{(t(q)-i)^{t(q)-i}}
            \sqrt{\frac{t(q)(s(q)-i)}{s(q)(t(q)-i)}}
            (1+O(q^{-3}))\\
        &=\left(\frac{t(q)}{s(q)} \right)^{t(q)} \frac{1}{s(q)^{s(q)-t(q)}}
            \left(\frac{s(q)-i}{t(q)-i} \right)^{t(q)-i}
            (s(q)-i)^{s(q)-t(q)}
            \sqrt{\frac{O(q^6)}{O(q^6)}}(1+O(q^{-3}))\\
        &= \left(\frac{s(q)-i}{s(q)}\right)^{s(q)-t(q)}
            \left(\frac{s(q)}{t(q)} \right)^{-t(q)}
            \left(\frac{s(q)-i}{t(q)-i} \right)^{t(q)-i}
            O(1+q^{-3}),
\end{align*}
where the third equality is by Stirling's formula. We now deal with the first three factors separately.

The first factor simplifies as follows:
\[
    \left(\frac{s(q)-i}{s(q)}\right)^{s(q)-t(q)}
    = (1-c)^{-q^2 + O(q)}
\]
by \cref{e:fqgq}.

The second factor expands as follows, by the product limit definition of $e^x$:
\begin{align*}
    \left(\frac{s(q)}{t(q)} \right)^{-t(q)} 
    &= \left(1 - \frac{t(q)-s(q)}{t(q)} \right)^{-t(q)}
    = \left(1 + \frac{t(q)-s(q)}{-t(q)} \right)^{-t(q)}\\
    &= \exp \left((t(q)-s(q)) - \frac{1}{2}\frac{(t(q)-s(q))^2}{-t(q)} 
        +O \left(\frac{(t(q)-s(q))^3}{(-t(q))^2} \right)\right).
\end{align*}

Similarly, the third factor expands as follows, by the product limit definition of $e^x$:
\begin{align*}
            \left(\frac{s(q)-i}{t(q)-i} \right)^{t(q)-i}
    &= \left(1 - \frac{t(q)-s(q)}{t(q)-i} \right)^{t(q)-i}
    = \left(1 + \frac{-(t(q)-s(q))}{t(q)-i} \right)^{t(q)-i}\\
    &= \exp \left(-(t(q)-s(q)) - \frac{1}{2}\frac{(t(q)-s(q))^2}{t(q)-i} 
        +O \left(\frac{(t(q)-s(q))^3}{(t(q)-i)^2} \right)\right).
\end{align*}

When we multiply the second and third factors, we get
\[
    \exp\left(\frac{1}{2}(t(q)-s(q))^2\left(\frac{1}{t(q)} - \frac{1}{t(q)-i} \right) + O\left(\frac{(q^2)^3}{(q^3)^2}\right) \right)
\]
\begin{align*}
    &=     \exp\left(\frac{1}{2}(t(q)-s(q))^2\left(\frac{-i}{t(q)(t(q)-i)} \right) + O(1) \right)\\
    &=     O\left(\exp\left(\frac{1}{2}(t(q)-s(q))^2\left(\frac{-i/s(q)}{t(q)(t(q)/s(q)-i/s(q))} \right)\right) \right)\\
    &=     O\left(\exp\left(\frac{1}{2}(q^2+O(q))^2\left(\frac{-c}{(q^3/6+O(q^2))(1+O(q^{-1})-c)} \right) \right) \right)\\
    &=     O\left(\exp\left(\frac{1}{2}(q^2+O(q))^2\left(\frac{-6c}{((1-c)q^3+O(q^2))} \right) \right) \right)\\
    &= O\left(\exp\left(\frac{-3c}{1-c} \frac{q^4 + O(q^3)}{q^3 + O(q^2)} \right)\right) = O\left(\exp \left(\frac{-3cq}{1-c}\right)\right),
\end{align*}
using~\cref{e:fqgq} in the third equation.
The result now follows from multiplying together the four simplified factors.
\end{proof}

We close the paper with an interesting observation about the betti numbers of $\Erq$.

\begin{remark}[{\bf Log-concavity of betti numbers when $q\leq 3$}]\label{log-concave}
    When $q \leq 3$, it is shown in~\cref{e:q34} that $\UU_q^r$ is completely determined by facets of type $\U_\ba^r$, where $|\ba| = r - 1$. In this case, the complex $\UUrq$ has appeared before in the context of subdivisions of simplicial complexes. More specifically, for $q \leq 3$ the complex $\UUrq$ is the {\it $r$-fold edgewise subdivision} of a $(q-1)$--simplex. Properties of $f$ and $h$-vectors of edgewise subdivisions of complexes were studied in depth by Athanasiadis and Jochemko in~\cite{Athanasiadis,Jochemko}. 
    
     In~\cite{Jochemko} it is shown that, when $r \geq q$, the coefficients $h_i$ of the $h$-polynomial of the $r$-fold edgewise subdivision of a $(q-1)$--simplex are a {\bf log-concave} sequence.  That is
     \[
     h_i^2 \geq h_{i-1}h_{i+1}.
     \]
     This implies the coefficients are also {\bf unimodal}, that is, for some $k$
     \[
        h_1 \leq \cdots \leq h_{k-1} \leq h_k \geq h_{k+1} \geq \cdots h_d.
     \]
     For $q \leq 3$ and $r \geq 1$,~\cite[Lemma~4.2]{Lenz} together with the main results of~\cite{extremal,Jochemko} and some simple computation for boundary cases imply the betti numbers of $\Erq$ are log-concave, and hence unimodal.
   \end{remark}

    When $r = 3$, using our previous results which characterize the faces of $\SS_q^3$, we can check using Macaulay2~\cite{M2} that the $f$-vector of $\SS_q^3$ is log-concave for $q \leq 8$.

\begin{question}  \label{q:log-concavity}

The following are natural questions that arise from \cref{log-concave}.
\begin{enumerate}
    \item\label{i:unimodality}  ({\bf Unimodality}) Is the betti vector of $\Erq$ always unimodal?  In other words, for a given $r, q$, is there a $k$ such that the following inequalities hold?
    \[
        \beta_1(\Erq) \leq \dots \leq \beta_{k-1}(\Erq) \leq \beta_k(\Erq) \geq \beta_{k+1}(\Erq) \geq\dots \geq \beta_{\pd(\Erq)}(\Erq) 
    \]
    \item\label{i:log-concavity} ({\bf Log-concavity}) Is the betti vector of $\Erq$ always log-concave?  In other words, does the inequality 
    \[
        \beta_i(\Erq)^2 \geq \beta_{i - 1}(\Erq) \beta_{i + 1}(\Erq)
    \]
hold for all $q$, $r$ and $i$?
\end{enumerate}
As log-concavity implies unimodality, Question~\eqref{i:log-concavity} is stronger than Question~\eqref{i:unimodality}.  
A possible strategy to positively answer either question might be: Prove \cref{c:conjecture} to reduce the problem to one about face numbers of the Scarf complex $\Srq$; generalize \cref{t:sharpbound3,t:i-bound} to $r>3$ in order to establish a formula for the face numbers of $\Srq$, and show that the dominant term in that formula is a binomial coefficient; and then rely on the log-concavity of binomial coefficients. In particular, a positive answer to~\cref{q:dominantterm} could lead to a positive answer to~\cref{q:log-concavity}. 

\end{question}

\section*{Acknowledgements}

This project started in the Interactions Between Topological Combinatorics and Combinatorial Commutative Algebra workshop hosted by the Banff International Research Station in Spring 2023, and was completed  at the Simons Laufer Mathematical Sciences Institute in Berkeley during a two week visit in Summer 2024.  We would like to thank  BIRS and SLMath for providing us with wonderful facilities leading to productive stays. 

This material is based upon work supported by the National Science Foundation under Grant No. DMS-1928930, while the authors were in residence at the Mathematical Sciences Research Institute in Berkeley, California, during Summer 2024.   
Our research was assisted by computations carried out using the computer algebra software Macaulay2~\cite{M2} and Sage~\cite{sage}. 




\begin{thebibliography}{1234}

\bibitem{matroid-polytopes} 
Ardila, F., Benedetti, C., Doker, J., 
{\it Matroid polytopes and their volumes}, 
Disc.\ Comp.\ Geom. 43, no.~4 (2010) 841--854.

\bibitem{Athanasiadis} 
Athanasiadis, C.~A.,  
{\it Face numbers of uniform triangulations of simplicial complexes}, 
Int. Math. Res. Not. IMRN (2022), no.~20, 15756--15787.

\bibitem{BW} 
Batzies, E., Welker, V., 
{\it Discrete Morse theory for cellular resolutions}, 
J. Reine Angew. Math., 543 (2002) 147--168.

\bibitem{permutation-polytopes} 
Baumeister, B., Haase, C., Nill, B., Paffenholz, A., 
{\it On permutation polytopes}, 
Adv.\ Math.\ 222 (2009) 431--452.

 \bibitem{BS} 
 Bayer, D., Sturmfels, B., 
 {\it Cellular resolutions of monomial modules}, 
 J. Reine Angew. Math. 503 (1998) 123--140.

 \bibitem{BPS} Bayer, D., Peeva, I., Sturmfels, B., 
 {\it Monomial resolutions}, 
 Math. Res. Lett. 5, no.~1-2 (1998) 31--46.

 \bibitem{Char} 
 Chari, M., 
 {\it On discrete Morse functions and combinatorial decompositions}, 
 Disc.\ Math. 217 (2000), no. 1-3, 101 - 113.

\bibitem{Halifax}  
Chau, T., Duval, A., Faridi, S., Holleben, T., Morey, S., \c{S}ega, L.~M., 
{\it  Algebraic properties of powers of extremal ideals}, 
in preparation.

\bibitem{CK24} Chau, T., Kara, S., 
{\it Barile-Macchia resolutions}, 
J. Alg. Combin., v. 59, no. 2, 2024, 413--472.

\bibitem{CHM24} Chau, T., H\`{a}, H. T., Maithani, A., 
{\it Minimal cellular resolutions of powers of graphs}, 
\hyperlink{https://arxiv.org/abs/2404.04380}{arXiv:2404.04380 [math.AC]}.

\bibitem{L2}
Cooper, S., El Khoury, S., Faridi, S., Mayes-Tang, S., Morey, S., \c{S}ega, L.~M., Spiroff, S., 
{\it Simplicial resolutions for the second power of square-free monomial ideals},
Women in Commutative Algebra, 193--205, Assoc.~Women Math.~Ser., 29, Springer, Cham, 2021.

\bibitem{Morse} 
Cooper, S., El Khoury, S., Faridi, S., Mayes-Tang, S., Morey, S., \c{S}ega, L.~M., Spiroff, S., 
{\it Morse resolutions of powers of square-free monomial ideals of projective dimension one},
J. Alg.\ Combin., (2022), no. 4, 1085 - 1122.

 \bibitem{Lr}
 Cooper, S., El Khoury, S., Faridi, S., Mayes-Tang, S., Morey, S., \c{S}ega, L.~M., Spiroff, S., 
 {\it Simplicial resolutions of powers of square-free monomial ideals}, 
 Alg.\ Combin., v.~7, no.~1,  (2024), 77--107.
 
\bibitem{Hasan} 
Cooper, S.,  Faridi, S., Mahmood, H., 
{\it Simplicial Resolutions of Powers of Monomial Ideals},
preprint (2025).

\bibitem{Laplace} 
de Laplace, M., 
{\it Oeuvres compl\`etes, Vol. 7}, 
r\'{e}\'{e}dit\'{e} par Gauthier-Villars, Paris, 1886.
 
\bibitem{DST} 
De~Loera, J.~A., Sturmfels, B., Thomas, R.~R., 
{\it Gr\"obner bases and triangulations of the second hypersimplex}, 
Combinatorica, v.~14, no.~3, 409--424.

\bibitem{matching-polytopes} 
Eisley, B., Matshshita, K., Vindas-Mel\'endez, A., 
{\it Matching polytopes, Gorensteinness, and the integer decomposition property}, 
\hyperlink{https://arxiv.org/abs/2407.08820}{arXiv:2407.08820 [math.CO]}.

 \bibitem{extremal}
 El~Khoury, S., Faridi, S., \c{S}ega,
 L.~M., Spiroff, S.,
 {\it The Scarf complex and betti numbers of powers of extremal ideals}, J. Pure Appl.\ Alg.\ 228 (2024) 107577.
  
 \bibitem{FHHM24}
 Faridi, S., H\`{a}, H.T., Hibi, T., Morey, S., 
 {\it Scarf complexes of graphs and their powers}, 
 \hyperlink{https://arxiv.org/abs/2403.05439}{arXiv:2403.05439 [math.AC]}.

 \bibitem{For}  
Forman, R., 
{\it Morse theory for cell complexes},
Adv.\ Math.\ 134 (1998), no. 1, 90 - 145.
  
\bibitem{GGL}
Gabri\'elov, A.M., Gel'fand, I.M., Losik, M.V., 
{\it Combinatorial computation of characteristic classes},
Funct.\ Anal.\ Its Appl.\ 9, 103--115 (1975).

 \bibitem{M2} 
 Grayson, D.R., Stillman, M.E., 
 {\it Macaulay2, a software system for research in algebraic geometry}, 
 Available at \hyperlink{http://www.math.uiuc.edu/Macaulay2/}{http://www.math.uiuc.edu/Macaulay2/}.

\bibitem{HHZ04} 
Herzog, J., Hibi, T., Zheng, X., 
{\it Dirac's theorem on chordal graphs and Alexander duality}, 
Eur.\ J. Combin.\ 25 (2004), no. 7, 949--960. 

\bibitem{Jochemko} 
Jochemko, K., 
{\it On the real-rootedness of the Veronese construction for rational formal power series}, 
Int.\ Math.\ Res.\ Not.\ IMRN, no.~15, 4780--4798 (2018).
  
\bibitem{Jo} Jonsson, J., 
{\it Simplicial Complexes of Graphs}, 
Lecture Notes in Mathematics, {\bf 1928}, Springer, Berlin (2008), no.~12 (2017) 5453--5464.
  
\bibitem{Lenz}
Lenz, M., 
{\it The $f$-vector of a representable-matroid complex is log-concave}, 
Adv.\  Appl.\ Math.\ 51 (2013), no.~5, 543--545.
 
\bibitem{L}
Lyubeznik, G., 
{\it A new explicit finite free resolution of ideals generate by monomials in an $R$-sequence},
J. Pure Appl.\ Alg.\ 51 (1998) 193--195.

\bibitem{Mer}
Mermin, J., 
{\it Three simplicial resolutions},  
Progress in commutative algebra 1, 127--141, de Gruyter, Berlin, (2012).

\bibitem{Merris} Merris, R., 
{\it Combinatorics, 2nd ed.}, 
Wiley-Interscience, Hoboken, NJ, (2003).

 \bibitem{P} 
 Peeva, I., 
 {\it Graded syzygies}, 
 Algebra and Applications, 14. Springer-Verlag London, Ltd., London, (2011). 
 
 \bibitem{Postnikov}
Postnikov, A., 
{\it Permutohedra, associahedra, and beyond},
Int.\ Math.\ Res.\ Not.\ IMRN, (2009). 

\bibitem{Stanley} 
Stanley, R. P., 
{\it Eulerian partitions of a unit hypercube}, 
Higher Combinatorics, p.\ 49, Reidel, Dordrecht/Boston, 1977.

\bibitem{sage} 
Stein, W.~A., et~al., 
{\it {S}age {M}athematics {S}oftware ({V}ersion 10.2)}, 
The Sage Development Team, 2023, \hyperlink{http:www.sagemath.org}{http://www.sagemath.org}.

\bibitem{T} 
Taylor, D., 
{\it Ideals generated by monomials in an $R$-sequence}, 
Ph.D. Thesis, University of Chicago (1966).

\bibitem{Ziegler-lectures} 
Ziegler, G., 
{\it Lectures on 0/1-Polytopes}, 
Polytopes — Combinatorics and Computation, DMV Seminar, vol 29. Birkh\"auser, Basel (2000).

\bibitem{Ziegler} 
Ziegler, G., 
{\it Convex polytopes: Extremal constructions and $f$-vector shapes}, 
Geometric Combinatorics, 617--691, IAS/Park City Math.\ Ser., 13, Amer.\ Math.\ Soc.\, Providence, RI, 2007.
 
  \end{thebibliography}
\end{document}